\documentclass[11pt]{amsart}
\usepackage{graphicx}
\usepackage{amssymb,float}
\usepackage{bm}
\usepackage[shortlabels]{enumitem}
\usepackage{amsmath, amscd,commath}
\usepackage{mathrsfs,todonotes}
\usepackage{tikz,tikz-cd,soul,eucal}
\usepackage{hyperref}
\usepackage[nameinlink,capitalise,noabbrev]{cleveref}
\usepackage[margin=1in,marginparwidth=0.8in, marginparsep=0.1in]{geometry}

\usetikzlibrary{calc,arrows}

\numberwithin{equation}{section}

\newtheorem{thm}{Theorem}[section]

\newtheorem{cor}[thm]{Corollary}
\newtheorem{lma}[thm]{Lemma}

\theoremstyle{definition}
\newtheorem{dfn}[thm]{Definition}
\newtheorem{ex}[thm]{Example}

\theoremstyle{remark}
\newtheorem{rmk}[thm]{Remark}

\numberwithin{equation}{section}

\newcommand{\R}{{\mathbb{R}}}
\newcommand{\Z}{{\mathbb{Z}}}

\DeclareMathOperator{\CZ}{CZ}
\DeclareMathOperator{\crit}{crit}
\DeclareMathOperator{\Skel}{Skel}
\DeclareMathOperator{\ind}{ind}
\DeclareMathOperator*{\colim}{colim}
\DeclareMathOperator*{\hocolim}{hocolim}
\DeclareMathOperator{\id}{id}

\newcommand{\suchthat}{\;\ifnum\currentgrouptype=16 \middle\fi|\;}

\begin{document}

\title{Simplicial descent for Chekanov--Eliashberg dg-algebras}
\author{Johan Asplund}
\address{Department of Mathematics, Columbia University, 2990 Broadway, New York, NY 10027}
\email{asplund@math.columbia.edu}

\begin{abstract}
	We introduce a type of surgery decomposition of Weinstein manifolds we call \emph{simplicial decompositions}. The main result of this paper is that the Chekanov--Eliashberg dg-algebra of the attaching spheres of a Weinstein manifold satisfies a descent (cosheaf) property with respect to a simplicial decomposition. Simplicial decompositions generalize the notion of Weinstein connected sum and we show that there is a one-to-one correspondence (up to Weinstein homotopy) between simplicial decompositions and so-called good sectorial covers. As an application we explicitly compute the Chekanov--Eliashberg dg-algebra of the Legendrian attaching spheres of a plumbing of copies of cotangent bundles of spheres of dimension at least three according to any plumbing quiver. We show by explicit computation that this Chekanov--Eliashberg dg-algebra is quasi-isomorphic to the Ginzburg dg-algebra of the plumbing quiver.
\end{abstract}

\maketitle
\tableofcontents
\section{Introduction}
	\subsubsection{Main results}
	Let $X$ be a $2n$-dimensional Weinstein manifold with ideal contact boundary $\partial X$, and let $(V,\lambda)$ be a $(2n-2)$-dimensional Weinstein manifold. A Weinstein hypersurface is a Weinstein embedding $(V, \lambda_V := \eval[0]{\lambda}_{V}) \hookrightarrow (X \setminus \Skel X, \lambda)$ such that the induced map $V \longrightarrow \partial X$ is an embedding, see \cite{avdek2012liouville,eliashberg2018weinstein,sylvan2019on,ganatra2020covariantly,alvarez2020positive}. Throughout this paper we use the notation $V \hookrightarrow \partial X$ to denote a Weinstein hypersurface. 

	The main construction in this paper is that of a \emph{simplicial decomposition} of a Weinstein manifold $X$, which is a tuple $(C, \boldsymbol V)$ consisting of a simplicial complex $C$ and a set $\boldsymbol V$. The set $\boldsymbol V$ contains Weinstein manifolds and Weinstein hypersurfaces of various dimensions and is in bijection with the set of faces and face inclusions of $C$. Roughly speaking, using each Weinstein manifold $V \in \boldsymbol V$ we construct a \emph{simplicial handle}, such that when glued together according to $C$ gives $X$ up to Weinstein isomorphism, see \cref{sec:intro_simplex_decompositions} for a more precise description. The definition of a simplicial decomposition generalizes the notion of Weinstein connected sum \cite{avdek2012liouville,eliashberg2018weinstein,alvarez2020positive} and if we allow for Weinstein cobordisms with non-empty negative ends, simplicial decompositions may be used to give a surgery description of stopped Weinstein manifolds used in the study of partially wrapped Fukaya categories \cite{ekholm2017duality,sylvan2019on,ganatra2020covariantly,asplund2019fiber,asplund2021chekanov}.

	In this paper we extensively use Weinstein handle decompositions, by which we mean a choice of presentation of a Weinstein manifold $X$ as a subcritical Weinstein manifold $X_0$ together with some number of Legendrian spheres in its boundary. We will denote such a choice of handle decomposition by $h$. A choice of Weinstein handle decomposition of each $V\in \boldsymbol V$, denoted by $\boldsymbol h := \left\{h_V\right\}_{V\in \boldsymbol V}$, determines a Weinstein handle decomposition of $X$ in the surgery presentation given by the simplicial decomposition. The main result of this paper is the following.
	\begin{thm}[\cref{thm:ce_descent}]\label{thm:ce_descent_intro}
		Let $(C, \boldsymbol V)$ be a simplicial decomposition of a Weinstein manifold $X$ and let $\varSigma(\boldsymbol h)$ denote the union of the Legendrian attaching spheres for the Weinstein top handles of $X$ in the surgery presentation given by the simplicial decomposition.

		Then there is a quasi-isomorphism of dg-algebras 
		\[
			CE^\ast(\varSigma(\boldsymbol h);X_0) \cong \colim_{\sigma_k\in C} CE^\ast(\varSigma_{\supset \sigma_k}(\boldsymbol h); X(\sigma_k)_0)\, ,
		\]
		where $X(\sigma_k)$ is a Weinstein cobordism obtained from $X$ by replacing the Weinstein manifolds $V^{2n-2i}_{\sigma_i} \in \boldsymbol V$ with a symplectization of a contactization for every $\sigma_i \not \supset \sigma_k$, see \cref{sec:simplical_descent} for details. The Legendrian submanifold $\varSigma_{\supset \sigma_k}(\boldsymbol h)$ is the union of top attaching Legendrian spheres of $X(\sigma_k)$.
	\end{thm}
	The colimit above should be interpreted as follows: For each $k$-face $\sigma_k \in C$ there are dg-algebras $CE^\ast(\varSigma_{\supset \sigma_k}(\boldsymbol h); X(\sigma_k)_0)$ and $\mathcal A_{\sigma_k}(\boldsymbol h)$ which satisfies the following.
	\begin{enumerate}
		\item There is a quasi-isomorphism $\mathcal A_{\sigma_k}(\boldsymbol h) \cong CE^\ast(\varSigma_{\supset \sigma_k}(\boldsymbol h); X(\sigma_k)_0)$.
		\item For each inclusion of faces $\sigma_k \subset \sigma_{k+1}$ there is a reverse inclusion of dg-algebras
		\[
			\mathcal A_{\sigma_{k+1}}(\boldsymbol h) \subset \mathcal A_{\sigma_k}(\boldsymbol h)\, .
		\]
	\end{enumerate}
	Because all involved dg-algebras are semi-free, the colimit is again a semi-free dg-algebra which is generated by the union of the generators, see \cref{sec:intro_leg_attach_and_sectorial_covers} and \cref{sec:simplical_descent} for more details.
	\begin{rmk}
		When imposing an action bound on the generators and allowing for sufficiently thin handles in the surgery construction of $X$ and $X(\sigma_k)$ outlined in \cref{sec:intro_simplex_decompositions} and \cref{sec:intro_leg_attach_and_sectorial_covers} below, the quasi-isomorphism in \cref{thm:ce_descent_intro} is in fact an isomorphism of dg-algebras on the chain level. This is in line with similar surgery formulas in \cite{bourgeois2012effect,ekholm2017duality,asplund2021chekanov}.
	\end{rmk}
	\begin{rmk}[About coefficient rings]
		Let $\mathbb F$ be a field. The Chekanov--Eliashberg dg-algebra of a spin Legendrian submanifold $\varLambda$ is defined as a dg-algebra over the semi-simple idempotent ring $\boldsymbol k := \bigoplus_{i\in \pi_0(\varLambda)} \mathbb F e_i$ where $\left\{e_i\right\}_{i\in \pi_0(\varLambda)}$ is a set of mutually orthogonal idempotents \cite[Section 4.1]{bourgeois2012effect} \cite[Section 3.2]{ekholm2017duality}. The Legendrian submanifolds $\varSigma_{\supset \sigma_k}(\boldsymbol h) \subset \partial X(\sigma_k)_0$ whose Chekanov--Eliashberg dg-algebras appear in \cref{thm:ce_descent_intro} have different number of components in general, which means that they are dg-algebras over different rings. The colimit in \cref{thm:ce_descent_intro} is thus taken in the category of associative, non-commutative and non-unital dg-algebras over varying non-unital rings, see \cref{sec:simplical_descent} for details. All the dg-algebras appearing in \cref{thm:ce_descent_intro} are unital, but the maps between them are not.
	\end{rmk}
	\subsubsection{Good sectorial covers}
	Let $X = X_1 \cup \cdots \cup X_{m}$ be a sectorial cover of a Weinstein manifold in the sense of Ganatra--Pardon--Shende \cite{ganatra2022sectorial}. We call the sectorial cover \emph{good} if for any subset $\varnothing \neq A \subset \left\{1,\ldots,m\right\}$ we have a Weinstein isomorphism
	\begin{equation}\label{eq:good_sectorial_cover_intro}
		N \left(\bigcap_{i \in A} X_i\right) \cong X^{2n-2k}_A \times T^\ast \R^k\, ,
	\end{equation}
	where $k:= \abs A - 1$, where $X^{2n-2k}_A$ is a $(2n-2k)$-dimensional Weinstein manifold, and where $N$ denotes a neighborhood that is cylindrical with respect to the Liouville vector field on $X$. Furthermore we require \eqref{eq:good_sectorial_cover_intro} to extend the existing Weinstein isomorphisms $N \left(\bigcap_{i \in A} \partial X_i\right) \cong X^{2n-2k}_A \times T^\ast \R^k$ coming from the assumption that we have a sectorial cover, see \cref{dfn:good_sectorial_cover} for details. We prove the following result.

	\begin{thm}[\cref{thm:one-to-one_corr_covers_and_decomps}]\label{thm:one-to-one_corr_covers_and_decomps_intro}
		Let $X$ be a Weinstein manifold. There is a one-to-one correspondence (up to Weinstein homotopy) between good sectorial covers of $X$ and simplicial decompositions of $X$.
	\end{thm}

	\begin{rmk}
		In this correspondence the simplicial complex $C$ is the \v{C}ech nerve of the corresponding good sectorial cover, and therefore \cref{thm:one-to-one_corr_covers_and_decomps_intro} is akin to a version of the nerve theorem for Weinstein manifolds.
	\end{rmk}
	\begin{rmk}
		Like good covers of topological spaces (covers for which any non-empty finite intersection is contractible), good sectorial covers exist in abundance. Any sectorial cover has a refinement that is a good sectorial cover, see \cref{lma:good_sectorial_cover_refinement}.
	\end{rmk}
	\begin{rmk}
		Combining the two main theorems \cref{thm:ce_descent_intro} and \cref{thm:one-to-one_corr_covers_and_decomps_intro} with the Legendrian surgery formula of Bourgeois--Ekholm--Eliashberg \cite{bourgeois2012effect} we recover the following colimit in wrapped Floer cohomology:
		\[
			HW^\ast(C; X) \cong \colim_{\sigma_k\in C} HW^\ast(C_{\sigma_k}; X(\sigma_k)) \, ,
		\]
		where $C_{\sigma_k}$ is the union of cocore disks of the top handles of $X(\sigma_k)$, see \cref{cor:sectorial_descent}. This colimit also follows from the more general sectorial descent result of Ganatra--Pardon--Shende \cite{ganatra2022sectorial} on the level of wrapped Fukaya categories which is the motivating result for the present paper. We note however that the results in \cite{ganatra2022sectorial} \emph{does not} imply our main result \cref{thm:ce_descent_intro}.
	\end{rmk}
	\begin{rmk}
			The idea of decomposing the Chekanov--Eliashberg dg-algebra into smaller pieces and obtaining a Seifert--van Kampen-type result is not new. Such results were proved for one-dimensional Legendrian knots in $\R^3$ by Sivek \cite{sivek2011bordered} and for two-dimensional Legendrian surfaces in a $1$-jet space by Harper and Sullivan \cite{harper2014bordered}. In a series of papers, Rutherford and Sullivan \cite{rutherford2020cellularI,rutherford2019cellularII,rutherford2019cellularIII} define a version of the Chekanov--Eliashberg dg-algebra for two-dimensional Legendrian surfaces in a $1$-jet space which admit a cellular decomposition such that the Chekanov--Eliashberg dg-algebra is obtained as a colimit of dg-subalgebras associated to each cell \cite[Remark 8.1]{rutherford2019cellularII}, see \cref{rmk:unknot_pieces} for further discussion.
		\end{rmk}

	In \cref{sec:intro_simplex_decompositions} we describe the construction of a simplicial decomposition in the two cases of $C$ being a $1$-simplex and a $2$-simplex. In \cref{sec:intro_leg_attach_and_sectorial_covers} we give a geometric description of the dg-subalgebras $CE^\ast(\varSigma_{\supset \sigma_k}(\boldsymbol h); X(\sigma_k)_0)$ appearing in \cref{thm:ce_descent_intro}.

	\subsection{Simplicial decompositions}\label{sec:intro_simplex_decompositions}
		Let $X$ be a Weinstein manifold. Throughout the paper we let $\varDelta^m$ denote an $m$-simplex centered at the origin in $\R^m$. Now we describe how to construct a simplicial decomposition $(\varDelta^m, \boldsymbol V)$ of $X$ in the two cases $m \in \left\{1,2\right\}$.
		\begin{description}
			\item[$m=1$]
				Let
				\[
					\boldsymbol V := \left\{X_1^{2n}, X_2^{2n}, V^{2n-2}, \iota_1 \colon V \hookrightarrow \partial X_1, \iota_2 \colon V \hookrightarrow \partial X_2\right\}
				\]
				where
				\begin{itemize}
					\item $X^{2n}_1$ and $X^{2n}_2$ are Weinstein $2n$-manifolds corresponding to vertices of $\varDelta^1$.
					\item $(V^{2n-2}, \lambda_V)$ is a Weinstein $(2n-2)$-manifold corresponding to the edge of $\varDelta^1$.
					\item $\iota_1, \iota_2$ are Weinstein hypersurfaces corresponding to the inclusion of the vertices into the edge.
				\end{itemize}

				Consider the Weinstein cobordism $(V \times D_\varepsilon T^\ast \varDelta^1, \lambda_V + 2xdy+ydx)$ where $D_\varepsilon T^\ast \varDelta^1$ is the $\varepsilon$-disk cotangent bundle of $\varDelta^1$, $y$ is a coordinate in the $\varDelta^1$-factor and $x$ is a coordinate in the fiber direction. We refer to this Weinstein cobordism as a simplicial handle. The negative end of the simplicial handle is $V \times D_\varepsilon T^\ast(\eval[0]{\varDelta^1}_{\partial \varDelta^1}) = (V \times (-\varepsilon,\varepsilon)) \sqcup (V \times (-\varepsilon,\varepsilon))$. Using the Weinstein hypersurfaces $\iota_1,\iota_2 \in \boldsymbol V$ we attach the simplicial handle $V \times D_\varepsilon T^\ast \varDelta^1$ to $X_1\sqcup X_2$ along $(V \times (-\varepsilon,\varepsilon))\sqcup (V \times (-\varepsilon,\varepsilon))$ and denote the resulting Weinstein manifold by $X_1 \#_V X_2$, see \cref{fig:1-simplex_handle}. If $X \cong X_1 \#_V X_2$, then we say that $(\varDelta^1, \boldsymbol V)$ is a simplicial decomposition of $X$.

				\begin{figure}[!htb]
					\centering
					\includegraphics{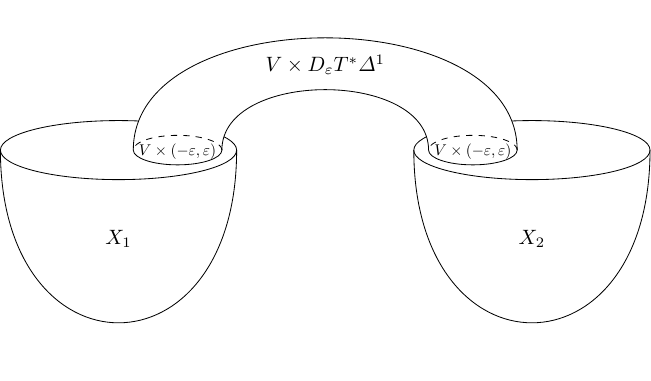}
					\caption{Simplicial decomposition of the Weinstein manifold $X \cong X_1 \#_V X_2$.}
					\label{fig:1-simplex_handle}
				\end{figure}

			\item[$m=2$]
				Let 
				\[
					\boldsymbol V := \{X_1^{2n}, X_2^{2n}, X_3^{2n}, V^{2n-2}_1,V^{2n-2}_2,V^{2n-2}_3,W^{2n-4}, \iota_1, \iota_2, \iota_3, \iota_{12}, \iota_{13}, \iota_{23}\}\, ,
				\]
				\begin{itemize}
					\item $X_1, X_2$ and $X_3$ are Weinstein $2n$-manifolds corresponding to vertices of $\varDelta^2$.
					\item $V_1, V_2$ and $V_3$ are Weinstein $(2n-2)$-manifolds corresponding to edges of $\varDelta^2$.
					\item $(W,\lambda_W)$ is a Weinstein $(2n-4)$-manifold corresponding to the face of $\varDelta^ 2$.
					\item $\iota_i \colon W \hookrightarrow \partial V_i$ is a Weinstein hypersurface for $i\in \left\{1,2,3\right\}$ corresponding to the inclusion of edges into the face.
					\item $\iota_{12} \colon V_1 \#_W V_2 \hookrightarrow \partial X_3$, $\iota_{23} \colon V_2 \#_W V_3 \hookrightarrow \partial X_1$ and $\iota_{31} \colon V_3 \#_W V_1 \hookrightarrow \partial X_2$ are Weinstein hypersurfaces corresponding to the inclusion of vertices into $\varDelta^2$.
				\end{itemize} 

				\begin{figure}[!htb]
					\centering
					\includegraphics{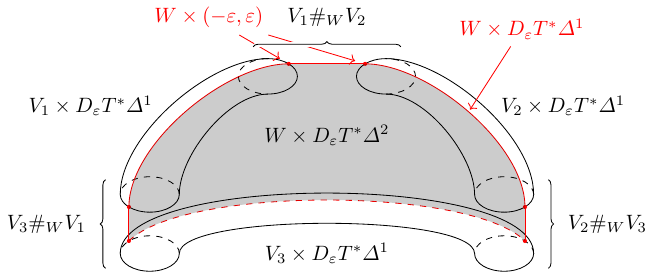}
					\caption{The result $\mathcal H^2(\boldsymbol V)$ after the first gluing step in the construction of a a simplicial decomposition.}
					\label{fig:2-simplex_handle}
				\end{figure}

				Starting off we consider three simplicial handles $V_i \times D_\varepsilon T^\ast \varDelta^1$ for $i\in \left\{1,2,3\right\}$ as in the $m=1$ case above. Using each of the Weinstein hypersurfaces $\iota_i \in \boldsymbol V$ for $i\in \left\{1,2,3\right\}$ we get Weinstein hypersurfaces
				\[
					\iota_i \times \id \colon W \times D_\varepsilon T^\ast \varDelta^1 \hookrightarrow \partial_+(V_i \times D_\varepsilon T^\ast \varDelta^1), \quad i\in \left\{1,2,3\right\}\, .
				\]
				Then consider the simplicial handle $(W \times D_\varepsilon T^\ast \varDelta^2, \lambda_W + 2y_1dx_1 + x_1dy_1 + 2y_2dx_2+x_dy_2)$ where $(y_1,y_2)$ are coordinates in the $\varDelta^2$-factor and $(x_1,x_2)$ are coordinates in the fiber directions. The negative end of this Weinstein cobordism is $W \times D_\varepsilon T^\ast (\eval[0]{\varDelta^2}_{\partial \varDelta^2})$. Now glue $W \times D_\varepsilon T^\ast \varDelta^2$ to $\bigsqcup_{i=1}^3 (V_i \times D_\varepsilon T^\ast \varDelta^1)$ using the Weinstein hypersurfaces $\iota_i \times \id$ as attaching maps and denote the result by $\mathcal H^2(\boldsymbol V)$, see \cref{fig:2-simplex_handle}. The resulting Weinstein cobordism $\mathcal H^2(\boldsymbol V)$ is a Weinstein cobordism with negative end
				\[
					\partial_- \mathcal H^2(\boldsymbol V) = (V_1 \#_W V_2) \sqcup (V_2 \#_W V_3) \sqcup (V_3\#_W V_1)\, .
				\]
				Finally we attach $\mathcal H^2(\boldsymbol V)$ to $\bigsqcup_{i=1}^3 X_i$ using the Weinstein hypersurfaces $\iota_{12},\iota_{23}, \iota_{31} \in \boldsymbol V$ as attaching maps, see \cref{fig:triple_join}. We denote the resulting Weinstein $2n$-manifold by $\# \boldsymbol V$.

				\begin{figure}[!htb]
					\centering
					\includegraphics{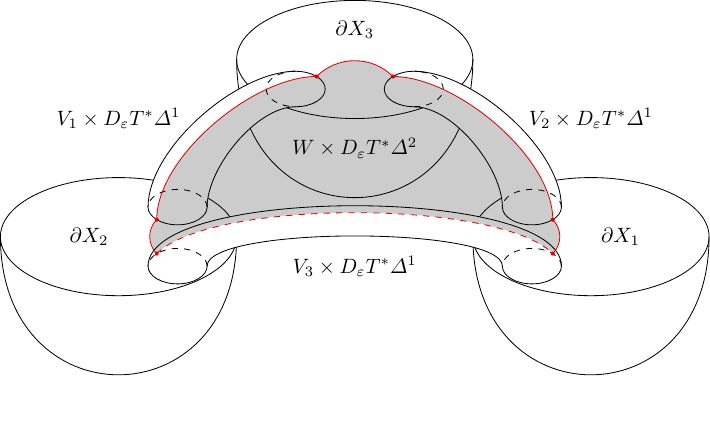}
					\caption{A simplicial decomposition of $X \cong \# \boldsymbol V$.}
					\label{fig:triple_join}
				\end{figure}
				If $X \cong \# \boldsymbol V$, we say that $(\varDelta^2, \boldsymbol V)$ is a simplicial decomposition of $X$.
		\end{description}

		The general construction of a simplicial decomposition of $X$ for any $m$-dimensional simplicial complex $C$ is done in \cref{sec:construction_m_simplex_handle}.

	\subsection{Legendrian attaching spheres and dg-subalgebras}\label{sec:intro_leg_attach_and_sectorial_covers}
		Let $(C, \boldsymbol V)$ be a simplicial decomposition of a Weinstein manifold $X$. Let $\boldsymbol V_0$ be obtained by replacing every Weinstein manifold $V \in \boldsymbol V$ with its subcritical part $V_0$. Let $\varSigma(\boldsymbol h)$ denote the union of the Legendrian attaching $(n-1)$-spheres of $X$ in $\partial X_0$. This union can be decomposed as
		\[
			\varSigma(\boldsymbol h) = \bigcup_{\substack{\sigma_k\in C_k \\ 0 \leq k \leq m}} \varSigma_{\sigma_k}(\boldsymbol h)\, ,
		\]
		where $\varSigma_{\sigma_k}(\boldsymbol h)$ is a collection of Legendrian $(n-1)$-spheres in $\partial X_0$ associated to each $k$-face of $C$. Roughly speaking, $\varSigma_{\sigma_k}(\boldsymbol h)$ lives over the part of $X_0$ corresponding to the $k$-face $\sigma_k$ in the surgery description of $X_0$, and is constructed from the union of top core disks of $V_{\sigma_k}^{2n-2k}$, see \cref{fig:intro_attaching_spheres} and \cref{sec:leg_attaching_data_for_simplex_handles} for details. We use the notation $\varSigma_{\supset \sigma_k}(\boldsymbol h) := \bigcup_{\substack{\sigma_i\in C_i \\ k \leq i \leq m}} \varSigma_{\sigma_i}(\boldsymbol h)$. 

		For each $k$-face $\sigma_k\in C$ we consider the Weinstein cobordism $X(\sigma_k)$ which is obtained from $X$ by replacing the Weinstein manifolds $V^{2n-2i}_{\sigma_i} \in \boldsymbol V$ with a symplectization of a contactization for every $\sigma_i \not \supset \sigma_k$. Then it is the case that $\varSigma_{\supset \sigma_k}(\boldsymbol h)$ is the union of top attaching spheres of $X(\sigma_k)$, see \cref{lma:attaching_spheres_stopped_away_from_sigma_k}.
		\begin{figure}[!htb]
			\centering
			\includegraphics{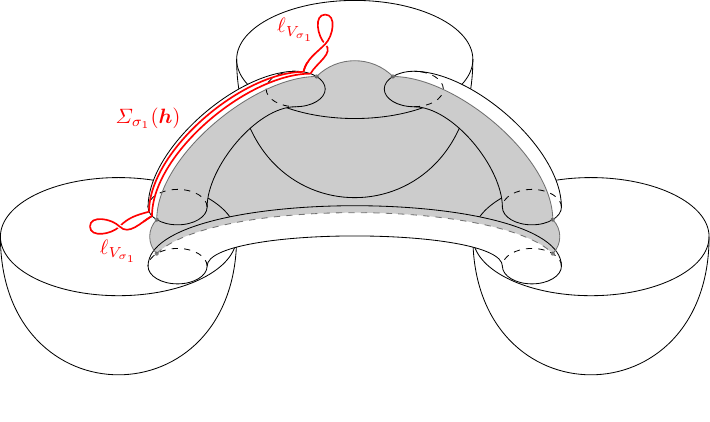}
			\caption{The figure shows the Legendrian sphere $\varSigma_{\sigma_1}(\boldsymbol h)$ in $\partial X_0$ which lives over the face $\sigma_1$ in the special case when $V_{\sigma_1} \in \boldsymbol V$ has exactly one Weinstein top handle whose core disk is denoted by $\ell_{V_{\sigma_1}}$.}
			\label{fig:intro_attaching_spheres}
		\end{figure}
		We denote the Chekanov--Eliashberg dg-algebra of $\varSigma(\boldsymbol h)$ by $CE^\ast(\varSigma(\boldsymbol h); X_0)$, and associated to each $k$-face $\sigma_k\in C$ we have the Chekanov--Eliashberg dg-algebra $CE^\ast(\varSigma_{\supset \sigma_k}(\boldsymbol h); X(\sigma_k)_0)$ appearing in \cref{thm:ce_descent_intro}. The generators of these dg-algebras are Reeb chords, ie trajectories of the Reeb vector field that start and end on the Legendrian submanifold. It is shown in that each Reeb chord in the generating set of $CE^\ast(\varSigma(\boldsymbol h); X_0)$ are located at the ``center'' of each face of the simplicial complex $C$, see \cref{lma:Reeb_chords_one_to_one_correspondence}. In a similar spirit, each Reeb chord in the generating set of $CE^\ast(\varSigma_{\supset \sigma_k}(\boldsymbol h); X(\sigma_k)_0)$ located at the ``center'' of each face $\sigma_i$ that contains the face $\sigma_k$. It is clear from the above description that for an inclusion of faces $\sigma_\ell \subset \sigma_k$, there is an inclusion from the set of generators of $CE^\ast(\varSigma_{\supset \sigma_k}(\boldsymbol h); X(\sigma_k)_0)$ to the set of generators of $CE^\ast(\varSigma_{\supset \sigma_\ell}(\boldsymbol h); X(\sigma_\ell)_0)$. 

		In fact, being a little bit more precise, we can upgrade this statement to an inclusion of the entire Chekanov--Eliashberg dg-algebras. Namely, let $\boldsymbol \varepsilon$ denote a parameter which controls the size of all the handles going into the surgery description of $X$ and $X(\sigma_k)$ coming from the given simplicial decomposition. Let $\mathfrak a > 0$ be a real number and consider the Chekanov--Eliashberg dg-algebra generated by Reeb chords of action bounded from above by $\mathfrak a$. For any $\mathfrak a > 0$, we can shrink the size $\boldsymbol \varepsilon$ a sufficient amount so that we obtain an inclusion of dg-algebras
		\begin{equation}\label{eq:intro_dg-subalg}
			CE^\ast(\varSigma_{\supset \sigma_k}(\boldsymbol h, \boldsymbol \varepsilon); X(\sigma_k)_0^{\boldsymbol \varepsilon}, \mathfrak a) \subset CE^\ast(\varSigma(\boldsymbol h, \boldsymbol \varepsilon); X_0^{\boldsymbol \varepsilon}, \mathfrak a)\, .
		\end{equation}
		Similarly, if $\sigma_\ell \subset \sigma_k$ is inclusion of faces, we have an inclusion of dg-algebras
		\[
			CE^\ast(\varSigma_{\supset \sigma_k}(\boldsymbol h, \boldsymbol \varepsilon); X(\sigma_k)_0^{\boldsymbol \varepsilon}, \mathfrak a) \subset CE^\ast(\varSigma_{\supset \sigma_\ell}(\boldsymbol h, \boldsymbol \varepsilon); X(\sigma_\ell)_0^{\boldsymbol \varepsilon}, \mathfrak a)\, .
		\]
		The key point in proving that we have such an inclusion involves showing that holomorphic curves have controlled behavior with respect to the simplicial decomposition, see \cref{sec:acs_and_hol_curves_in_simplicial_handles} for details.
		
	\addtocontents{toc}{\protect\setcounter{tocdepth}{-5}}
	\subsection*{Outline}
		In \cref{sec:simplicial_handles} we first define the basic building blocks used to inductively build simplicial handles. In \cref{sec:construction_m_simplex_handle} we define simplicial decompositions. In \cref{sec:leg_attaching_data_for_simplex_handles} we geometrically describe the Legendrian attaching spheres for the top Weinstein handles. In \cref{sec:acs_and_hol_curves_in_simplicial_handles} we give an account for almost complex structures and holomorphic disks. In \cref{sec:simplical_descent} we prove the main result.

		\cref{sec:simplicial_decompositions} is devoted to the one-to-one correspondence between simplicial decompositions and good sectorial covers. In \cref{sec:rel_to_sectorial_descent} we show that the main result in addition to the surgery formula of Bourgeois--Ekholm--Eliashberg recovers sectorial descent in wrapped Floer cohomology which is implied by the sectorial descent of Ganatra--Pardon--Shende \cite{ganatra2022sectorial}.

		In \cref{sec:applications} we discuss applications. We show how simplicial decompositions allow for the study of Legendrian submanifolds-with-boundary-and-corners in \cref{sec:relative_version}, and we describe the generators geometrically from a surgery perspective. In \cref{sec:ce_cosheaf} we construct the Chekanov--Eliashberg cosheaf. Moreover in \cref{sec:plumbing} we compute the Chekanov--Eliashberg dg-algebra of the Legendrian top attaching spheres of a plumbing of $T^\ast S^n$ for $n\geq 3$, and show that it is quasi-isomorphic to the Ginzburg dg-algebra.
	\addtocontents{toc}{\protect\setcounter{tocdepth}{2}}
	\subsection*{Acknowledgments}
		The author would like to thank Tobias Ekholm, Georgios Dimitroglou Rizell and Oleg Lazarev for valuable discussions. He is also grateful to Do\u{g}ancan Karaba\c{s} and Zhenyi Chen for their interest in the project and careful reading of an earlier version of this paper. Finally the author thanks the anonymous referee for many suggestions which improved the paper. The author was supported by the Knut and Alice Wallenberg Foundation.
\section{Simplicial decompositions}\label{sec:simplicial_handles}
	In \cref{sec:basic_building_blocks} we define a certain Weinstein cobordism which constitute the basic building block used to construct simplicial handles in \cref{sec:construction_m_simplex_handle}. Further in \cref{sec:construction_m_simplex_handle} we construct simplicial decompositions in which one of the main characters is a simplicial handle. After discussing almost complex structures and holomorphic disks in simplicial handles in \cref{sec:acs_and_hol_curves_in_simplicial_handles} we finally state and prove the main result of the paper in \cref{sec:simplical_descent}.
	\subsection{Basic building blocks for simplicial handles}\label{sec:basic_building_blocks}
		The presentation of the material in this section closely follows and generalizes \cite[Section 2.2]{asplund2021chekanov}.
		
		In this section we fix some integer $m \in \left\{1,\ldots,n\right\}$. Let $(V, \lambda)$ be a Weinstein $(2n-2m)$-manifold with Liouville form $\lambda$, Liouville vector field $z$, handle decomposition $h$ and exhausting Morse function $\phi \colon V \longrightarrow [0,\infty)$ with a single non-degenerate minimum with value $0$. Let $\varDelta^m \subset \R^m$ denote an $m$-simplex centered at the origin. We add the technical assumptions that $c_1(V) = 0$ and $\pi_1(V) = 0$, to make formulas for grading and dimension in this section simpler. They are not needed for any result in this paper, cf \cite[Section 1.5]{eliashberg2000introduction}, \cite[Section 2.2]{ekholm2007legendrian} and \cite[Section 7.3]{ekholm2017symplectic}.

		The basic building block is the Weinstein cobordism $V \times T^\ast \varDelta^m$ with negative end $V \times T^\ast \eval[0]{\varDelta^m}_{\partial \varDelta^m}$. The cobordism $V \times T^\ast \varDelta^m$ furthermore comes with a handle decomposition induced by $h$. We will use an explicit model of the handle that we describe next.

		Consider symplectic $\R^{2m}$ with coordinates $(x_1,\ldots,x_m,y_1,\ldots,y_m)$ and symplectic form $\sum_{i=1}^m dx_i \wedge dy_i$. Consider the product $V \times \R^{2m}$ with symplectic form $\omega = d \lambda + \sum_{i=1}^m dx_i \wedge dy_i$. Then 
		\begin{equation}\label{eq:basic_building_block_Liouville}
			Z = z + \sum_{i=1}^m (2x_i \partial_{x_i} - y_i \partial_{y_i})\, ,
		\end{equation}
		is a Liouville vector field for $\omega$, and $\phi(v) + \sum_{i=1}^m (x_i^2 - \frac 12 y_i^2)$ is an exhausting Morse function on $V \times \R^{2m}$. Let $H^{m}_{\varepsilon }(V) \cong V \times \R^{2m}$ be a Weinstein cobordism with positive and negative ends the contact hypersurfaces
		\begin{equation}\label{eq:building_blocks_contact_boundaries}
			G^{m}_{\pm \varepsilon}(V) = \left\{(v,\boldsymbol x,\boldsymbol y) \in V \times \R^{2m} \suchthat \phi(v) + \sum_{i=1}^m \left(x_i^2 - \frac 12 y_i^2\right) = \pm \varepsilon\right\}\, ,
		\end{equation}
		with induced contact form
		\[
			\alpha_{\pm \varepsilon} = \eval{\left(\lambda + \sum_{i=1}^m (2x_i dy_i + y_i dx_i)\right)}_{G^{m}_{\pm \varepsilon}}\, .
		\]
		The corresponding Reeb vector field is
		\begin{equation}\label{eq:basic_building_block_reeb}
			R_{\pm \varepsilon} = \beta(v,\boldsymbol x, \boldsymbol y) \left(\sum_{i=1}^m (2x_i \partial_{y_i} + y_i \partial_{x_i}) + r_\lambda(v, \boldsymbol x, \boldsymbol y)\right)\, ,
		\end{equation}
		where 
		\[
			\beta(v,\boldsymbol x, \boldsymbol y) = \left(\sum_{i=1}^m (4x_i^2+y_i^2) + \eval[0]{\lambda}_{G^m_{\pm \varepsilon}}(r_\lambda(v,\boldsymbol x, \boldsymbol y))\right)^{-1}\, ,
		\]
		and $r_\lambda$ is a smooth vector field in $\ker \eval[0]{d \lambda}_{G^m_{\pm \varepsilon}}$ at points where $z\neq 0$, and equal to zero where $z = 0$.
		\begin{lma}\label{lma:reeb_vf_over_the_origin}
			The subset $\left\{(v,\boldsymbol 0, \boldsymbol 0) \in V \times \R^{2m} \suchthat \phi(v) = \varepsilon\right\} \subset G^{m}_\varepsilon$ is invariant under the flow of $R_\varepsilon$ and in this subset $R_\varepsilon$ agrees with the Reeb vector field on $(\partial V,\eval[0]{\lambda}_{\partial V})$.
		\end{lma}
		\begin{proof}
			It follows from explicit expression \eqref{eq:basic_building_block_reeb} of the Reeb vector field on $G^{m}_\varepsilon$. Namely, at points $(v,\boldsymbol 0, \boldsymbol 0)$ we have
			\[
				R_{\varepsilon}(v, \boldsymbol 0, \boldsymbol 0) = \frac{r_\lambda(v,\boldsymbol 0, \boldsymbol 0)}{\eval[0]{\lambda}_{\left\{\phi(v) =  \varepsilon\right\}}(r_\lambda(v, \boldsymbol 0, \boldsymbol 0))}\, ,
			\]
			whose flow keeps $\left\{(v,\boldsymbol 0, \boldsymbol 0) \in V \times \R^{2m} \suchthat \phi(v) = \varepsilon\right\} \subset G^{m}_\varepsilon$ invariant since $r_\lambda(v, \boldsymbol 0,\boldsymbol 0) \in \ker \eval[0]{d \lambda}_{\left\{\phi(v) = \varepsilon\right\}} \subset \ker d \phi$, and this expression of $R_\varepsilon(v, \boldsymbol 0, \boldsymbol 0)$ agrees with the Reeb vector field on $(\partial V, \eval[0]{\lambda}_{\partial V})$.
		\end{proof}
		Let $V_0$ denote the subcritical part of $V$. Let $\ell = \bigcup_{j} \ell_j$ denote the union of the core disks of the top handles $h_j$ in $h$, and similarly let $\partial \ell = \bigcup_{j} \partial \ell_j \subset \partial V_0$ be the union of the attaching $(n-m-1)$-spheres for the top handles $h_j$. Consider
		\begin{equation}\label{eq:dl_trivial_extension}
			\overline{\partial \ell} := \left\{(v,\boldsymbol x, \boldsymbol y) \in V_0 \times \R^{2m} \suchthat \boldsymbol x = \boldsymbol 0, \; v \in \partial \ell \times [0,\infty) ,\; (v,\boldsymbol 0, \boldsymbol y) \in G^{m}_{\varepsilon}(V_0)\right\}\, .
		\end{equation}
		Then $\overline{\partial \ell}$ is diffeomorphic to $\partial \ell \times \R^{m}$ and trivially extends $\partial \ell$ over $V_0 \times \R^{2m}$.

		For a Reeb orbit $\gamma \subset Y^{2n-1}$ we define the \emph{contact homology grading} as
		\begin{equation}\label{eq:contact_homology_grading}
			\abs \gamma := \CZ(\gamma) + (n-3)\, ,
		\end{equation}
		where $\CZ(\gamma)$ is the Conley--Zehnder index as defined in \cite[Section 1.2]{eliashberg2000introduction}. For a Reeb chord $c$ of a Legendrian submanifold $\varLambda^{n-1} \subset Y^{2n-1}$ we define its grading by
		\[
			\abs c := - \CZ(c) + 1\, ,
		\]
		where $\CZ(c)$ is the Conley--Zehnder index of $c$ as defined in \cite[Section 2.2]{ekholm2005contact}. We write $\abs{\gamma}_{Y}$ and $\abs c_Y$ when we want to make the ambient manifold in which the Conley--Zehnder index in either case is computed explicit. Note that this is the cohomological grading for the Chekanov--Eliashberg dg-algebra grading which coincides with the grading used in \cite{ekholm2017duality,asplund2021chekanov} and differs by a sign from the grading used in \cite{bourgeois2012effect}.
		\begin{lma}\label{lma:basic_building_block_reeb}
			The Reeb chords of  $\overline{\partial \ell} \subset G^{m}_{\varepsilon}$ all lie over $(\boldsymbol x, \boldsymbol y) = (\boldsymbol 0, \boldsymbol 0)$ and are in one-to-one grading preserving correspondence with the Reeb chords of $\partial \ell \subset \partial V_0$. Furthermore, the Reeb orbits in $G^{m}_{\varepsilon}$ are in one-to-one correspondence with Reeb orbits in $\partial V_0$, and for a Reeb orbit $\gamma$, we have $\abs{\gamma}_{G^{m}_{\varepsilon}} = \abs{\gamma}_{\partial V_0} + m$.
		\end{lma}
		\begin{proof}
			That all Reeb orbits of $G^m_\varepsilon$ and Reeb chords of $\overline{\partial \ell}$ lie over $(\boldsymbol x, \boldsymbol y) = (\boldsymbol 0, \boldsymbol 0)$ follows from the explicit description of the flow of the Reeb vector field $R_\varepsilon$ in \eqref{eq:basic_building_block_reeb}. After changing the speed of $R_\varepsilon$, the $(\boldsymbol x, \boldsymbol y)$-coordinate of the flow of $R_\varepsilon$ is described by the solution of the following system of linear differential equations
			\[
				\begin{cases}
					\dot x_i(t) = y_i(t) \\
					\dot y_i(t) = 2x_i(t)
				\end{cases}\quad i\in \left\{1,\ldots,m\right\}\, .
			\]
			With the initial value $(\boldsymbol x(0), \boldsymbol y(0))$, the solution of this system of linear differential equations is given by
			\begin{equation}\label{eq:Reeb_flow_in_handle}
				\begin{cases}
					x_i(t) = x_i(0) \cosh(\sqrt 2 t) + \frac{y_i(0)}{\sqrt 2} \sinh(\sqrt 2 t) \\
					y_i(t) = y_i(0) \cosh(\sqrt 2 t) + \sqrt 2x_i(0) \sinh(\sqrt 2 t)
				\end{cases} \quad i \in \left\{1,\ldots,m\right\}\, .
			\end{equation}
			Unless $(\boldsymbol x(0), \boldsymbol y(0)) = (\boldsymbol 0, \boldsymbol 0)$, the $(\boldsymbol x, \boldsymbol y)$-coordinate of a flow line moves along the curve
			\[
				\left\{2x_i^2-y_i^2 = 2x_i(0)^2-y_i(0)\right\}\, ,
			\]
			with non-zero speed in each $(x_i,y_i)$-plane for $i\in \left\{1,\ldots,m\right\}$. The $(\boldsymbol x, \boldsymbol y)$-coordinate of flow lines starting at the origin are fixed by the expression in \eqref{eq:Reeb_flow_in_handle}.

			By \cref{lma:reeb_vf_over_the_origin} we have that $R_\varepsilon$ coincides with the Reeb vector field on $(\partial V_0, \eval[0]{\lambda}_{\partial V_0})$ over $(\boldsymbol x, \boldsymbol y) = (\boldsymbol 0, \boldsymbol 0)$. This is enough to conclude that there are one-to-one correspondences of Reeb orbits in $G^m_\varepsilon$, and Reeb chords of $\overline{\partial \ell} \subset G^m_\varepsilon$ with Reeb orbits in $\partial V_0$ and Reeb chords of $\partial \ell \subset \partial V_0$, respectively.

			What is left to show is that these one-to-one correspondences preserves index of Reeb chords, and raises index by $m$ of Reeb orbits, respectively. Since Reeb orbits in $G^m_\varepsilon$ lie over the origin $(\boldsymbol x, \boldsymbol y) = (\boldsymbol 0, \boldsymbol 0)$, and since $R_\varepsilon$ coincides with the Reeb vector field on $(\partial V_0, \eval[0]{\lambda}_{\partial V_0})$ over the origin, the Conley--Zehnder index of $\gamma$ considered as a Reeb orbit in $G^m_\varepsilon$ is unchanged through the one-to-one correspondence. However, since $\dim G^m_\varepsilon = 2n-1$ we have
			\[
				\abs \gamma_{G^m_\varepsilon} = \CZ(\gamma) + (n-3) = \abs \gamma_{\partial V_0} + m\, .
			\]
			The linearized Reeb flow in the $\boldsymbol y$-directions in $G^m_\varepsilon$ is responsible for the potential change in Conley--Zehnder index when comparing $\abs{c}_{G^m_\varepsilon}$ to $\abs{c}_{\partial V_0}$. By \eqref{eq:basic_building_block_reeb}, the Reeb vector field along $\overline{\partial \ell}$ is
			\[
				R_{\varepsilon}(v,\boldsymbol 0, \boldsymbol y) = \beta(v, \boldsymbol 0, \boldsymbol y) \left(\sum_{i=1}^m (y_i \partial_{x_i}) + r_\lambda(v,\boldsymbol 0,\boldsymbol y)\right)\, .
			\]
			After rescaling $R_\varepsilon(v,\boldsymbol 0, \boldsymbol y)$ the linearized Reeb flow is given by
			\[
				d \varPsi^t(v, \boldsymbol 0, \boldsymbol y) = d \psi^t(v) \oplus S(t)\, ,
			\]
			where $S(t) := \begin{pmatrix}\boldsymbol 1_{m \times m} & t \cdot \boldsymbol 1_{m \times m} \\ \boldsymbol 0_{m \times m} & \boldsymbol 1_{m \times m} \end{pmatrix}$ is a symplectic shear in the $(\boldsymbol x, \boldsymbol y)$-planes and $d \psi^t$ is the linearized flow of the Reeb vector field $r_\lambda$ in $\partial V_0$. The tangent space of $\overline{\partial \ell}$ is the $y$-axis, and we see that if $T_{(v,\boldsymbol 0,\boldsymbol y)} \overline{\partial \ell}$ is spanned by $\begin{pmatrix}\boldsymbol 0_{m \times 1} \\ \boldsymbol 1_{m \times 1}\end{pmatrix}$, then  $d \varPsi^1(v,\boldsymbol 0, \boldsymbol y)(T_{(v,\boldsymbol 0,\boldsymbol y)} \overline{\partial \ell})$ is spanned by $\begin{pmatrix}\boldsymbol 1_{m \times 1} \\ \boldsymbol 1_{m \times 1}\end{pmatrix}$ which is obtained from $\begin{pmatrix}\boldsymbol 0_{m \times 1} \\ \boldsymbol 1_{m \times 1}\end{pmatrix}$ by a negative rotation in the $(\boldsymbol x, \boldsymbol y)$-planes. Hence if we take the positive close-up the Conley--Zehnder index is unchanged, and therefore $\abs c_{G^m_\varepsilon} = \abs c_{\partial V_0}$.
		\end{proof}

	\subsubsection{Almost complex structure and holomorphic disks in the basic building block}\label{sec:acs_hol_curves_basic_building_block}
		We first define an almost complex structure on $V$ as follows. Each Weinstein handle $h_V$ of $V$ comes equipped with the Weinstein Morse function $\phi_V$, and we have contact hypersurfaces $\partial_{\pm \delta} h_V := \phi_{h_V}^{-1}(\pm \delta)$. Pick an almost complex structure $J_{V,0}$ along $\partial_\delta h_V$ which takes the Liouville vector field to the Reeb vector field. We extend $J_{V,0}$ as an almost complex structure $J_V$ over $h_V$ that leaves the contact planes in $\partial_{\pm \delta} h_V$ invariant for all $\delta > 0$, and agrees with the Liouville invariant almost complex structure on the symplectization ends built on $\partial_{\pm \delta}h_V$ for any $\delta > 0$. We call such $J_V$ \emph{handle adapted}.

		Let $z$ be the Liouville vector field of $V$, and let $r := J_Vz$. We fix an exact symplectomorphism between $\widetilde H^{m}_{\varepsilon }(V) :=H^m_\varepsilon(V) \setminus (\Skel(V) \times \left\{\boldsymbol x=0\right\})$ and the symplectization of $\partial_+H^m_\varepsilon(V)$. We will define an almost complex structure on $G^m_\varepsilon$ that takes the Liouville vector field $Z$ to the Reeb vector field $R_\varepsilon$ and then extend it to all of $\widetilde H^m_\varepsilon(V)$ by translation along the Liouville vector field.

		Consider the tangent space of $G^m_\varepsilon(V)$. Then the following holds true.
		\begin{itemize}
			\item At $(\boldsymbol x, \boldsymbol y) = (\boldsymbol 0, \boldsymbol 0)$, the tangent space is spanned by $\ker d \phi$, $\partial_{x_i}$ and $\partial_{y_i}$ for $i\in \left\{1,\ldots,m\right\}$. Here we define $J$ as $J_V$ on $TV$, and $J \partial_{x_i} = \partial_{y_i}$ for $i\in \left\{1,\ldots,m\right\}$.
			\item At points $(v, \boldsymbol x, \boldsymbol y) \in G^m_\varepsilon(V)$, where $v \not\in \crit(\phi)$ and $(\boldsymbol x, \boldsymbol y) \neq (\boldsymbol 0, \boldsymbol 0)$, the tangent space is the sum of the $(2n-2m-1)$-dimensional space $\ker d \phi$ and the $2m$-dimensional subspaces spanned by the vector fields $-(4x_i^2+y_i^2)z + d \phi(z)(2x_i \partial_{x_i} - y_i \partial_{y_i})$ and $2x_i \partial_{y_i} + y_i \partial_{x_i}$ for $i\in \left\{1,\ldots,m\right\}$. Here we define $J$ as the restriction of $J_V$ on the contact hyperplanes in $\ker d \phi$, and
			\[
				Jz = \beta(v,\boldsymbol x, \boldsymbol y) r, \quad J(2x_i \partial_{x_i} - y_i \partial_{y_i}) = \beta(v,\boldsymbol x, \boldsymbol y) (2x_i \partial_{y_i}+y_i \partial_{x_i}) \quad i \in \left\{1,\ldots,m\right\}\, .
			\]
			\item At points $(v,\boldsymbol x, \boldsymbol y)$ where $v\in \crit(\phi)$, we define $J$ to be equal to $J_V$ on $TV$ and
			\[
				J(2x_i \partial_{x_i} - y_i \partial_{y_i}) = \frac{1}{\sum_{i=1}^m (4x_i^2 + y_i^2)}(2x_i \partial_{y_i} + y_i \partial_{x_i}) \quad i\in \left\{1,\ldots,m\right\}\, .
			\]
			(Note that $v\in \crit(\phi)$ and $(v, \boldsymbol x, \boldsymbol y) \in G^m_\varepsilon(V)$ imply $(\boldsymbol x, \boldsymbol y) \neq (\boldsymbol 0, \boldsymbol 0)$.)
		\end{itemize}
		Since $J_V$ is handle adapted it follows that $J$ is smooth. We finally extend $J$ by translation along the Liouville vector field. Let $\pi_{i} \colon H^m_\varepsilon(V) \longrightarrow \R^{2}_{(x_i,y_i)}$ be the projection to the $(x_i,y_i)$-plane.
		\begin{lma}\label{lma:complexproj}
			The codimension $2$ hyperplanes $V_{(x_i,y_i)} := \pi^{-1}_i(x_i,y_i)$ are $J$-complex. In particular, any holomorphic disk in the symplectization of $G^m_\varepsilon(V)$ intersecting $V_{(x_i,y_i)}$ is either contained in $V_{(x_i,y_i)}$ or intersects $V_{(x_i,y_i)}$ positively. Moreover the intersection number is locally constant in $(x_i,y_i)$.
		\end{lma}
		\begin{proof}
		The splitting $TH^m_\varepsilon(V) = TV \oplus T \R^{2m-2} \oplus T \R^2$ is preserved along the flow lines of the Liouville vector field $Z = z + \sum_{i=1}^m (2x_i \partial_{x_i} - y_i \partial_{y_i})$ and $J$ takes $TV_{(x_i,y_i)}$ to $TV_{(x_i,y_i)}$ by definition.

		By work of Moreno and Siefring \cite{moreno2019holomorphic,siefringsymplectic} it follows that intersections between holomorphic disks and $V_{(x_i,y_i)}$ are contained in a compact set if they exist. By local analysis found in eg \cite[Lemma 3.4]{ionel2003relative} it follows that such a holomorphic disk is either contained in $V_{(x_i,y_i)}$ or intersects $V_{(x_i,y_i)}$ positively.
		\end{proof}
		Before we consider holomorphic disks in $\widetilde H^m_\varepsilon(V_0)$ with boundary on $\R \times \overline{\partial \ell}$ which play a role in the Chekanov--Eliashberg dg-algebra of $\overline{\partial \ell} \subset G^m_\varepsilon$, we touch upon the technical parts of the definition of the Chekanov--Eliashberg dg-algebra of $\partial \ell \subset \partial V_0$ involving anchoring.

		By \cite[Lemma 2.1]{asplund2021chekanov} we have that given a generic Weinstein handle decomposition of $V_0$, it holds that any Reeb orbit in $\partial V_0$ lies in the middle of some standard Weinstein handle of index $k$ (ie the orbit projects to $(\boldsymbol 0,\boldsymbol 0) \in D^k \times D^{2n-2m-k}$). Then fix an \emph{index definite} Weinstein structure on $V_0$. Namely, choose a Liouville form such that near the middle of each Weinstein handle along the contact boundary it is arbitrarily close to the standard contact form, and that it moreover is such that the minimal Conley--Zehnder index of a Reeb orbit is $(n-m-k)+1$, see \cite[Lemma 2.1]{asplund2021chekanov}. For $\dim V_0 = 2n - 2m > 2$ this means that the minimal grading of any Reeb orbit is $1$ (which happens for the critical index $k = n-m$). In the lowest dimensional case, $\dim V_0 = 2n-2m = 2$ there are simple Reeb orbits with grading $0$ and its effects must be taken into account, see \cref{rmk:anchoring} for context regarding the role of anchoring in the Chekanov--Eliashberg dg-algebra. For a complete discussion on these matters, see \cite[Section 2.3]{asplund2021chekanov}.

		\begin{rmk}\label{rmk:anchoring}
			The Chekanov--Eliashberg dg-algebra of a Legendrian $\varLambda \subset \partial X$ was originally defined by Chekanov \cite{chekanov2002differential} and Eliashberg \cite{eliashberg1998invariants} and is part of the more general symplectic field theory package \cite{eliashberg2000introduction}. Its differential is typically defined by counting $J$-holomorphic disks in $\R \times \partial X$ with boundary condition on $\R \times \varLambda$. The fact that $\partial X$ is filled by $X^{\text{in}}$ actually means that we need to count \emph{anchored} $J$-holomorphic disks, see \cref{fig:anchored_disks}. 
			\begin{figure}[!htb]
				\centering
				\includegraphics[scale=0.75]{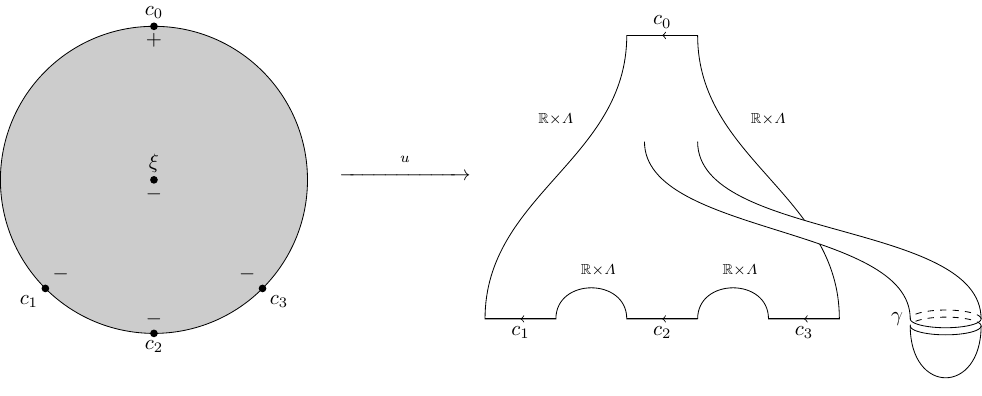}
				\caption{An anchored $J$-holomorphic disk involved in the definition of the differential of the Chekanov--Eliashberg dg-algebra.} \label{fig:anchored_disks}
			\end{figure}
			These are $J$-holomorphic buildings with two levels, with one level in $\R \times \partial X$ and one level in $X^{\text{in}}$. The domain of the map whose target is $\R \times \partial X$ is a disk with boundary punctures and interior punctures near which the map is asymptotic to a cylinder over a contractible Reeb chord near $\left\{-\infty\right\} \times \partial X$. Each component of the map whose target is $X^{\text{in}}$ is a $J$-holomorphic map $S^2 \setminus \left\{\text{pt}\right\} \longrightarrow X^{\text{in}}$ which near the puncture is asymptotic to a contractible Reeb orbit $\gamma$ of $\partial X$. The moduli space of such maps has dimension $\abs \gamma$ using the well-known dimension formula \cite[Theorem A.1]{cieliebak2010compactness}, where $\abs \gamma$ is the contact homology grading \eqref{eq:contact_homology_grading}, and only zero-dimensional moduli spaces are used in the definition of the differential. In particular, if $\abs \gamma \geq 1$, the differential can be defined without anchoring. This is our version of dynamical convexity or index positivity, cf eg \cite[Definition 1.1]{hutchings2016cylindrical} and \cite[Section 9.5]{cieliebak2018symplectic}.
		\end{rmk}

		\begin{lma}\label{lma:middlesubcrit}
			Assume $n-m > 1$. For any $\mathfrak a > 0$ there is some $\varepsilon > 0$ small enough so that any $J$-holomorphic disk in the symplectization $\R \times G^m_\varepsilon$ with boundary on $\R \times \overline{\partial \ell}$ and with positive puncture at a Reeb chord over $(\boldsymbol x, \boldsymbol y) = \boldsymbol 0$ corresponding to a Reeb chord of $\partial \ell$ of action $ < \mathfrak a$, lies entirely over $(\boldsymbol x,\boldsymbol y) = \boldsymbol 0$. Also, there is a one-to-one correspondence between such rigid $J$-holomorphic disks and rigid $J_{V_0}$-holomorphic disks in $\R \times \partial V_0$ with boundary on $\R \times \partial \ell$.
		\end{lma}
		\begin{proof}
			This is exactly \cite[Lemma 2.7]{asplund2021chekanov}
		\end{proof}
		\begin{cor}\label{cor:dgsubalg}
		For any $\mathfrak a > 0$ there is some $\varepsilon > 0$ small enough so that the Chekanov--Eliashberg dg-algebra of $\overline{\partial \ell} \subset G^m_\varepsilon(V_0)$ generated by Reeb chords of action $< \mathfrak a$ is canonically isomorphic to the Chekanov--Eliashberg dg-algebra of $\partial \ell \subset \partial V_{0}$ generated by Reeb chords of action $< \mathfrak a$.
		\end{cor}
		\begin{proof}
			The correspondence of generators follows from \cref{lma:basic_building_block_reeb}. The correspondence of holomorphic disks contributing to the differential follows from \cref{lma:middlesubcrit} in the case $n - m > 1$. In the case $n - m = 1$ it is more technical as one has to take anchored $J$-holomorphic disks into account (see \cref{rmk:anchoring}) and it follows from the corresponding identification of holomorphic disks described in \cite[Remark 2.9]{asplund2021chekanov} and \cite{ekholm2015legendrian}.
		\end{proof}
	\subsection{Construction of simplicial decompositions}\label{sec:construction_m_simplex_handle}
		The main purpose of this section is to define simplicial decompositions of a Weinstein manifold $X$, which gives a surgery presentation of $X$. Along the way we construct simplicial handles by gluing together basic building blocks which are used to define a joining operation that generalizes Weinstein connected sum.
		\subsubsection{Gluing of basic building blocks}\label{sec:gluing_of_basic_building_blocks}
			Let $V^{(k)}$ and $W^{(k+1)}$ be two Weinstein $(2n-2k)$ and $(2n-2k-2)$-dimensional manifolds. Let $0 < \varepsilon_W < \varepsilon_V$ and recall the construction of the basic building blocks $H^{k+1}_{\varepsilon_W}(W)$ and $H^k_{\varepsilon_V}(V)$ from \cref{sec:basic_building_blocks}. Using the given Weinstein hypersurface $\iota \colon W \hookrightarrow \partial V$ we now explain how we glue together $H^{k+1}_{\varepsilon_W}(W)$ and $H^k_{\varepsilon_V}(V)$.

			First note that $\iota$ induces a Weinstein hypersurface
			\[
				\iota \times \id \colon W \times \R^{2k} \hookrightarrow (V \times \R^{2k}) \setminus (\Skel V \times \left\{\boldsymbol x = \boldsymbol 0\right\})\, .
			\]
			Next we consider the negative cylindrical end of the basic building block $H^{k+1}_{\varepsilon_W}(W)$ which is identified with $(W \times \R^{2k+2}) \setminus \left(\Skel W \times \left\{\boldsymbol y = \boldsymbol 0\right\}\right)$. Picking some Weinstein hypersurface $j \colon \R^{2k} \hookrightarrow \R^{2k+2} \setminus \left\{\boldsymbol 0\right\}$ we obtain a Weinstein hypersurface
			\[
				\id \times j \colon W \times \R^{2k} \hookrightarrow (W \times \R^{2k+2}) \setminus \left(\Skel W \times \left\{\boldsymbol y = \boldsymbol 0\right\}\right)\, .
			\]
			Using $\id \times j$ and $\iota \times \id$ we may glue the negative boundary of $H^{k+1}_{\varepsilon_W}(W)$ to the positive boundary of $H^k_{\varepsilon_V}(V)$ along the Weinstein hypersurface $W \times \R^{2k}$, following \cite[Section 2.6.2]{alvarez2020positive}. We will write $\overline \iota$ to mean a choice of $j$ and ``the gluing map'' used to glue the basic building blocks together.
		\subsubsection{Join over a $1$-simplex}\label{sec:1-simplex_handles}
			Denote the vertices and the edge $\varDelta^1$ by $(\sigma_0)_1$, $(\sigma_0)_2$ and $\sigma_1$ respectively, with face inclusions $(\sigma_0)_1, (\sigma_0)_2 \supset \sigma_1$.

			\begin{dfn}[Join over a $1$-simplex]\label{dfn:join_m=1}
				Consider the tuple $(\varDelta^1, \boldsymbol V)$ where
				\[
					\boldsymbol V := \{V_{(\sigma_0)_1}^{(0)}, V_{(\sigma_0)_2}^{(0)}, V^{(1)}_{\sigma_1}, \iota_{(\sigma_0)_1}, \iota_{(\sigma_0)_2}\}\, ,
				\]
				where $\iota_{(\sigma_0)_i} \colon V^{(1)}_{\sigma_1} \hookrightarrow \partial V^{(0)}_{(\sigma_0)_i}$ is a Weinstein hypersurface for $i\in \left\{1,2\right\}$. Let $\boldsymbol \varepsilon := \varepsilon_1$ be a positive real number. Define
				\[
					H_{\boldsymbol \varepsilon}(\boldsymbol V) := \{V_{(\sigma_0)_1}, V_{(\sigma_0)_2}, H^1_{\varepsilon_1}(V_{\sigma_1}), \overline \iota_{(\sigma_0)_1}, \overline \iota_{(\sigma_0)_2}\}\, ,
				\]
				where $\overline \iota_{(\sigma_0)_i}$ are the induced Weinstein hypersurfaces of basic building blocks defined in \cref{sec:gluing_of_basic_building_blocks}. Define $\# \boldsymbol V$ to be the Weinstein $2n$-manifold obtained as the result of gluing $V_{(\sigma_0)_1} \sqcup V_{(\sigma_0)_2}$ to the negative end of $H^1_{\varepsilon_1}(V_{\sigma_1})$ using $\overline \iota_{(\sigma_0)_1}$  and $\overline \iota_{(\sigma_0)_2}$ as attaching maps, see \cref{fig:join_over_1_simplex}.
			\end{dfn}

			\begin{rmk}
				The joining operation for $m=1$ in \cref{dfn:join_m=1} is also known as Weinstein connected sum \cite{avdek2012liouville,eliashberg2018weinstein,alvarez2020positive}. Note that at the topological level this is not necessarily a boundary connected sum of the underlying manifolds with boundary, since $V^{(1)}$ may not be a disk.
			\end{rmk}

			\begin{dfn}
				Let $\boldsymbol V$ be a set of Weinstein manifolds and Weinstein hypersurfaces parametrized by the faces of a simplex, as in \cref{dfn:join_m=1}. For a given $k$-face $\sigma_k$, define $\boldsymbol V_{\supsetneq \sigma_k} \subset \boldsymbol V$ to be the subset consisting of only those Weinstein manifolds $V_{\sigma_i}\in \boldsymbol V$ for which $\sigma_i \supsetneq \sigma_k$, and the corresponding Weinstein hypersurfaces.
			\end{dfn}

			\begin{figure}[!htb]
				\centering
				\includegraphics[scale=0.75]{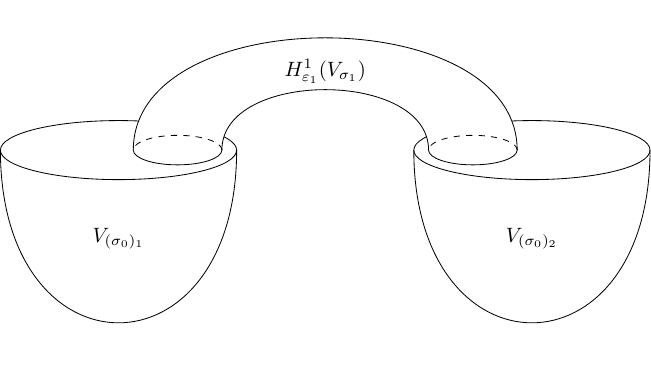}
				\caption{The Weinstein manifold $\# \boldsymbol V$ constructed as the gluing of various Weinstein manifolds over a $1$-simplex.}\label{fig:join_over_1_simplex}
			\end{figure}

		\subsubsection{Join over a $2$-simplex}\label{sec:2-simplex_handles}
			Denote the faces of $\varDelta^2$ by $(\sigma_0)_i$, $(\sigma_1)_j$ and $\sigma_2$ respectively where $i,j\in \left\{1,2,3\right\}$, with the obvious face inclusions.
			\begin{dfn}[Join over a $2$-simplex]\label{dfn:join_m=2}
				Consider the tuple $(\varDelta^2, \boldsymbol V)$ where
				\[
					\boldsymbol V := \{V^{(0)}_{(\sigma_0)_1},V^{(0)}_{(\sigma_0)_2},V^{(0)}_{(\sigma_0)_3}, V^{(1)}_{(\sigma_1)_1},V^{(1)}_{(\sigma_1)_2},V^{(1)}_{(\sigma_1)_3}, V^{(2)}, \iota_{(\sigma_1)_1},\iota_{(\sigma_1)_2},\iota_{(\sigma_1)_2},\iota_{(\sigma_0)_1},\iota_{(\sigma_0)_2},\iota_{(\sigma_0)_3}\},
				\]
				where $\iota_{(\sigma_1)_j} \colon V^{(2)} \hookrightarrow \partial V^{(1)}_{(\sigma_1)_j}$ is a Weinstein hypersurface for $j\in \left\{1,2,3\right\}$, and where 
				\[
					\iota_{(\sigma_0)_i} \colon \# \boldsymbol V_{\supset (\sigma_0)_i} \hookrightarrow \partial V_{(\sigma_0)_i}, \quad i \in \left\{1,2,3\right\}
				\]
				are Weinstein hypersurfaces. Let $\boldsymbol \varepsilon := \left\{0 < \varepsilon_2 < \varepsilon_1\right\}$ be a finite strictly decreasing sequence of positive real numbers. Define
				\[
					H_{\boldsymbol \varepsilon}(\boldsymbol V) := \{V^{(0)}_{(\sigma_0)_i}, H^1_{\varepsilon_1}(V^{(1)}_{(\sigma_1)_j}), H^2_{\varepsilon_2}(V^{(2)}), \overline \iota_{(\sigma_1)_j},\overline \iota_{(\sigma_0)_i}\},
				\]
				where $\overline \iota_{(\sigma_1)_j}$ and $\overline \iota_{(\sigma_0)_i}$ are the induced gluing maps basic building blocks defined in \cref{sec:gluing_of_basic_building_blocks}. 

				Define $\# \boldsymbol V$ to be the Weinstein $2n$-manifold obtained as follows:
				\begin{enumerate}
					\item First glue $V_{(\sigma_1)_1} \sqcup V_{(\sigma_1)_2} \sqcup V_{(\sigma_1)_3}$ to the negative end of $H^2_{\varepsilon_2}(V_{\sigma_2})$ using $\overline \iota_{(\sigma_1)_j}$ as gluing maps. The result is a Weinstein cobordism $X$ with negative ends $(\# \boldsymbol V_{\supsetneq (\sigma_0)_1}) \sqcup (\# \boldsymbol V_{\supsetneq (\sigma_0)_2}) \sqcup (\# \boldsymbol V_{\supsetneq (\sigma_0)_3})$, see \cref{fig:join_over_2_simplex_step1}.
					\item Finally glue $V_{(\sigma_0)_1} \sqcup V_{(\sigma_0)_2} \sqcup V_{(\sigma_0)_3}$ to the negative end of $X$, using $\overline \iota_{(\sigma_0)_i}$ as gluing maps, see \cref{fig:join_over_2_simplex}.
				\end{enumerate}
			\end{dfn}
			\begin{figure}[!htb]
				\centering
				\includegraphics{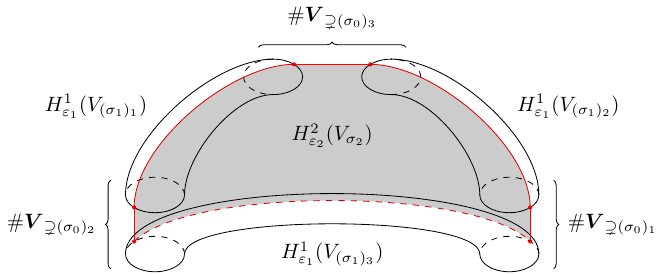}
				\caption{The Weinstein manifold $X$ obtained after gluing the edges to the $2$-face of the $2$-simplex.}\label{fig:join_over_2_simplex_step1}
			\end{figure}
			\begin{figure}[!htb]
				\centering
				\includegraphics{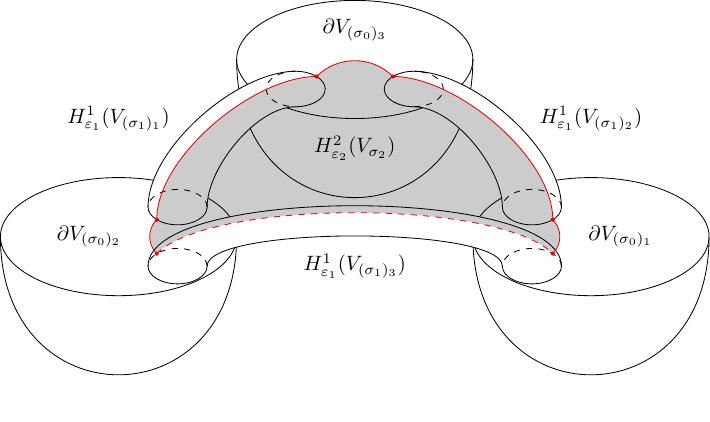}
				\caption{The Weinstein manifold $\# \boldsymbol V$ constructed as the gluing of various Weinstein manifolds over a $2$-simplex.}\label{fig:join_over_2_simplex}
			\end{figure}
		\subsubsection{Simplicial decompositions}\label{sec:simplicial_decompositions}
			We first define the the joining operation inductively. To that end, assume that the joining operation has been has been defined for all $p\in \left\{1,\ldots,m-1\right\}$.
			\begin{dfn}[Join over an $m$-simplex]\label{dfn:join_general_m}
				Consider the tuple $(\varDelta^m, \boldsymbol V)$ where $\boldsymbol V$ consists of
				\begin{enumerate}
					\item A Weinstein manifold $V^{(k)}_{\sigma_k}$ for each $k$-face $\sigma_k\in \varDelta^m$.
					\item A Weinstein hypersurface $\iota_{\sigma_k} \colon \# \boldsymbol V_{\supsetneq \sigma_k} \hookrightarrow \partial V^{(k)}_{\sigma_k}$ for each $k$-face $\sigma_k \in \varDelta^m$.
				\end{enumerate}
				Let $\boldsymbol \varepsilon := \left\{0 < \varepsilon_m < \cdots < \varepsilon_1\right\}$ be a finite strictly decreasing sequence of positive real numbers. Define $H_{\boldsymbol \varepsilon}(\boldsymbol V)$ as the set consisting of 
				\begin{enumerate}
					\item The basic building block $H^k_{\varepsilon_k}(V^{(k)}_{\sigma_k}) \cong V^{(k)}_{\sigma_k} \times D_{\varepsilon_k} T^\ast \varDelta^k$ for each $k$-face $\sigma_k\in \varDelta^m$.
					\item The induced gluing map $\overline \iota_{\sigma_k}$ of basic building blocks as defined in \cref{sec:gluing_of_basic_building_blocks} for each $k$-face $\sigma_k \in \varDelta^m$.
				\end{enumerate}
				Define $\# \boldsymbol V$ to be the Weinstein $2n$-manifold obtained as the successive gluing of basic building blocks using the gluing maps $\overline \iota_{\sigma_k}$, starting with $H^m_{\varepsilon_m}(V^{(m)}_{\sigma_m})$ and decreasing dimension by one for each successive gluing step, until all basic building blocks have been exhausted.
			\end{dfn}
			\begin{dfn}\label{dfn:surgery_parameter}
				Let $\boldsymbol \varepsilon := \left\{0 < \varepsilon_m < \cdots < \varepsilon_1\right\}$ be a finite strictly decreasing sequence of positive real numbers. For $c\in \R$ we define $\boldsymbol \varepsilon < c$ to mean $\varepsilon_1 < c$, and define $\boldsymbol \varepsilon > c$ to mean $\varepsilon_m > c$.
			\end{dfn}
			Before we define simplicial decompositions, let us generalize \cref{dfn:join_general_m} to arbitrary simplicial complexes instead of only simplices. Let $C$ be an $m$-dimensional simplicial complex. Denote by $F$ the set of facets of $C$, ie the set of faces in $C$ which are not contained in any other face in $C$.
			\begin{dfn}
				Let $C$ be a simplicial complex. For a $k$-face $\sigma_k \in C$ and a facet $f\in F$ containing $\sigma_k$, we define $\boldsymbol V_{\supsetneq \sigma_k, f} \subset \boldsymbol V$ to be the subset consisting of only those Weinstein manifolds $V_{\sigma_i}\in \boldsymbol V$ for which $f \supset \sigma_i \supsetneq \sigma_k$, and the corresponding Weinstein hypersurfaces.
			\end{dfn}
			\begin{dfn}[Join over a simplicial complex]\label{dfn:join_over_simplicial_complex}
				Let $(C, \boldsymbol V)$ be a tuple consisting of a simplicial complex and $\boldsymbol V$ as in \cref{dfn:join_general_m}, with the only exception that ``$\# \boldsymbol V_{\supsetneq \sigma_k}$'' should be replaced with ``$\bigsqcup_{f\in F}(\# \boldsymbol V_{\supsetneq \sigma_k,f})$''. Then define $\# \boldsymbol V$ in the same way as in \cref{dfn:join_general_m}.
			\end{dfn}
			\begin{dfn}[Simplicial decomposition]\label{dfn:simplicial_decomposition}
				A \emph{simplicial decomposition} of a Weinstein manifold $X$ is a tuple $(C, \boldsymbol V)$ consisting of a simplicial complex $C$, and a set $\boldsymbol V$ defined the same way as in \cref{dfn:join_general_m}, such that $X \cong \# \boldsymbol V$.
			\end{dfn}
			\begin{rmk}[Generalizations beyond simplicial complexes]\label{rmk:generalizations_beyond_simplicial_complexes}
				In the special case of $\dim C = 1$, we allow $C$ to be an arbitrary graph in the obvious way, which is useful for the application in \cref{sec:plumbing}. Namely, if $\sigma_1 \in C$ is a loop at the vertex $\sigma_0\in C$ the correct Weinstein hypersurface to consider in \cref{dfn:join_m=1} for attaching the $1$-simplex handle is $\iota \colon V^{(1)}_{\sigma_1} \sqcup V^{(1)}_{\sigma_1} \hookrightarrow \partial V^{(0)}_{\sigma_0}$. We point out that for $\dim C > 1$ we may allow $C$ to be a CW complex such that all of its gluing maps are embeddings, and the definition extends in the obvious way although we do not pursue the details at the present time.
			\end{rmk}
		\begin{lma}\label{lma:unique_flow_line}
			Let $(C, \boldsymbol V)$ be a simplicial decomposition. For each $k$-face $\sigma_k\in C$ the Liouville vector field $Z$ in $\# \boldsymbol V$ has only one critical point $(v_k, \boldsymbol 0, \boldsymbol 0) \in V^{(k)}_{\sigma_k} \times D_{\varepsilon_k}T^\ast \varDelta^k$ of Morse index $\ind_{\phi_{\sigma_k}}(v_k) + k$ for each $v_k\in \crit(\phi_{\sigma_k})$.

			Furthermore, for each inclusion of faces $\sigma_{k} \subset \sigma_{k+1}$ and each pair of critical points $v_{k}\in \crit(\phi_{\sigma_{k}})$ and $v_{k+1}\in \crit(\phi_{\sigma_{k+1}})$ there is a unique flow line of the Liouville vector field $Z$ connecting the critical point $(v_{k+1}, \boldsymbol 0, \boldsymbol 0) \in V^{(k+1)}_{\sigma_{k+1}} \times D_{\varepsilon_{k+1}}T^\ast \varDelta^{k+1}$ of $Z$ with the critical point $(v_k, \boldsymbol 0, \boldsymbol 0) \in V^{(k)}_{\sigma_{k}} \times D_{\varepsilon_k}T^\ast \varDelta^k$ of $Z$. The image of such a flow line under the projection $V^{(k)}_{\sigma_{k}} \times D_{\varepsilon_k}T^\ast \varDelta^k \longrightarrow D_{\varepsilon_k}T^\ast \varDelta^k$ belongs to the unique straight line in $C$ connecting the center of $\sigma_{k+1}$ with the center of $\sigma_k$.
		\end{lma}
		\begin{proof}
			It follows from the explicit expression the Liouville vector field in each building block, \eqref{eq:basic_building_block_Liouville}.
		\end{proof}
		\begin{rmk}
			The slogan of \cref{lma:unique_flow_line} is that the Liouville vector field in $\# \boldsymbol V$ restricted to the basic building block corresponding to the $k$-face $\sigma_k$ is the gradient of a Morse function on $\varDelta^k$ with one critical point of index $k$, see \cref{fig:liouville_simplex}.
			\begin{figure}[!htb]
				\centering
				\includegraphics{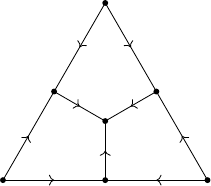}
				\caption{The Liouville vector field in $\# \boldsymbol V$ restricted to a $2$-face.}\label{fig:liouville_simplex}
			\end{figure}
		\end{rmk}
		\subsection{Legendrian attaching spheres}\label{sec:leg_attaching_data_for_simplex_handles}
			For a simplicial decomposition $(C, \boldsymbol V)$ we now describe the top Legendrian attaching spheres.
			
			Assume that for each Weinstein $(2n-2k)$-manifold $V^{(k)}_{\sigma_k}\in \boldsymbol V$ we have chosen a Weinstein handle decomposition $h_{\sigma_k}$, which encodes both the Weinstein handles and the attaching maps. Let $\boldsymbol h := \bigcup_{\substack{\sigma_k\in C\\ 0 \leq k \leq m}} h_{\sigma_k}$. The goal of this section is to describe how $\boldsymbol h$ induces a handle decomposition of $\# \boldsymbol V$.

			Throughout this section and the rest of the paper we denote by $V^{(k)}_{\sigma_k,0}$ the subcritical part of $V^{(k)}_{\sigma_k}$. In a similar fashion we let $\boldsymbol V_0$ denote the set obtained from $\boldsymbol V$ by replacing every $V^{(k)}_{\sigma_k}\in \boldsymbol V$ with its subcritical part $V^{(k)}_{\sigma_k,0}$ and keeping the same Weinstein hypersurfaces. Let $\ell_{\sigma_k} = \bigcup_{j} \ell_{\sigma_k,j}$ denote the union of the core disks of the top handles in $h_{\sigma_k}$. Similarly we let $\partial \ell_{\sigma_k} = \bigcup_{j} \partial \ell_{\sigma_k,j} \subset \partial V^{(k)}_{\sigma_k,0}$ be the union of the attaching $(n-k-1)$-spheres for the top handles in $h_{\sigma_k}$.

			\begin{dfn}
				Let $(C, \boldsymbol V)$ be a simplicial decomposition of $X$. Using the induced gluing maps $\overline \iota_{\sigma_k}$ on basic building blocks as in \cref{dfn:join_general_m} we define
				\begin{align*}
					E_j(\ell_{\sigma_k}) &:= \bigcup_{\sigma_j \subset \sigma_k} \overline \iota_{\sigma_j}((\ell_{\sigma_k} \times \R^{k-j-1}) \times \R^{j}), \quad E(\ell_{\sigma_k}) := \bigcup_{j=0}^{k-1} E_j(\ell_{\sigma_k})\\
					E_j(\partial \ell_{\sigma_k}) &:= \bigcup_{\sigma_j \subset \sigma_k}\overline \iota_{\sigma_j}((\partial \ell_{\sigma_k} \times \R^{k-j-1}) \times \R^{j}), \quad E(\partial \ell_{\sigma_k}) := \bigcup_{j=0}^{k-1} E_j(\partial \ell_{\sigma_k})\, ,
				\end{align*}
				where we regard $\ell_{\sigma_k} \times \R^{k-j-1}$ and $\partial \ell_{\sigma_k} \times \R^{k-j-1}$ as subsets of $V^{(k)}_{\sigma_k} \times \R^{2(k-j-1)} \subset \# \boldsymbol V_{\supsetneq \sigma_j}$ for each $\sigma_j \subset \sigma_k$ as in \eqref{eq:dl_trivial_extension}, see \cref{fig:extending_core_disks}.
			\end{dfn}

				\begin{figure}[!htb]
					\centering
					\includegraphics[scale=1.2]{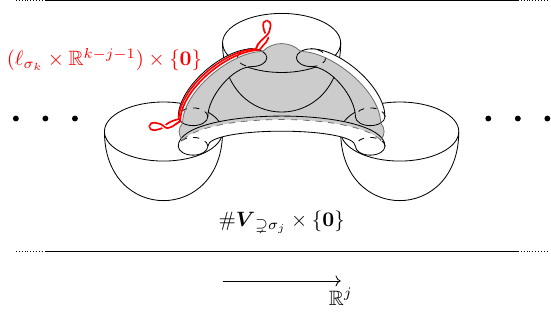}
					\caption{The figure shows the Weinstein hypersurface $\left(\# \boldsymbol V_{\supsetneq \sigma_j}\right) \times \R^{2j} \hookrightarrow \partial_+(V^{(j)}_{\sigma_j} \times \R^{2j})$ in the slice $\left\{\boldsymbol 0\right\} \in \R^{2j}$. The Legendrian submanifold $(\ell_{\sigma_k} \times \R^{k-j-1}) \times \R^j \subset \partial_+(V^{(j)}_{\sigma_j} \times \R^{2j})$ in the slice $\left\{\boldsymbol 0\right\} \in \R^{2j}$ is depicted in red.} \label{fig:extending_core_disks}
				\end{figure}
			\begin{dfn}[Attaching spheres for simplicial handles]\label{dfn:attaching_spheres_simpl_handles}
				Let $(C, \boldsymbol V)$ be a simplicial decomposition of $X$. Define $\varSigma_{\sigma_k}(\boldsymbol h) := \overline{\partial \ell_{\sigma_k}} \sqcup_{E(\partial \ell_{\sigma_k})} E(\ell_{\sigma_k})$ where $\overline{\partial \ell_{\sigma_k}} \subset V^{(k)}_{\sigma_k,0} \times \R^{2k}$ is defined as in \eqref{eq:dl_trivial_extension}. We also define 
				\begin{equation}\label{eq:attach_spheres}
					\varSigma(\boldsymbol h) := \bigcup_{\substack{\sigma_k\in C \\ 0 \leq k \leq m}} \varSigma_{\sigma_k}(\boldsymbol h), \quad \varSigma_{\supset \sigma_k}(\boldsymbol h) := \bigcup_{\substack{\sigma_i \supset \sigma_k \\ k \leq i \leq m}}\varSigma_{\sigma_i}(\boldsymbol h)\, .
				\end{equation}
			\end{dfn}
			\begin{lma}\label{lma:attaching_spheres_simpl_handle}
				We have that $\varSigma_{\sigma_k}(\boldsymbol h)$ is a union of Legendrian $(n-1)$-spheres for each $\sigma_k\in C$ where $k \in \left\{0,\ldots,m\right\}$. Moreover $\varSigma(\boldsymbol h)$ is the union of attaching spheres for $H^m_{\boldsymbol \varepsilon}(\boldsymbol V)$ in $\partial X_0$.
			\end{lma}
			\begin{proof}
				For the first statement, gluing together each $\overline \iota_{\sigma_j}((\ell_{\sigma_k} \times \R^{k-j-1})\times \R^j) \cong \ell_{\sigma_k} \times D^{k-1}$ over $\partial \sigma_k\cong S^{k-1}$ gives $E(\ell_{\sigma_k}) \cong \ell_{\sigma_k} \times S^{k-1}$. Similarly $E(\partial \ell_{\sigma_k}) \cong \partial \ell_{\sigma_k} \times S^{k-1}$. Then because $\overline{\partial \ell_{\sigma_k}} \cong \partial \ell_{\sigma_k} \times D^k$, we have
				\[
					\varSigma_{\sigma_k}(\boldsymbol h) \cong (S^{n-k-1} \times D^k) \sqcup_{S^{n-k-1} \times S^{k-1}} (D^{n-k} \times S^{k-1}) \cong S^{n-1}\, .
				\]
				For the second statement start with $m = 1$ and consider $X = X_1 \#_V X_2$. Then we consider an exhausting Morse function $\phi \colon V \longrightarrow [0,\infty)$ with one non-degenerate minimum of value $0$, such that $V_0 = \phi^{-1}([0,\frac 12 \varepsilon])$ and $\phi^{-1}(\frac 32 \varepsilon)$ is the set of critical point of critical index for $V$. The exhausting Morse function in the simplicial handle $V \times \R^{2k}$ is $\psi(v, \boldsymbol x, \boldsymbol y) := \phi(v) + \sum_{i=1}^k \left(x_i^2 - \frac 12 y_i^2\right)$, see \cref{sec:basic_building_blocks}, and we have that $\psi^{-1}(\varepsilon) = G^k_\varepsilon(V_0)$ is the contact boundary of the simplicial handle, which contains the attaching spheres which are $\varSigma(\boldsymbol h)$. The critical points in the core disk of the top handles of the simplicial handle are contained in $\psi^{-1}(\frac 32 \varepsilon)$ and sublevel sets $\psi^{-1}([0,\rho))$ for $\rho > \frac 32 \varepsilon$ give $V \times \R^{2k}$. For $m > 1$ it follows from the gluing procedure in \cref{sec:gluing_of_basic_building_blocks}.
			\end{proof}
			\begin{figure}[!htb]
				\centering
				\includegraphics[scale=0.7]{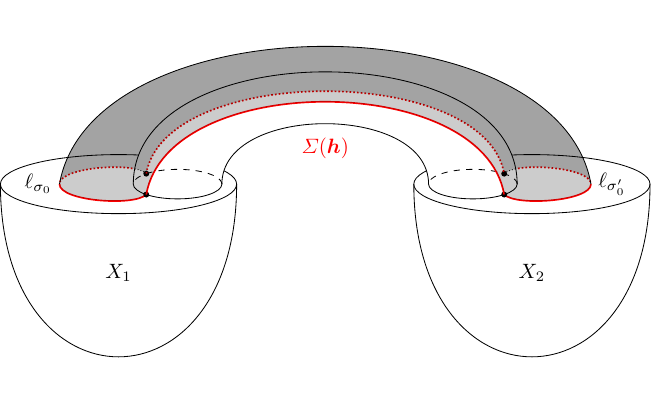}
				\includegraphics[scale=0.7]{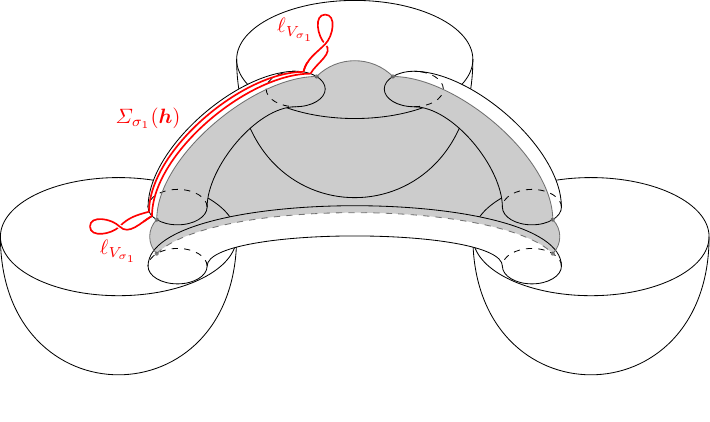}
				\caption{Left: Assuming that $X_1$ and $X_2$ are subcritical, $\varSigma(\boldsymbol h)$ is the Legendrian top attaching sphere of $X_1 \#_V X_2$. Right: The attaching sphere $\varSigma_{\sigma_1}(\boldsymbol h)$ associated to the $1$-face $\sigma_1$.}
			\end{figure}
			Let
			\[
				\pi_{\sigma_k} \colon V^{(k)}_{\sigma_k} \times \R^{2k} \longrightarrow \R^{2k}\, ,
			\]
			be the projection to the $\R^{2k}$-factor, and let $\eval[0]{\pi_{\sigma_k}}_{G^k_\varepsilon}$ be the restriction to the upper contact boundary, see \eqref{eq:building_blocks_contact_boundaries} for the definition of $G^k_\varepsilon$ in the basic building block $V^{(k)}_{\sigma_k} \times \R^{2k}$.

			\begin{dfn}\label{dfn:reeb_chords_k-face}
				\begin{enumerate}
					\item Let $\mathcal R(\sigma_k)$ denote the set of Reeb chords of $\partial \ell_{\sigma_k} \cup \bigcup_{\substack{\sigma_i \supset \sigma_{k} \\ k+1 \leq i \leq m}} (\ell_{\sigma_{i}} \times \R^{i-k-1}) \subset \partial V^{(k)}_{\sigma_k,0}$ for $k \in \left\{0,\ldots,m\right\}$.
					\item Let $\mathcal R(\sigma_k, \mathfrak a)$ denote the set of Reeb chords in $\mathcal R(\sigma_k)$ of action $< \mathfrak a$.
				\end{enumerate}
			\end{dfn}
			\begin{rmk}
				In the definition above, note that $\ell_{\sigma_{i}} \times \R^{i-k-1}$ is considered as a Legendrian submanifold in $\partial V^{(k)}_{\sigma_k,0}$ via the maps $\iota_{\sigma_j}$ (and compositions of such). 
			\end{rmk}
			Let $p_{\sigma_k} = \boldsymbol 0 \in \R^{2k}$ be the unique critical point of the Liouville vector field $\sum_{i=1}^k (2x_i \partial_{x_i} - y_i \partial_{y_i})$ in the building block $V^{(k)}_{\sigma_k} \times \R^{2k}$ corresponding to $\sigma_k \in C_k$, see \cref{lma:unique_flow_line}. We note that for $k = 0$, the definition of the basic building block still makes sense: It is $V^{(0)}_{\sigma_0} \times \left\{0\right\}$.
			\begin{lma}\label{lma:Reeb_chords_one_to_one_correspondence}
				For all $\mathfrak a > 0$ there exists some $\delta > 0$ and an arbitrary small perturbation of $\boldsymbol V_0$ such that for all $0 < \boldsymbol \varepsilon < \delta$ the following holds. The Reeb chords of $\varSigma(\boldsymbol h) \subset \partial \# \boldsymbol V_0$ of action $< \mathfrak a$ are contained in the set
					\[
						\bigcup_{\substack{\sigma_k\in C_k\\ 0 \leq k \leq m}}(\eval[0]{\pi_{\sigma_k}}_{G^k_{\varepsilon_k}})^{-1}(p_{\sigma_k}) \subset \partial \# \boldsymbol V_0\, ,
					\]
					and they are in one-to-one grading preserving correspondence with Reeb chords in the set $\bigcup_{\substack{\sigma_k\in C_k \\ 0 \leq k \leq m}} \mathcal R(\sigma_k, \mathfrak a)$.
			\end{lma}
			\begin{proof}
				First it follows from \cref{sec:basic_building_blocks} and \cref{lma:basic_building_block_reeb} that all Reeb chords are contained in $\bigcup_{\substack{\sigma_k\in C_k \\ 0 \leq k \leq m}}(\eval[0]{\pi_{\sigma_k}}_{G^k_\varepsilon})^{-1}(p_{\sigma_k}) \subset \partial \# \boldsymbol V_0$. Next, after an arbitrary small perturbation of each $V^{(k+1)}_{\sigma_{k+1},0} \subset \partial V^{(k)}_{\sigma_k,0}$, we claim that we can shrink the parameter $\boldsymbol \varepsilon$ (ie assume $\boldsymbol \varepsilon < c$ for $c$ small enough) used in the definition of $\# \boldsymbol V_0$ a sufficient amount so that no Reeb chords of $\varSigma(\boldsymbol h)$ starting in $(\eval[0]{\pi_{\sigma_k}}_{G^k_{\varepsilon_k}})^{-1}(p_{\sigma_k})$ of action $< \mathfrak a$ enters any $V^{(k+1)}_{\sigma_{k+1},0}$. It follows from a dimension argument as follows (cf \cite[Lemma 3.4]{asplund2021chekanov}). Reeb chords of $\varSigma(\boldsymbol h)$ starting in $(\eval[0]{\pi_{\sigma_k}}_{G^k_{\varepsilon_k}})^{-1}(p_{\sigma_k})$ of action $< \mathfrak a$ are in one-to-one correspondence with Reeb chord of 
				\[
					\partial \ell_{\sigma_k} \cup \bigcup_{\substack{\sigma_{i} \supset \sigma_{k} \\ k+1 \leq i \leq m}} (\ell_{\sigma_{i}} \times \R^{i-k-1}) \subset \partial V^{(k)}_{\sigma_k,0}
				\]
				of action $< \mathfrak a$ by \cref{lma:basic_building_block_reeb}. Now, $\dim \partial \ell_{\sigma_k} = n-k-1$ and $V^{(k+1)}_{\sigma_{k+1},0} \subset \partial V^{(k)}_{\sigma_k,0}$ is viewed as a subcritical Weinstein neighborhood of its $(n-k-2)$-dimensional skeleton. Thus, generically, no such Reeb chord meets the skeleton of $V^{(k+1)}_{\sigma_{k+1},0}$, and given that the action is bounded from above by $\mathfrak a$, we can shrink the size of $\boldsymbol \varepsilon$ thus shrink the size of the neighborhood of the skeleton of $V^{(k+1)}_{\sigma_{k+1},0}$ so that no Reeb chord meets this neighborhood of the skeleton.
			\end{proof}
		\subsection{Almost complex structure and holomorphic disks}\label{sec:acs_and_hol_curves_in_simplicial_handles}
		Let $(C, \boldsymbol V)$ be a simplicial decomposition. In this section we consider almost complex structures and holomorphic disks the symplectization of $\# \boldsymbol V$.

		Let $J_{V^{(k)}_{\sigma_k}}$ be a \emph{handle adapted} almost complex structure on $V^{(k)}_{\sigma_k}$, see the beginning of \cref{sec:acs_hol_curves_basic_building_block} for the definition. We fix an exact symplectomorphism between $\R \times \partial \# \boldsymbol V$ and the complement of the skeleton in $\# \boldsymbol V$, which we denote by $\widetilde{\# \boldsymbol V}$. For a $k$-face $\sigma_k\in C$, let $J_{\sigma_k}$ be the cylindrical almost complex structure on $\widetilde H^k_{\varepsilon_k}(V^{(k)}_{\sigma_k})$ defined in \cref{sec:acs_hol_curves_basic_building_block} that takes the Liouville vector field to the Reeb vector field.

		By construction any point in $\# \boldsymbol V$ lies in $V^{(k)}_{\sigma_k} \times \R^{2k}$ for some $k\in \left\{0,\ldots,m\right\}$ and some $\sigma_k\in C$. By using cut-off functions supported in neighborhoods of basic building blocks we obtain a smooth almost complex structure $J(\boldsymbol V)$ in $\widetilde{\# \boldsymbol V}$, that agrees with $J_{\sigma_k}$ over the locus $V^{(k)}_{\sigma_k} \times \R^{2k}$ in $\widetilde{\# \boldsymbol V}$, for $k\in \left\{0,\ldots,m\right\}$ and $\sigma_k\in C_k$. Let $\pi_{\sigma_k,i} \colon V^{(k)}_{\sigma_k} \times \R^{2k} \longrightarrow \R^{2}_{(x_i,y_i)}$ be the projection to the $(x_i,y_i)$-plane. We now have a version of \cref{lma:complexproj} for $\# \boldsymbol V$.
		\begin{lma}\label{lma:complexproj_simplex_handle}
			Let $p\in V^{(k)}_{\sigma_k} \times \R^{2k}$ for some $k\in \left\{0,\ldots, m\right\}$ and $\sigma_k\in C$. The codimension $2$ hyperplanes $V_{\sigma_k,(x_i,y_i)} := \pi^{-1}_{\sigma_k,i}(x_i,y_i)$ are $J(\boldsymbol V)$-complex. In particular, any holomorphic disk in the symplectization of $G^k_{\varepsilon_{k}}(V^{(k)}_{\sigma_k})$ intersecting $V_{\sigma_k,(x_i,y_i)}$ is either contained in $V_{\sigma_k,(x_i,y_i)}$ or intersects $V_{\sigma_k,(x_i,y_i)}$ positively. Moreover the intersection number is locally constant in $(x_i,y_i)$.
		\end{lma}
		\begin{proof}
			Immediate consequence of \cref{lma:complexproj} and the construction of $\# \boldsymbol V$ in \cref{dfn:join_over_simplicial_complex}.
		\end{proof}
		We now consider holomorphic disks in $\R \times \partial \# \boldsymbol V_0$ with boundary on $\R \times \varSigma(\boldsymbol h)$ contributing to the differential in the Chekanov--Eliashberg dg-algebra of $\varSigma(\boldsymbol h)$, generalizing \cref{lma:middlesubcrit} and \cref{cor:dgsubalg}.

		Let $p_{\sigma_k} = \boldsymbol 0 \in \R^{2k}$ be the unique critical point of the Liouville vector field $\sum_{i=1}^k (2x_i \partial_{x_i} - y_i \partial_{y_i})$ in the building block corresponding to the $k$-face $\sigma_k\in C$ as in \cref{lma:unique_flow_line}.
		\begin{lma}\label{lma:hol_curves_in_simplex_handle}
			For any $\mathfrak a > 0$ there is some $\delta > 0$ so that for all $0 < \boldsymbol \varepsilon < \delta$ we have that the following holds. Any $J(\boldsymbol V)$-holomorphic disk in the symplectization $\R \times \partial \# \boldsymbol V_0$ with boundary on $\R \times \varSigma(\boldsymbol h)$ and with positive puncture at a Reeb chord over $p_{\sigma_k}$ of action $< \mathfrak a$ either lies entirely over $p_{\sigma_k}$ or is contained over the stable manifold of $p_{\sigma_\ell}$ for some $\sigma_\ell \supset \sigma_k$.
		\end{lma}
		\begin{proof}
			The first part follows from \cref{lma:middlesubcrit}. If such a holomorphic disk does not lie entirely over $p_{\sigma_k}$ it is allowed to have some additional negative punctures at Reeb chords over $p_{\sigma_{\ell}}$ for $\sigma_{\ell} \supset \sigma_k$. In this case the holomorphic disk lies over the stable manifold of $p_{\sigma_\ell}$ of the Liouville vector field through the projection $V^{(\ell)}_{\sigma_\ell} \times \R^{2 \ell} \longrightarrow \R^{2 \ell} \cap \left\{\boldsymbol x = \boldsymbol 0\right\}$.
		\end{proof}
		\begin{dfn}
			We say that a $J(\boldsymbol V)$-holomorphic disk in $\R \times \partial \# \boldsymbol V_0$ with boundary on $\R \times \varSigma(\boldsymbol h)$ is \emph{simplicial} if it has a positive puncture at a Reeb chord of $\varSigma(\boldsymbol h)$ in $\partial_+(V^{(k)}_{\sigma_k,0} \times \R^{2k})$ and any negative puncture of the $J(\boldsymbol V)$-holomorphic disk is at a Reeb chords of $\varSigma(\boldsymbol h)$ in $\partial_+(V^{(\ell)}_{\sigma_\ell,0} \times \R^{2 \ell})$ for some $\sigma_\ell \supset \sigma_k$.
		\end{dfn}
		\begin{cor}\label{cor:no_curve_crosses_handle}
			For any $\mathfrak a > 0$ there is some $\delta > 0$ so that for all $0 < \boldsymbol \varepsilon < \delta$ we have that the following holds. Any $J(\boldsymbol V)$-holomorphic disk in $\R \times \partial \# \boldsymbol V_0$ with boundary on $\R \times \varSigma(\boldsymbol h)$ with positive puncture at a Reeb chord of $\varSigma(\boldsymbol h)$ of action $< \mathfrak a$ is simplicial.
		\end{cor}
		\begin{proof}
			Suppose that there exists a $J(\boldsymbol V)$-holomorphic disk that is not simplicial. In $\widetilde{\# \boldsymbol V_0}$ such a curve would intersect the hypersurface $H_{\sigma_k,i} := \pi^{-1}_{\sigma_k,i}\left\{y_i = 0\right\} \subset V^{(k)}_{\sigma_k,0} \times \R^{2k}$ for some $i\in \left\{1,\ldots,k\right\}$ and some $\sigma_k\in C$. The codimension one hypersurface $H_{\sigma_k,i}$ is foliated by the codimension two $J(\boldsymbol V)$-complex submanifolds $V_{(x_i,y_i)}$. So by \cref{lma:complexproj_simplex_handle} the $J(\boldsymbol V)$-holomorphic disk is either contained in one of the leaves $V_{(x_i,0)}$ or it must intersect every leaf over $H_{\sigma_k,i}$ non-trivially. Either case gives a contradiction. In the first case, if the $J(\boldsymbol V)$-holomorphic disk is contained in one of the leaves $V_{(x_i,0)}$, its image in the projection $V^{(k)}_{\sigma_k} \times \R^{2k} \longrightarrow \R^{2k} \cap \left\{\boldsymbol x = \boldsymbol 0\right\}$ must either be contained in a flow line of the Liouville vector field on $H^m_{\boldsymbol \varepsilon}(\boldsymbol V_0)$ or contained over a critical point $p_{\sigma_k} = \boldsymbol 0 \in \R^{2k}$, which means that such a $J(\boldsymbol V)$-holomorphic curve must be simplicial by \cref{lma:unique_flow_line}. If it is contained over a critical point $p_{\sigma_k} = \boldsymbol 0 \in \R^{2k}$, it corresponds to a holomorphic disk in $\R \times \partial V^{(k)}_{\sigma_k,0}$ with boundary on $\R \times \left(\partial \ell_{\sigma_k} \cup \bigcup_{\substack{\sigma_{i} \supset \sigma_{k} \\ k+1 \leq i \leq m}} (\ell_{\sigma_{i}} \times \R^{i-k-1})\right)$ by \cref{lma:middlesubcrit} and can thus only have negative punctures at Reeb chords of $\varSigma(\boldsymbol h)$ in $\partial_+(V^{(k)}_{\sigma_k,0} \times \R^{2k})$, and hence it must again be simplicial.

			In the second case that the $J(\boldsymbol V)$-holomorphic disk intersects every leaf of the foliation over $H_{\sigma_k,i}$ non-trivially, it must be that the $J(\boldsymbol V)$-holomorphic disk must have a positive asymptotic at a Reeb chord of $\varSigma(\boldsymbol h)$ which does not lie over the origin in $V^{(k)}_{\sigma_k,0} \times \R^{2k}$ which yields a contradiction by \cref{lma:basic_building_block_reeb} (cf \cite[Lemma 3.5]{asplund2021chekanov}).
		\end{proof}
	\subsection{Simplicial descent for Chekanov--Eliashberg dg-algebras}\label{sec:simplical_descent}
		Let $(C, \boldsymbol V)$ be a simplicial decomposition of $X$. Let $\varSigma(\boldsymbol h)$ denote the union of the Legendrian attaching $(n-1)$-spheres of $X$ in $\partial X_0$, see \cref{dfn:attaching_spheres_simpl_handles}. In this section we prove the main result of the paper, namely that $CE^\ast(\varSigma(\boldsymbol h); X_0)$ is quasi-isomorphic to a colimit of a diagram of dg-subalgebras.

		Recall from \cref{dfn:join_over_simplicial_complex} that $\boldsymbol \varepsilon$ denotes a sequence of real numbers controlling the size of the various simplicial handles used to construct $X \cong \# \boldsymbol V$. We use the notation $X^{\boldsymbol \varepsilon}$ to make this dependence explicit. Let $\varSigma(\boldsymbol h, \boldsymbol \varepsilon) \subset X_0^{\boldsymbol \varepsilon}$ be the corresponding union of Legendrian attaching spheres.

		Let $\mathbb F$ be a field and let $\boldsymbol k := \bigoplus_{i \in \pi_0(\varSigma(\boldsymbol h))} \mathbb F e_i$ where $\left\{e_i\right\}_{i\in \pi_0(\varSigma(\boldsymbol h))}$ is a set of mutually orthogonal idempotents. For any $\sigma_k\in C$ define $\boldsymbol k_{\sigma_k} := \bigoplus_{i\in \pi_0(\varSigma_{\supset \sigma_k}(\boldsymbol h))} \mathbb F e_i$. We note that $\left\{e_i\right\}_{i\in \pi_0(\varSigma_{\supset \sigma_k}(\boldsymbol h))} \subset \left\{e_i\right\}_{i\in \pi_0(\varSigma(\boldsymbol h))}$ since we have $\pi_0(\varSigma_{\supset \sigma_k}(\boldsymbol h)) \subset \pi_0(\varSigma(\boldsymbol h))$ by construction, see \eqref{eq:attach_spheres}. For any face inclusion $\sigma_\ell \subset \sigma_k$ we have a non-unital ring morphism $\boldsymbol k_{\sigma_k} \longrightarrow \boldsymbol k_{\sigma_\ell}$, $e_i \mapsto e_i$. 

		Let $\mathbf{dga}$ denote the category of associative, non-commutative, non-unital dg-algebras over varying non-unital rings. Concretely, the category $\mathbf{dga}$ consists of the following.
		\begin{description}
			\item[Objects] Pairs $(A,R)$ where $R$ is a non-unital ring and where $A$ is an associative, non-commutative, non-unital dg-algebras over $R$.
			\item[Morphisms] A morphism $(A,R) \longrightarrow (B,S)$ consists of a non-unital ring morphism $\varphi \colon R \longrightarrow S$ and a non-unital morphism $f \colon A \longrightarrow \varphi^\ast B$ of dg-algebras over $R$. The pullback $\varphi^\ast B$ is given by restriction of scalars.
		\end{description}
		More abstractly, the projection functor $(A,R) \longmapsto R$ is both a fibration and a opfibration in the sense of Gray (cf \cite[Example 4.7]{gray1966fibred}). The fiber over $R$ is the category $\mathbf{dga}(R)$ of associative, non-commutative, non-unital dg-algebras over $R$. Since the base and each fiber has colimits, it follows from \cite[Corollary 4.3]{gray1966fibred} that $\mathbf{dga}$ has colimits. Since all dg-algebras appearing in this paper are semi-free, the colimit is again a semi-free dg-algebra generated by the union of the generators.

		\begin{lma}\label{lma:diagram_of_sub_algebras}
			For any $\mathfrak a > 0$ there is some $\delta > 0$ so that for all $0 < \boldsymbol \varepsilon < \delta$ we have that the following holds.
			\begin{enumerate}
				\item For each $k$-face $\sigma_k\in C$ there is a semi-free dg-algebra $\mathcal A_{\sigma_k}(\boldsymbol h; \boldsymbol \varepsilon, \mathfrak a)$ over $\boldsymbol k_{\sigma_k}$ with generating set $\bigcup_{\sigma_i \supset \sigma_k} \mathcal R(\sigma_i, \mathfrak a)$ and differential coinciding with the differential in $CE^\ast(\varSigma(\boldsymbol h,\boldsymbol \varepsilon);X_0^{\boldsymbol \varepsilon})$. In particular there is an inclusion of dg-algebras over $\boldsymbol k$ induced by the inclusion on the set of generators
				\[
					\mathcal A_{\sigma_k}(\boldsymbol h; \boldsymbol \varepsilon, \mathfrak a) \subset CE^\ast(\varSigma(\boldsymbol h,\boldsymbol \varepsilon);X_0^{\boldsymbol \varepsilon})\, ,
				\]
				where $\mathcal A_{\sigma_k}(\boldsymbol h; \boldsymbol \varepsilon, \mathfrak a)$ is viewed as a dg-algebra over $\boldsymbol k$ by restriction of scalars.
				\item For each inclusion of faces $\sigma_k \subset \sigma_{k+1}$ there is an inclusion of dg-algebras over $\boldsymbol k_{\sigma_k}$ induced by the inclusion on the set of generators
				\[
					\mathcal A_{\sigma_{k+1}}(\boldsymbol h; \boldsymbol \varepsilon, \mathfrak a) \subset \mathcal A_{\sigma_k}(\boldsymbol h; \boldsymbol \varepsilon, \mathfrak a)\, ,
				\]
				where $\mathcal A_{\sigma_{k+1}}(\boldsymbol h; \boldsymbol \varepsilon, \mathfrak a)$ is viewed as a dg-algebra over $\boldsymbol k_{\sigma_k}$ by restriction of scalars.
			\end{enumerate}
		\end{lma}

		\begin{proof}
			\begin{enumerate}
				\item The differential of $\mathcal A_{\sigma_k}(\boldsymbol h; \boldsymbol \varepsilon, \mathfrak a)$ coincides with the differential of $CE^\ast(\varSigma(\boldsymbol h, \boldsymbol \varepsilon); X^{\boldsymbol \varepsilon}_0)$ by definition, so it suffices to prove that there is an inclusion on the set of generators. This follows from item (2) of \cref{lma:Reeb_chords_one_to_one_correspondence} because
				\[
					\bigcup_{\sigma_i \supset \sigma_k} \mathcal R(\sigma_i, \mathfrak a) \subset \bigcup_{\substack{\sigma_k \in C_k \\ 0 \leq k \leq m}} \mathcal R(\sigma_k, \mathfrak a)\, .
				\]
				\item Again, as above, it follows because we have an inclusion on the set of generators
				\[
					\bigcup_{\sigma_i \supset \sigma_{k+1}} \mathcal R(\sigma_i, \mathfrak a) \subset \bigcup_{\sigma_i \supset \sigma_k} \mathcal R(\sigma_i, \mathfrak a)\, .
				\]
			\end{enumerate}
		\end{proof}
		Let $\mathfrak a > 0$ and let $0 < \boldsymbol \varepsilon < \delta$ where $\delta > 0$ is chosen so that \cref{lma:diagram_of_sub_algebras} holds. Let $\left\{\mathcal A_{\sigma_k}(\boldsymbol h; \boldsymbol \varepsilon, \mathfrak a)\right\}_k$ denote the commutative diagram with objects $\mathcal A_{\sigma_k}(\boldsymbol h; \boldsymbol \varepsilon, \mathfrak a)$ for $\sigma_k\in C_k$, $k\in \left\{0,\ldots,m\right\}$ and arrows being the inclusions in \cref{lma:diagram_of_sub_algebras}.
		\begin{lma}\label{lma:cohomology_indep_of_size}
			For any $\boldsymbol \varepsilon > \boldsymbol \varepsilon'$ we have a canonical quasi-isomorphism of dg-algebras over $\boldsymbol k$
			\[
				\varPsi \colon CE^\ast(\varSigma(\boldsymbol h, \boldsymbol \varepsilon); X_0^{\boldsymbol \varepsilon}) \longrightarrow CE^\ast(\varSigma(\boldsymbol h, \boldsymbol \varepsilon'); X_0^{\boldsymbol \varepsilon'})\, .
			\]
		\end{lma}
		\begin{proof}
			This is the usual invariance argument for the Chekanov--Eliashberg dg-algebra, and we will not repeat the full argument here. It is carried out exactly like in \cite[Section 5.5]{ekholm2017symplectic} \cite[Lemma 5.10 and Appendix A]{ekholm2015legendrian}.

			An overview of the argument is that the map $\varPsi$ is the cobordism map induced by the Weinstein cobordism obtained by deforming $X_0^{\boldsymbol \varepsilon}$ to $X_0^{\boldsymbol \varepsilon'}$.  One proves that $\varPsi$ is in fact a chain homotopy equivalence by constructing the inverse cobordism. The cobordism map of the inverse cobordism is then shown to be a homotopy inverse to $\varPsi$, since the composition of the cobordisms is (up to deformation) a trivial Weinstein cobordism. 
		\end{proof}
		\begin{lma}\label{lma:action_window_limits}
			Let $\left\{\mathfrak a_\ell\right\}_{\ell=1}^\infty$ and $\left\{\delta_\ell\right\}_{\ell=1}^\infty$ be two sequences of positive real numbers such that $\mathfrak a_\ell \to \infty$, $\delta_\ell \to 0$ and so that \cref{lma:diagram_of_sub_algebras} holds for the tuples $(\boldsymbol \varepsilon_\ell, \mathfrak a_\ell)$ where $0 < \boldsymbol \varepsilon_\ell < \delta_\ell$. Then we have a diagram
			\begin{equation}\label{eq:coupled_diagram}
				\begin{tikzcd}[row sep=scriptsize]
					\cdots \rar{\left\{\varphi_{\ell-1,k}\right\}_k} & \left\{\mathcal A_{\sigma_{k}}(\boldsymbol h; \boldsymbol \varepsilon_{\ell}, \mathfrak a_\ell)\right\}_k \rar{\left\{\varphi_{\ell,k}\right\}_k} & \left\{\mathcal A_{\sigma_{k}}(\boldsymbol h; \boldsymbol \varepsilon_{\ell+1}, \mathfrak a_{\ell+1})\right\}_k \rar{\left\{\varphi_{\ell+1,k}\right\}_k} & \cdots
				\end{tikzcd}
			\end{equation}
			where each $\varphi_{\ell,k}$ is a canonical morphism of dg-algebras over $\boldsymbol k_{\sigma_k}$ such that for each inclusion of faces $\sigma_k \subset \sigma_{k+1}$ the following square commutes.
			\[
				\begin{tikzcd}[row sep=scriptsize, column sep=scriptsize]
					\mathcal A_{\sigma_{k+1}}(\boldsymbol h; \boldsymbol \varepsilon_{\ell}, \mathfrak a_\ell) \rar{\varphi_{\ell,k+1}} \dar[phantom,description][rotate=90]{\subset} & \mathcal A_{\sigma_{k+1}}(\boldsymbol h; \boldsymbol \varepsilon_{\ell+1}, \mathfrak a_{\ell+1}) \dar[phantom,description][rotate=90]{\subset}\\
					\mathcal A_{\sigma_{k}}(\boldsymbol h; \boldsymbol \varepsilon_{\ell}, \mathfrak a_\ell) \rar{\varphi_{\ell,k}} & \mathcal A_{\sigma_{k}}(\boldsymbol h; \boldsymbol \varepsilon_{\ell+1}, \mathfrak a_{\ell+1})
				\end{tikzcd}
			\]
		\end{lma}
		\begin{proof}
			The maps $\varphi_\ell$ are defined by studying the following diagram.
			\begin{equation}\label{eq:diagram_action_size_of_handles}
				\begin{tikzcd}[row sep=scriptsize, column sep=scriptsize]
					& \vdots  \dar[phantom,description][rotate=-90]{\subset} & \vdots  \dar[phantom,description][rotate=-90]{\subset} & \\
					\cdots \rar{\left\{\psi_{\ell-1,k}\right\}_k} & \left\{\mathcal A_{\sigma_{k}}(\boldsymbol h; \boldsymbol \varepsilon_{\ell}, \mathfrak a_\ell)\right\}_k \rar{\left\{\psi_{\ell,k}\right\}_k} \drar[dashed]{\left\{\varphi_{\ell,k}\right\}_k} \dar[phantom,description][rotate=-90]{\subset} & \left\{\mathcal A_{\sigma_{k}}(\boldsymbol h; \boldsymbol \varepsilon_{\ell+1}, \mathfrak a_{\ell})\right\}_k \dar[phantom,description][rotate=-90]{\subset} \rar{\left\{\psi_{\ell+1,k}\right\}_k} & \cdots\\
					\cdots \rar{\left\{\psi_{\ell-1,k}\right\}_k} & \left\{\mathcal A_{\sigma_{k}}(\boldsymbol h; \boldsymbol \varepsilon_{\ell}, \mathfrak a_{\ell+1})\right\}_k \rar{\left\{\psi_{\ell,k}\right\}_k} \dar[phantom,description][rotate=-90]{\subset} & \left\{\mathcal A_{\sigma_{k}}(\boldsymbol h; \boldsymbol \varepsilon_{\ell+1}, \mathfrak a_{\ell+1})\right\}_k \dar[phantom,description][rotate=-90]{\subset} \rar{\left\{\psi_{\ell+1,k}\right\}_k} & \cdots \\
					& \vdots & \vdots &
				\end{tikzcd}
			\end{equation}
			In each horizontal arrow $\psi_{\ell,k}$ is the restriction of the quasi-isomorphism $\varPsi$ in \cref{lma:cohomology_indep_of_size} to the dg-subalgebra $\mathcal A_{\sigma_k}(\boldsymbol h; \boldsymbol \varepsilon_\ell, \mathfrak a_\ell)$ for each $k$. In particular it follows that for each inclusion of faces $\sigma_k \subset \sigma_{k+1}$ the following square commutes.
			\[
				\begin{tikzcd}[row sep=scriptsize, column sep=scriptsize]
					\mathcal A_{\sigma_{k+1}}(\boldsymbol h; \boldsymbol \varepsilon_{\ell}, \mathfrak a_\ell) \rar{\psi_{\ell,k+1}} \dar[phantom,description][rotate=90]{\subset} \ar[rr,bend left, looseness=0.3, "\varphi_{\ell,k+1}"] & \mathcal A_{\sigma_{k+1}}(\boldsymbol h; \boldsymbol \varepsilon_{\ell+1}, \mathfrak a_{\ell}) \dar[phantom,description][rotate=90]{\subset} \rar[phantom, description]{\subset} & A_{\sigma_{k+1}}(\boldsymbol h; \varepsilon_{\ell+1}, \mathfrak a_{\ell+1}) \dar[phantom,description][rotate=90]{\subset}\\
					\mathcal A_{\sigma_{k}}(\boldsymbol h; \boldsymbol \varepsilon_{\ell}, \mathfrak a_\ell) \rar{\psi_{\ell,k}} \ar[rr,bend right, looseness=0.3, swap, "\varphi_{\ell,k}"] & \mathcal A_{\sigma_{k}}(\boldsymbol h; \boldsymbol \varepsilon_{\ell+1}, \mathfrak a_{\ell}) \rar[phantom, description]{\subset} & A_{\sigma_{k+1}}(\boldsymbol h; \varepsilon_{\ell}, \mathfrak a_{\ell+1})
				\end{tikzcd}\, .
			\]
		\end{proof}
		Let $\mathcal A_{\sigma_k}(\boldsymbol h) := \colim_{\ell\to \infty}\mathcal A_{\sigma_k}(\boldsymbol h; \boldsymbol \varepsilon_\ell, \mathfrak a_\ell)$ be the colimit of the diagram \eqref{eq:coupled_diagram} for a fixed $\sigma_k\in C_k$. We also let $\mathcal A_{\sigma_k}(\boldsymbol h, \boldsymbol \varepsilon) := \colim_{\ell\to \infty} \mathcal A_{\sigma_k}(\boldsymbol h; \boldsymbol \varepsilon, \mathfrak a_\ell)$, be the colimit of the diagram of inclusions in wider action windows, where $\boldsymbol \varepsilon$ is independent of $\ell$. Both of these colimits are taken in the category $\mathbf{dga}(\boldsymbol k_{\sigma_k})$.
		\begin{cor}\label{cor:subalgebras_diagram_no_actions}
			For each inclusion of faces $\sigma_k \subset \sigma_{k+1}$ there is an inclusion of dg-algebras over $\boldsymbol k_{\sigma_k}$ induced by the inclusion on the set of generators
			\[
				\mathcal A_{\sigma_{k+1}}(\boldsymbol h) \subset \mathcal A_{\sigma_k}(\boldsymbol h)\, .
			\]
			Moreover, for each $\sigma_k \in C_k$ we have a canonical quasi-isomorphism of dg-algebras over $\boldsymbol k_{\sigma_k}$
			\[
				\mathcal A_{\sigma_k}(\boldsymbol h; \boldsymbol \varepsilon') \cong \mathcal A_{\sigma_k}(\boldsymbol h)\, .
			\]
		\end{cor}
		\begin{proof}
			The first statement follows immediately from \cref{lma:action_window_limits}, and the second statement follows from the diagram
			\[
				\begin{tikzcd}[row sep=scriptsize, column sep=scriptsize]
					\cdots \rar & \mathcal A_{\sigma_k}(\boldsymbol h; \boldsymbol \varepsilon_\ell, \mathfrak a_\ell) \rar \dar{\cong} & \mathcal A_{\sigma_k}(\boldsymbol h; \boldsymbol \varepsilon_{\ell+1}, \mathfrak a_{\ell+1}) \rar \dar{\cong} & \cdots \\
					\cdots \rar[phantom,description]{\subset}& \mathcal A_{\sigma_k}(\boldsymbol h; \boldsymbol \varepsilon', \mathfrak a_\ell) \rar[phantom, description]{\subset} & \mathcal A_{\sigma_k}(\boldsymbol h;\boldsymbol \varepsilon', \mathfrak a_{\ell+1}) \rar[phantom,description]{\subset}& \cdots 
				\end{tikzcd}\, ,
			\]
			where $(\boldsymbol \varepsilon_\ell, \mathfrak a_\ell)$ satisfies the conditions of \cref{lma:diagram_of_sub_algebras} for each $\ell\in \Z_{\geq 1}$. The vertical arrows are restrictions of the quasi-isomorphisms $\varPsi$ in \cref{lma:cohomology_indep_of_size} to the dg-subalgebras $\mathcal A_{\sigma_k}(\boldsymbol h; \boldsymbol \varepsilon, \mathfrak a)$.
		\end{proof}
		\begin{thm}\label{thm:ce_descent_first}
			There is a quasi-isomorphism of dg-algebras over $\boldsymbol k$
			\[
				CE^\ast(\varSigma(\boldsymbol h); X_0) \cong \colim_{\sigma_k\in C_k} \mathcal A_{\sigma_k}(\boldsymbol h)\, .
			\]
		\end{thm}
		\begin{proof}
			For any collection of pairs $(\boldsymbol \varepsilon_\ell, \mathfrak a_\ell)$ such that \cref{lma:diagram_of_sub_algebras} holds for each $\ell\in \Z_{\geq 1}$, we have a canonical isomorphism of dg-algebras over $\boldsymbol k$ already on the chain level
			\[
				CE^\ast(\varSigma(\boldsymbol h, \boldsymbol \varepsilon_\ell); X_0^{\boldsymbol \varepsilon_\ell}, \mathfrak a_\ell) \cong \colim_{\sigma_k\in C_k} \mathcal A_{\sigma_k}(\boldsymbol h; \boldsymbol \varepsilon_\ell, \mathfrak a_\ell)\, ,
			\]
			by the definition of $\mathcal A_{\sigma_k}(\boldsymbol h; \boldsymbol \varepsilon, \mathfrak a)$ and by \cref{lma:Reeb_chords_one_to_one_correspondence}. The colimit in the right hand side is taken in the category $\mathbf{dga}$. By combining \cref{lma:cohomology_indep_of_size} and \cref{cor:subalgebras_diagram_no_actions} we obtain
			\begin{align*}
				CE^\ast(\varSigma(\boldsymbol h); X_0) &\cong \colim_{\mathfrak a \to \infty}CE^\ast(\varSigma(\boldsymbol h, \boldsymbol \varepsilon); X_0^{\boldsymbol \varepsilon}, \mathfrak a) \cong \colim_{\mathfrak a \to \infty}\colim_{\sigma_k\in C_k} \mathcal A_{\sigma_k}(\boldsymbol h; \boldsymbol \varepsilon, \mathfrak a) \\
				&\cong \colim_{\sigma_k\in C_k} \colim_{\mathfrak a \to \infty} \mathcal A_{\sigma_k}(\boldsymbol h; \boldsymbol \varepsilon, \mathfrak a) \cong \colim_{\sigma_k\in C_k} \mathcal A_{\sigma_k}(\boldsymbol h)\, .
			\end{align*}
		\end{proof}
		\subsubsection{Localizing dg-subalgebras}\label{sec:localizing}
			Next we give geometric meaning to the dg-subalgebras $\mathcal A_{\sigma_k}(\boldsymbol h)$ for $k\in \left\{0,\ldots,m\right\}$. We recall the following notation from \cref{dfn:attaching_spheres_simpl_handles}.
			\begin{align*}
				\varSigma_{\supset \sigma_k}(\boldsymbol h) = \bigcup_{\substack{\sigma_i \supset \sigma_k \\ k \leq i \leq m}}\varSigma_{\sigma_i}(\boldsymbol h)\, .
			\end{align*}
			For each $k$-face $\sigma_k \in C$ define
			\begin{equation}\label{eq:stopped_handle_data}
				\widetilde V^{(i)}_{\sigma_i}(\sigma_k) := \begin{cases}
					V^{(i)}_{\sigma_i}& \text{if } \sigma_i \supset \sigma_k \\
					(-\infty,0] \times \bigsqcup_{f\in F}(\# \widetilde{\boldsymbol V}_{\supsetneq \sigma_i,f}(\sigma_k)) \times \R & \text{otherwise}
				\end{cases}
			\end{equation}
			where $\widetilde{\boldsymbol V}_{\supsetneq \sigma_i,f}(\sigma_k)$ is the set of Weinstein manifolds $\widetilde V^{(\ell)}_{\sigma_{\ell}}(\sigma_k)$ for $f \supset \sigma_\ell \supsetneq \sigma_i$ with the same Weinstein hypersurfaces as in $\boldsymbol V$, and the Weinstein hypersurfaces induced by the inclusion $\{0\} \times (W \times \left\{0\right\}) \hookrightarrow (-\infty,0] \times (W \times \R)$ for those $\widetilde V^{(\ell)}_{\sigma_\ell}(\sigma_k)$ of the form as in the bottom row of \eqref{eq:stopped_handle_data}. Finally define $\widetilde{\boldsymbol V}(\sigma_k) := \bigcup_{\substack{f\in F \\ \sigma_i \subset f}} \widetilde{\boldsymbol V}_{\supsetneq \sigma_i,f}(\sigma_k)$ and
			\begin{equation}\label{eq:x_stopped_away_from_sigma_k}
				X(\sigma_k) := \# \widetilde{\boldsymbol V}(\sigma_k)\, .
			\end{equation}
			\begin{rmk}
				\begin{enumerate}
					\item We call $X(\sigma_k)$ ``$X$ stopped away from $\sigma_k$''. This is because we replace all $V^{(i)}_{\sigma_i}$ with a half symplectization of a contactization, where $\sigma_i \not \supset \sigma_k$ and this is exactly how stopped Weinstein manifolds were constructed in \cite{asplund2021chekanov}, Also see \cite{asplund2019fiber,ekholm2017duality} for a similar geometric construction of stopped Weinstein manifolds.
					\item By stopping a Weinstein manifold away from $\sigma_k\in C$ the Chekanov--Eliashberg dg-algebra $CE^\ast(\varSigma_{\supset \sigma_k}(\boldsymbol h); X(\sigma_k)_0)$ is thought of as being the ``localization'' of $CE^\ast(\varSigma(\boldsymbol h); X_0)$ at the $k$-face $\sigma_k$ and its subfaces. More precisely, the only Reeb chords and holomorphic disks that exist are confined to building blocks in $X$ associated to $\sigma_k$ and its subfaces.
					\item Equivalently, stopping $X$ away from $\sigma_k$ can be viewed as a form of completion of $V^{(k)}_{\sigma_k} \times \R^{2k}$ stopped at the Weinstein hypersurface $\left(\bigsqcup_{f\in F} \# \boldsymbol V_{\supsetneq \sigma_k,f}\right) \times \R^{2k}$ by negative ends. This is in spirit similar to convexification as in \cite[Remark 2.30]{ganatra2020covariantly} and more precisely described in \cite[Figure 2.1 and Remark 2.10]{eliashberg2018weinstein}.
				\end{enumerate}
			\end{rmk}
			We now describe how a choice of handle decomposition of $V^{(i)}_{\sigma_i}$ for $\sigma_i \supset \sigma_k$ induces a handle decomposition of $X(\sigma_k)$. By the same construction as in \cref{dfn:attaching_spheres_simpl_handles} we consider $\varSigma_{\supset \sigma_k}(\boldsymbol h)$ as a Legendrian submanifold in $\partial X(\sigma_k)_0$.
			\begin{lma}\label{lma:attaching_spheres_stopped_away_from_sigma_k}
				The Legendrian submanifold $\varSigma_{\supset \sigma_k}(\boldsymbol h) \subset \partial X(\sigma_k)_0$ is the union of the Legendrian attaching spheres for $X(\sigma_k)$.
			\end{lma}
			\begin{proof}
				The proof is the same as the proof of the second part of \cref{lma:attaching_spheres_simpl_handle}.
			\end{proof}

			\begin{lma}\label{lma:correspondence_generators_ce_sphere}
				For all $\mathfrak a > 0$ there exists some $\delta > 0$ and an arbitrary small perturbation of $\boldsymbol V_0$ such that for all $0 < \boldsymbol \varepsilon < \delta$ we have that the Reeb chords of $\varSigma_{\supset \sigma_k}(\boldsymbol h, \boldsymbol \varepsilon) \subset \partial X(\sigma_k)_0^{\boldsymbol \varepsilon}$ are in one-to-one correspondence with the generators of $\mathcal A_{\sigma_k}(\boldsymbol h; \boldsymbol \varepsilon, \mathfrak a)$.
			\end{lma}
			\begin{proof}
				Below the action bound and for sufficiently thin handles the set of generators of $\mathcal A_{\sigma_k}(\boldsymbol h; \boldsymbol \varepsilon, \mathfrak a)$ is equal to the set $\bigcup_{\sigma_i \supset \sigma_k} \mathcal R(\sigma_i, \mathfrak a)$ and is geometrically described in \cref{lma:Reeb_chords_one_to_one_correspondence}. Namely, they correspond to Reeb chords of $\varSigma_{\supset \sigma_k}(\boldsymbol h,\boldsymbol \varepsilon) \subset \partial X_0^{\boldsymbol \varepsilon}$ appearing in the loci $\partial_+(V^{(i)}_{\sigma_i,0} \times \R^{2i})$ for $\sigma_i \supset \sigma_k$ in $\partial X_0^{\boldsymbol \varepsilon}$.

				Thus, considering $\varSigma_{\supset \sigma_k}(\boldsymbol h,\boldsymbol \varepsilon) \subset \partial X(\sigma_k)_0^{\boldsymbol \varepsilon}$, we see that there are Reeb chords appearing in the loci $\partial_+(V^{(i)}_{\sigma_i,0} \times \R^{2i})$ for $\sigma_i \supset \sigma_k$, since these loci are unchanged in the definition of $\widetilde{\boldsymbol V}(\sigma_k)$, see \eqref{eq:stopped_handle_data}. The previously existing Reeb chords of $\varSigma_{\supset \sigma_k}(\boldsymbol h,\boldsymbol \varepsilon) \subset \partial X_0^{\boldsymbol \varepsilon}$ lying over critical points in loci corresponding to $i$-faces $\sigma_i$ such that $\sigma_i \not \supset \sigma_k$ disappear when passing to $\partial X(\sigma_k)_0^{\boldsymbol \varepsilon}$, since the contact manifold in these loci are of the form $P \times \R$ (see \eqref{eq:stopped_handle_data}) where the Legendrian is contained in the $P$-factor.
			\end{proof}
			\begin{lma}\label{lma:ce_spheres}
				There is a quasi-isomorphism of dg-algebras over $\boldsymbol k_{\sigma_k}$
				\[
					\mathcal A_{\sigma_k}(\boldsymbol h) \cong CE^\ast(\varSigma_{\supset \sigma_k}(\boldsymbol h); X(\sigma_k)_0)
				\]
				for $k\in \left\{0,\ldots,m\right\}$.
			\end{lma}
			\begin{proof}
				With a given action bound and for sufficiently thin handles there is a one-to-one correspondence of generators by \cref{lma:correspondence_generators_ce_sphere}. The correspondence between the holomorphic curves counted by the differential below a given action bound for sufficiently thin handles is given by \cref{lma:hol_curves_in_simplex_handle} and \cref{cor:no_curve_crosses_handle}. Therefore we have a canonical isomorphism of dg-algebras over $\boldsymbol k_{\sigma_k}$ already on the chain level 
				\[
					\mathcal A_{\sigma_k}(\boldsymbol h; \boldsymbol \varepsilon, \mathfrak a) \cong CE^\ast(\varSigma_{\supset \sigma_k}(\boldsymbol h, \boldsymbol \varepsilon); X(\sigma_k)_0^{\boldsymbol \varepsilon}, \mathfrak a)\, .
				\]
				Repeating the proof of \cref{thm:ce_descent_first} verbatim now proves the result.
			\end{proof}
			Summarizing \cref{sec:simplical_descent} we now obtain our main result.
			\begin{thm}\label{thm:ce_descent}
				There is a quasi-isomorphism of dg-algebras over $\boldsymbol k$
				\[
					CE^\ast(\varSigma(\boldsymbol h);X_0) \cong \colim_{\sigma_k\in C_k} CE^\ast(\varSigma_{\supset \sigma_k}(\boldsymbol h); X(\sigma_k)_0)\, .
				\]
			\end{thm}
			\begin{proof}
				This is an immediate consequence of \cref{thm:ce_descent_first} and \cref{lma:ce_spheres}.
			\end{proof}
\section{Good sectorial covers}\label{sec:simplicial_decompositions}
	In \cref{sec:simplicial_decompositions_and_good_sectorial_covers} we show that there is a one-to-one correspondence (up to deformation) between simplicial decompositions with so-called good sectorial covers of $X$. In \cref{sec:rel_to_sectorial_descent} we use this one-to-one correspondence and show that it allows us to recover the sectorial descent result of Ganatra--Pardon--Shende on the level of wrapped Floer cohomology.
	\subsection{Simplicial decompositions and good sectorial covers}\label{sec:simplicial_decompositions_and_good_sectorial_covers}
		We give a brief account of the definition of Liouville (and Weinstein) sectors and sectorial covers, following \cite{ganatra2020covariantly,ganatra2022sectorial}, and refer the reader to loc.\@ cit.\@ for more details.
		\begin{dfn}[Liouville sector {\cite[Definition 1.1]{ganatra2020covariantly}}]
			A Liouville sector is a Liouville manifold-with-boundary $(X, \lambda, Z)$ for which there is a function $I \colon \partial X \longrightarrow \R$ such that:
			\begin{itemize}
				\item $I$ is \emph{linear at infinity}, meaning $ZI = I$ outside a compact set, where $Z$ denotes the Liouville vector field.
				\item The Hamiltonian vector field $X_I$ of $I$ is outward pointing along $\partial X$.
			\end{itemize}
		\end{dfn}
		For every Liouville sector $X$, one can modify the Liouville form to obtain a Liouville pair $(\overline X, F)$ called the \emph{convexification} of $X$, see \cite[Section 2.7]{ganatra2020covariantly}. Moreover, up to a contractible choice, there is a one-to-one correspondence between Liouville sectors and Liouville pairs \cite[Lemma 2.32]{ganatra2020covariantly}.
		\begin{dfn}[Weinstein sector]
			A \emph{Weinstein sector} is a Liouville sector $X$ such that its convexification $(\overline X, F)$ is a Weinstein pair up to deformation.
		\end{dfn}
		\begin{rmk}
			An alternative definition of Weinstein sector is given in \cite[Definition 2.7]{chantraine2017geometric}, which does not involve convexification.
		\end{rmk}
		\begin{dfn}[Enlargement of Liouville sectors]\label{dfn:enlargement}
			Let $(X, \lambda, I)$ be a Liouville sector. An \emph{enlargement} of $X$ is defined by $X^+ := \varphi^\varepsilon(X)$ where $\varphi^\varepsilon$, is the time-$\varepsilon$ flow of the Hamiltonian vector field $X_I$, for some $\varepsilon > 0$.
		\end{dfn}
		\begin{dfn}[Sectorial cover {\cite[Definition 12.2 and Definition 12.19]{ganatra2022sectorial}}]
			Let $X$ be a Liouville sector. Suppose $X = X_1 \cup \cdots \cup X_n$, where each $X_i$ is a manifold-with-corners with precisely two faces $\partial^1 X_i := X_i \cap \partial X$ and the point set topological boundary $\partial^2 X_i$ of $X_i \subset X$, meeting along the corner locus $\partial X \cap \partial^2 X_i = \partial^1 X_i \cap \partial^2 X_i$. Such a covering $X = X_1 \cup \cdots \cup X_n$ is called \emph{sectorial} iff $\forall i \in \left\{1,\ldots, n\right\}$ there are functions $I_i \colon N^Z(\partial^2 X_i) \longrightarrow \R$ (where $N^Z$ denotes a neighborhood which is cylindrical with respect to the Liouville vector field $Z$) which is linear at infinity such that:
			\begin{itemize}
				\item $X_{I_i}$ is outward pointing along $\partial^2 X_i$.
				\item $X_{I_i}$ is tangent to $\partial^2 X_j$ along $\partial^2 X_i \cap \partial^2 X_j$ for $i\neq j$.
				\item $[X_{I_i},X_{I_j}] = 0$ along $N^Z(\partial^2 X_i) \cap N^Z(\partial^2 X_j)$.
			\end{itemize}
		\end{dfn}
		\begin{lma}[{\cite[Lemma 12.11]{ganatra2022sectorial}}]
			Let $X$ be a Liouville sector and suppose $X = X_1 \cup \cdots \cup X_n$ is a sectorial cover. For any $\varnothing \neq A \subset \left\{1,\ldots,n\right\}$ we have a Liouville isomorphism
			\begin{equation}\label{eq:coords_nghd_bdry}
				N^Z \left(\bigcap_{i\in A} \partial^2 X_i\right) \cong (X_A^{(k)} \times T^\ast \R^k, \lambda_{X_A^{(k)}} + \lambda_{T^\ast \R^k} + df)
			\end{equation}
			where $k := \abs A - 1$,  $X^{(k)}_{A}$ is a $(2n-2k)$-dimensional Weinstein manifold and $f$ is a real-valued function on $X^{(k)}_A \times T^\ast \R^k$ with support in $K \times T^\ast \R^k$ for some compact $K \subset X^{(k)}_{A}$.
		\end{lma}

		\begin{dfn}[Good sectorial cover]\label{dfn:good_sectorial_cover}
			Let $X$ be a Weinstein manifold, and suppose $X=X_1 \cup \cdots \cup X_m$ is a sectorial cover. We say that the sectorial cover is \emph{good} if for every $\varnothing \neq	A \subset \{1,\ldots,m\}$ we have a Weinstein isomorphism
			\begin{equation}\label{eq:coordinates_intersections}
				N^Z\left(\bigcap_{i\in A} X_i\right) \cong (X^{(k)}_{A} \times T^\ast \R^k,\lambda_{X^{(k)}_{A}} + \lambda_{T^\ast \R^k} + df)\, ,
			\end{equation}
			that extends the isomorphism \eqref{eq:coords_nghd_bdry}, where $k := \abs A - 1$,  $X^{(k)}_{A}$ is a $(2n-2k)$-dimensional Weinstein manifold and $f$ is a real-valued function on $X^{(k)}_A \times T^\ast \R^k$ with support in $K \times T^\ast \R^k$ for some compact $K \subset X^{(k)}_{A}$.
		\end{dfn}
		\begin{rmk}\label{rmk:good_sectorial_cover_hypersurfaces}
			In view of \cite[Lemma 12.8 and Lemma 12.10]{ganatra2022sectorial} we have that the condition of $X_1 \cup \cdots \cup X_m$ being a good sectorial cover is equivalent to the existence of a number of sectorial hypersurfaces $H_1, \ldots, H_{m'}$ such that a neighborhood of any $(k+1)$-fold intersection $X_{i_1} \cap \cdots \cap X_{i_{k+1}}$ is Weinstein isomorphic to a neighborhood of the $k$-fold intersection of a subcollection of the sectorial hypersurfaces $H_{1}, \ldots, H_{m'}$. In the coordinates \eqref{eq:coordinates_intersections}, each sectorial hypersurface $H_{i_j}$ is defined as the preimage of the $j$-th coordinate plane in the projection $X^{(k)}_A \times T^\ast \R^k \longrightarrow \R^k$.

			Alternatively, we might take the preimage of the first orthant of the $j$-th coordinate plane to get a sectorial hypersurfaces with boundary and corners. Then the condition of $X_1 \cup \cdots \cup X_m$ being a good sectorial cover is equivalent to the same as above, except that the hypersurfaces $H_i$ are allowed to have boundaries and corners. This gives a configuration of hypersurfaces with boundary and corners which split $X$ into a number of sectors-with-corners, see \cref{lma:splitting_along_hypersurfaces_with_corners}.
		\end{rmk}
		
		We first consider a straightforward generalization of \cite[Construction 12.17]{ganatra2022sectorial}.
		\begin{lma}\label{lma:construction_sector-with-corners}
			Let $(C, \boldsymbol V)$ be a simplicial decomposition. Then there is a one-to-one correspondence (up to Weinstein homotopy) between objects of the following three kinds.
			\begin{itemize}
				\item A Weinstein $2n$-manifold $X$ together with a Weinstein hypersurface $\# \boldsymbol V \hookrightarrow \partial X$.
				\item A Weinstein $2n$-sector $X'$ with symplectic boundary expressed as $\# \boldsymbol V$.
				\item A Weinstein $2n$-sector-with-corners $X''$ with codimension $2k$ symplectic boundaries $V^{(k)}_{\sigma_{k-1}}$ which in turn have codimension $2(\ell-k)$ symplectic boundaries $V^{(\ell)}_{\sigma_{\ell-1}}$ for every inclusion of faces $\sigma_{\ell-1} \supset \sigma_{k-1}$ in $C$ for $k \in \left\{1,\ldots,m\right\}$.
			 \end{itemize} 
		\end{lma}
		\begin{proof}
			The correspondence between items of the first and second kind is exactly the content of \cite[Lemma 2.32]{ganatra2020covariantly} and also \cite[Section 2.5.1 and Section 2.5.3]{alvarez2020positive}.

			Consider an item of the third kind. We will construct an item of the second kind by induction on the dimension $m = \dim C$. For simplicity we assume $C = \varDelta^m$. If not, we perform the construction facet-by-facet.
			\begin{description}
				\item[$m=1$] This is the content of \cite[Construction 12.17]{ganatra2022sectorial} which we recall for the sake of completeness.  The boundary of $X''$ is of the form $\partial X'' = H_1 \cup H_2$ where $H_1 = V^{(1)}_1 \times \R$, $H_2 = V^{(1)}_2 \times \R$ and $H_1 \cap H_2 = V^{(2)} \times \R^2$. We first choose a compatible collection of coordinates by \cite[Lemma 12.8 and Lemma 12.12]{ganatra2022sectorial} in the sense that they strictly respect Liouville forms
				\begin{align}
					V^{(1)}_1 \times T^\ast \R_{\leq 0} &\hookrightarrow X'' \label{eq:coord1} \\
					V^{(1)}_2 \times T^\ast \R_{\leq 0} &\hookrightarrow X'' \label{eq:coord2} \\
					V^{(2)} \times T^\ast \R_{\leq 0}^2 &\hookrightarrow X'' \label{eq:coord3} \\
					V^{(2)} \times T^\ast \R_{\leq 0} &\hookrightarrow V^{(1)}_1 \label{eq:coord4} \\
					V^{(2)} \times T^\ast \R_{\leq 0} &\hookrightarrow V^{(1)}_2 \label{eq:coord5} \, .
				\end{align}
				After convexification of $V^{(1)}_1$ and $V^{(1)}_2$ we have that the coordinates \eqref{eq:coord4} and \eqref{eq:coord5} give Weinstein hypersurfaces $V^{(2)} \hookrightarrow \partial V^{(1)}_i$ for $i\in \left\{1,2\right\}$, which allows us to construct $V^{(1)}_1 \#_{V^{(2)}} V^{(1)}_2$, see \cref{dfn:join_m=1}. Then after smoothing the corners of $X''$, \eqref{eq:coord1}--\eqref{eq:coord3} give coordinates 
				\begin{align}
					V^{(1)}_1 \times T^\ast \R_{\leq 0} &\hookrightarrow X' \label{eq:coord1_smoothing} \\
					V^{(1)}_2 \times T^\ast \R_{\leq 0} &\hookrightarrow X' \label{eq:coord2_smoothing} \\
					V^{(2)} \times T^\ast (\R \times \R_{\leq 0}) &\hookrightarrow X' \label{eq:coord3_smoothing}\, ,
				\end{align}
				where $X'$ is equal to the smoothing of $X''$. These together give us coordinates
				\[
					V^{(1)}_1 \#_{V^{(2)}} V^{(1)}_2 \times T^\ast \R_{\leq 0} \hookrightarrow X'\, .
				\]
				near the boundary of $X'$, which tells us exactly that the symplectic boundary of $X'$ is $V^{(1)}_1 \#_{V^{(2)}} V^{(1)}_2$.
				\item[$m > 1$] Now assume by induction that the construction has been done for all $C$ of dimension $ < m$. For each $\sigma_k \in \varDelta^m$ for $k\in \left\{0,\ldots,m-1\right\}$ we can pick a compatible collection of coordinates 
				\[
					V^{(\ell+1)}_{\sigma_{\ell}} \times T^\ast \R_{\leq 0}^{\ell-k} \hookrightarrow V^{(k+1)}_{\sigma_{k}}, \quad \sigma_\ell \supset \sigma_k\, ,
				\]
				near the codimension $2k$ symplectic boundary face corresponding to $\sigma_k \in \varDelta^m_k$. The induction hypothesis then gives us coordinates
				\begin{equation}\label{eq:joined_symplectic_boundary_k_face}
					\# \boldsymbol V_{\supsetneq \sigma_k} \times T^\ast \R_{\leq 0} \hookrightarrow {V'}^{(k+1)}_{\sigma_k}, \quad \sigma_k\in \varDelta^m_k\, ,
				\end{equation}
				near the boundary of the smoothing of the codimension $2k$ symplectic boundary face, where ${V'}^{(k+1)}_{\sigma_k}$ is equal to the smoothing of $V^{(k+1)}_{\sigma_k}$. Using \eqref{eq:joined_symplectic_boundary_k_face} we construct the set $\boldsymbol V$ and then we construct $\# \boldsymbol V$, see \cref{sec:construction_m_simplex_handle}. We may finally smooth $X''$ to get coordinates
				\[
					\# \boldsymbol V \times T^\ast \R_{\leq 0} \hookrightarrow X'\, .
				\]
				near the boundary of $X'$, which which tells us exactly that the symplectic boundary of $X'$ is $\# \boldsymbol V$.
			\end{description}
			\begin{figure}[!htb]
				\centering
				\includegraphics{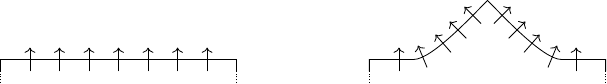}
				\caption{Modifying $\R \times \R_{\leq 0}$ (left) to $A$ (right) by introducing a corner near the origin.}\label{fig:2corner}
			\end{figure}
			\begin{figure}[!htb]
				\centering
				\includegraphics{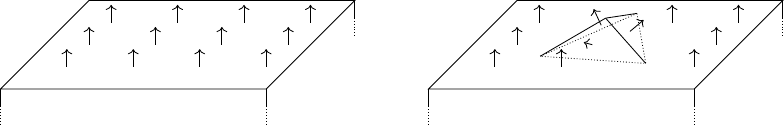}
				\caption{Modifying $\R^2 \times \R_{\leq 0}$ (left) by introducing a corner near the origin.}\label{fig:3corner}
			\end{figure}
			Let us now start with an item of the second kind. We aim to introduce a corner structure dictated by the simplicial complex $C$, which turns it into an item of the third kind. The case $C = \varDelta^1$ and the symplectic boundary is of the form $Q = V^{(1)}_1 \#_{V^{(2)}}V^{(1)}_2$ is done in \cite[Construction 12.17]{ganatra2022sectorial}. Namely, in $Q$ there are coordinates $V^{(2)} \times T^\ast \R \hookrightarrow Q$ near $V^{(2)}$, and so we get coordinates $V^{(2)} \times T^\ast (\R \times \R_{\leq 0}) \hookrightarrow X'$. We introduce a corner near the origin in $\R \times \R_{\leq 0}$ as shown in \cref{fig:2corner}, and so we replace $T^\ast(\R \times \R_{\leq 0})$ with $T^\ast A$ which near the origin gives coordinates $V^{(2)} \times T^\ast \R_{\leq 0}^2 \hookrightarrow X''$. We treat the case $C = \varDelta^m$ similarly. The symplectic boundary is of the form $Q = \# \boldsymbol V$ and near $V^{(m)}$ we have coordinates
			\begin{equation}\label{eq:introducing_corners}
				V^{(m)} \times T^\ast(\R^m \times \R_{\leq 0})\hookrightarrow X'\, .
			\end{equation}
			We introduce a corner structure near the origin in $\R^m \times \R_{\leq 0}$ as follows. By \cref{lma:unique_flow_line} there are exactly $m$ rays we denote by $\gamma_1,\ldots,\gamma_m$, each of which emanates from the origin in the (base of the second factor of the) basic building block $V^{(m)}_{\sigma_m} \times T^\ast \R^m$ and ends at the origin in the (base of the second factor of the) basic building block $V^{(m-1)}_{\sigma_{m-1}} \times T^\ast \R^{m-1}$ for $\sigma_m \supset \sigma_{m-1}$. As shown in \cref{fig:3corner} in the case $m=2$ we introduce a corner structure such that near the origin in the factor $\R^m \times \R_{\leq 0}$ in the domain of \eqref{eq:introducing_corners} there are coordinates $\R_{\leq 0}^{m+1}$ which are picked such that the $i$-th coordinate axis coincides with $\gamma_i \subset \R^m \times \left\{0\right\} \subset \R^m \times \R_{\leq 0}$ if one would undo the introduction of the corner structure by smoothing. After introducing a corner structure, denote the resulting Weinstein sector-with-corners by $X''$. Near $V^{(k+1)}_{\sigma_k}$ for any $\sigma_k \in \varDelta^m_k$ we have coordinates 
			\[
				V^{(k+1)}_{\sigma_k} \times T^\ast \R_{\leq 0}^{k+1} \hookrightarrow X'', \quad \sigma_k \in \varDelta^m_k\, ,
			\]
			and near $V^{(\ell+1)}_{\sigma_\ell}$ in $V^{(k+1)}_{\sigma_k}$ for $\sigma_\ell \subset \sigma_k$ we have coordinates
			\[
				V^{(\ell+1)}_{\sigma_\ell} \times T^\ast \R_{\leq 0}^{\ell-k} \hookrightarrow V^{(k+1)}_{\sigma_k}, \quad \sigma_\ell \supset \sigma_k\, ,
			\]
			and we have an item of the third kind as desired.
		\end{proof}
		We now give a characterization of the existence of a good sectorial cover in terms of sectorial hypersurfaces with boundary and corners.
		\begin{lma}\label{lma:splitting_along_hypersurfaces_with_corners}
			Let $X$ be a Weinstein manifold. There is a collection of sectorial hypersurfaces with boundary and corners in $X$ which splits $X$ into $m$ Weinstein sectors-with-corners if and only if $X$ admits a good sectorial cover $X = X_1 \cup \cdots \cup X_m$.
		\end{lma}
		\begin{proof}
			\begin{description}
				\item[$\Rightarrow$]
					If there are such a collection of sectorial hypersurfaces with boundary and corners $\left\{H_j\right\}_j$, we split $X$ into $m$ connected components. Taking the closure of each connected component gives $m$ Weinstein sectors-with-corners $X_1,\ldots, X_m$. After slight enlargement of each $X_i$ (see \cref{dfn:enlargement}), we obtain a sectorial cover $X = X_1^+ \cup \cdots \cup X_m^+$ such that for any $\varnothing \neq I \subset \left\{1,\ldots,m\right\}$ we have $N \left(\bigcap_{i\in I} X_i^+\right) = N(H_{i_1} \cap \cdots \cap H_{i_k})$, which by \cite[Lemma 12.12]{ganatra2022sectorial} gives us coordinates in a neighborhood of each intersection
					\[
						N \left(\bigcap_{i\in I} X_i^+\right) \cong X^{(k)}_I \times T^\ast \R^k\, ,
					\]
					respecting the Liouville forms, where $X^{(k)}_I$ is a $(2n-2k)$-dimensional Weinstein manifold. After smoothing corners of each Weinstein sector-with-corners $X_i^+$, we obtain a $(2n-2k)$-dimensional Weinstein sector and denoting the result by $X_i'$ by \cref{lma:construction_sector-with-corners}, which gives the good sectorial cover $X = X_1' \cup \cdots \cup X_m'$.
				\item[$\Leftarrow$]
					Let $C$ be the \v{C}ech nerve of the cover $X_1 \cup \cdots \cup X_m$ and fix an identification between $2^{\left\{1,\ldots,m\right\}}$ and all faces of $C$. For each facet $f\in F$ of $C$, we consider a neighborhood of the intersection
					\[
						N \left(\bigcap_{\sigma_0 \subset f} X_{\sigma_0}\right) \cong X^{(k)}_f \times T^\ast \R^k \, ,
					\]
					where $\dim f = k+1$, and the projection 
					\[
						\pi_0 \colon X^{(k)}_f \times T^\ast \R^k \longrightarrow \R^k\, ,
					\]
					to the zero section of the second factor. Identify $\R^k$ with the interior of a $k$-simplex centered at the origin. For any set of $\ell+1$ vertices $A \subset f_0$, where $f_0$ is the set of vertices of $f$, we choose coordinates such that the following diagram commutes
					\begin{equation}\label{eq:coordinates_facet_good_cover}
						\begin{tikzcd}[row sep=scriptsize, column sep=scriptsize]
							X^{(k)}_f \times T^\ast \R^k \rar[hook] \dar{\pi_0} & X^{(\ell)}_A \times T^\ast \R^{\ell} \dar{\pi_0} \\
						\R^k \rar{\pi_A} & \R^{\ell}
						\end{tikzcd}\, ,
					\end{equation}
					where $\pi_A$ is orthogonal projection on (the interior of) one of the $\ell$-faces determined by its vertices in $\varDelta^k$ using the standard inner product on $\R^k$. Let $B \varDelta^k$ denote the barycentric subdivision of $\varDelta^k$. Let $\varPi^k \subset \varDelta^k$ be the simplicial subcomplex generated by the set of vertices in $B \varDelta^k_0 \setminus \varDelta^k_0$, see \cref{fig:dual_spines}.
					\begin{figure}[!htb]
						\centering
						\includegraphics{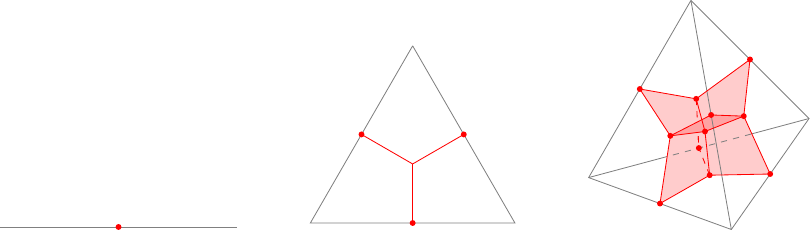}
						\caption{The subset $\varPi^k \subset \varDelta^k$ drawn in red for $k = 1,2,3$.}\label{fig:dual_spines}
					\end{figure}
					Using the coordinates as in \eqref{eq:coordinates_facet_good_cover}, we can define a projection
					\[
						\pi_f \colon \bigcup_{\sigma_0 \subset f} X_{\sigma_0} \longrightarrow \varDelta^k\, ,
					\]
					using the projections $\pi_0$ to the zero section of the second factor. Then define
					\[
						H_f := \pi_f^{-1}(\varPi^k)\, .
					\]
					By definition it is a union of sectorial hypersurfaces with boundary and corners which splits the facet $f$ into $k+1$ pieces. Finally, we let $H := \bigcup_{f \in F} H_f$. Again, this is a union of sectorial hypersurfaces with boundary and corners, and splitting $X$ along $H$ gives one Weinstein sector-with-corners for each vertex $\sigma_0\in C_0$.
			\end{description}
		\end{proof}
		\begin{thm}\label{thm:one-to-one_corr_covers_and_decomps}
			Let $X$ be a Weinstein manifold. There is a one-to-one correspondence (up to Weinstein homotopy) between good sectorial covers of $X$ and simplicial decompositions of $X$.
		\end{thm}
		\begin{proof}
			First consider a simplicial decomposition $(C, \boldsymbol V)$ of $X$. Applying \cref{lma:construction_sector-with-corners} to each $V^{(k)}_{\sigma_k} \in \boldsymbol V$ gives a Weinstein $(2n-2k)$-sector ${V'}^{(k)}_{\sigma_k}$ for each $k$-face $\sigma_k\in C_k$, $k\in \left\{0,\ldots,m\right\}$ with symplectic boundary expressed as $\bigsqcup_{f\in F} \# \boldsymbol V_{\supsetneq \sigma_k,f}$, and each ${V'}^{(k+1)}_{\sigma_{k+1}}$ is a Weinstein subsector of the symplectic boundary of ${V'}^{(k)}_{\sigma_k}$ for $\sigma_{k+1} \supset \sigma_k$. Let $({V'}^{(0)}_{\sigma_0})^+$ be a slight enlargement of each ${V'}^{(0)}_{\sigma_0}$ (see \cref{dfn:enlargement}). We then claim that
			\[
				X = \bigcup_{\sigma_0\in C_0} ({V'}^{(0)}_{\sigma_0})^+\, ,
			\]
			is a good sectorial cover. To see this, let $A \subset C_0$ be a subset of vertices of $C$ and consider $\bigcap_{\sigma_0\in A} ({V'}^{(0)}_{\sigma_0})^+$. We can without loss of generality assume that every vertex in $A$ all lie in the same facet of $C$, since otherwise we can write
			\begin{equation}\label{eq:wlog_vertices_in_same_facet}
				\bigcap_{\sigma_0\in A} ({V'}^{(0)}_{\sigma_0})^+ = \bigsqcup_{f\in F} \left(\bigcap_{\substack{\sigma_0 \in A \\ \sigma_0 \subset f}} ({V'}^{(0)}_{\sigma_0})^+\right)\, ,
			\end{equation}
			and then look at each intersection $\bigcap_{\substack{\sigma_0 \in A \\ \sigma_0 \subset f}} ({V'}^{(0)}_{\sigma_0})^+$ separately. Now, by construction we have that $N \left(\bigcap_{\sigma_0\in A} ({V'}^{(0)}_{\sigma_0})^+\right)$ is a small neighborhood of the Weinstein sector ${V'}^{(k)}_{\sigma_k}$ whose symplectic boundary is of the form $\bigsqcup_{f\in F} \# \boldsymbol V_{\supsetneq \sigma_k,f}$, where $\abs A = k+1$. After applying \cref{lma:construction_sector-with-corners} to ${V'}^{(k)}_{\sigma_k}$ for each $\sigma_k\in C_k$, $k\in \left\{1,\ldots,m\right\}$ we obtain a Weinstein $(2n-2k)$-manifold together with a Weinstein hypersurface 
			\[
				\bigsqcup_{f\in F} \# \boldsymbol V_{\supsetneq \sigma_k,f} \hookrightarrow \partial V^{(k)}_{\sigma_k}\, ,
			\]
			By this construction we get
			\[
				N \left(\bigcap_{\sigma_0\in A} ({V'}^{(0)}_{\sigma_0})^+\right) \cong V^{(k)}_{\sigma_k} \times T^{\ast} \R^k\, .
			\]
			Conversely, starting with a good sectorial cover $X = X_1 \cup \cdots \cup X_{m+1}$, we first let $C$ be the \v{C}ech nerve of the cover. Next apply \cref{lma:splitting_along_hypersurfaces_with_corners}, and after splitting along these sectorial hypersurfaces with boundary and corners we obtain a collection of Weinstein sectors-with-corners $\{X''_{\sigma_0}\}_{\sigma_0\in C_0}$. By construction $X''_{\sigma_0}$ has codimension $2k$ symplectic boundaries ${V''}^{(k)}_{\sigma_{k}}$ which in turn have codimension $2(\ell-k)$ symplectic boundaries ${V''}^{(\ell)}_{\sigma_{\ell}}$ for every inclusion of faces $\sigma_{\ell} \supset \sigma_{k}$ in $C$. Then by \cref{lma:construction_sector-with-corners} we get a Weinstein $(2n-2k)$-manifold $V^{(k)}_{\sigma_k}$ together with a Weinstein hypersurface
			\[
				\iota_{\sigma_k} \colon \bigsqcup_{f\in F} \# \boldsymbol V_{\supsetneq \sigma_k,f} \hookrightarrow \partial V^{(k)}_{\sigma_k}\, ,
			\]
			for each $\sigma_k\in C_k$ for $k\in \left\{0,\ldots,m\right\}$. Letting $\boldsymbol V := (\bigcup_{f\in F}\bigcup_{\substack{\sigma_k \in C_k \\ 0 \leq k \leq m}} \boldsymbol V_{\sigma_k,f}) \cup \{V^{(0)}_{\sigma_0}, \iota_{\sigma_0}\}_{\sigma_0\in C_0}$ we have $X \cong \# \boldsymbol V$ and we obtain a simplicial decomposition $(C, \boldsymbol V)$ of $X$.
		\end{proof}
		\begin{rmk}
			This correspondence in the case for good sectorial covers $X = X_1 \cup X_2$ with exactly two sectors in the cover is the content of \cite[Construction 12.18]{ganatra2022sectorial}.
		\end{rmk}
		We now prove that good sectorial covers exist in abundance.
		\begin{lma}\label{lma:good_sectorial_cover_refinement}
			Let $X = X_1 \cup \cdots \cup X_m$ be any sectorial cover of the Weinstein manifold $X$. Let $X_1' := X_1$ and $X_k' := X_k \setminus \bigcup_{j=1}^{k-1} \mathring{X_j}$ for $k > 1$. Then 
			\[
				X = (X_1')^+ \cup \cdots \cup (X_m')^+\, ,
			\]
			is a good sectorial cover of $X$, where $X^+$ denotes a small enlargement of the Weinstein sector $X$, see \cref{dfn:enlargement}.
		\end{lma}
		\begin{proof}
			By construction we have that for any subcollection of numbers $K := \left\{k_1,\ldots,k_\ell\right\} \subset \left\{1,\ldots,m\right\}$ with $k_1 < \cdots < k_\ell$ that
			\begin{align*}
				\bigcap_{k\in K} X_k' &= \bigcap_{k\in K} \left(X_k \setminus \bigcup_{j=1}^{k-1} \mathring{X_j}\right) = \left(X_{k_1} \cap \left(X_{k_2} \setminus \mathring{X_{k_1}}\right) \cap \cdots \cap \left(X_{k_\ell} \setminus \bigcup_{j=1}^{\ell-1} X_{k_j}\right)\right) \setminus \bigcup_{k\not\in K} \mathring{X_k} \\
				&= \left(X_{k_\ell} \cap \partial X_{k_1} \cap \left(\partial X_{k_2} \setminus \mathring{X_{k_1}}\right) \cap \cdots \cap \left(\partial X_{k_{\ell-1}} \setminus \bigcup_{j=1}^{\ell-2} \mathring{X_{k_j}}\right)\right) \setminus \bigcup_{k\not \in K}\mathring{X_k}
			\end{align*}
			is the intersection of a collection of $\ell-1$ sectorial hypersurfaces with boundary and corners in $X$, which means that $\bigcap_{k\in K} (X_k')^+$ is a neighborhood of such an intersection of sectorial hypersurfaces with boundary and corners. By \cite[Lemma 12.8 and Lemma 12.10]{ganatra2022sectorial} (see also \cref{rmk:good_sectorial_cover_hypersurfaces}) we can thus find coordinates as in \eqref{eq:coordinates_intersections}.
		\end{proof}
		We end this section by working out some simple examples.
		\begin{ex}\label{ex:simplicial_decomposition_cotangent}
			Let $M$ be an $n$-dimensional manifold (possibly with boundary). Let $M = M_1 \cup \cdots \cup M_m$ be a good cover of $M$ of codimension zero manifolds (possibly with boundary). Assume that every intersection among members of the collection $\left\{\partial M\right\} \cup \left\{\partial M_i\right\}_{i=1}^m$ is transverse. Then $\left\{T^\ast M_i\right\}_{i=1}^m$ is a good sectorial cover. Indeed by contractibility, we have
				\[
					N(T^\ast M_{i_1} \cap \cdots \cap T^\ast M_{i_{k+1}}) = N(T^\ast(M_{i_1} \cap \cdots \cap M_{i_{k+1}})) \cong T^\ast \R^n \cong T^\ast \R^{n-k} \times T^\ast \R^k\, .
				\]
		\end{ex}
		\begin{ex}\label{ex:correspondence}
		Consider the good sectorial cover $T^\ast (0,1) = T^\ast \left(0, \frac 34\right] \cup T^\ast \left[\frac 14, 1\right)$. The cotangent fiber $T^\ast_{1/2}(0,1) \subset T^\ast \left(0, \frac 34\right] \cap T^\ast \left[\frac 14, 1\right)$ is a sectorial hypersurface. Splitting along this hypersurface (and taking closures of each piece) gives two Weinstein sectors $X_1 = T^\ast \left(0, \frac 12\right]$ and $X_2 = T^\ast \left[\frac 12, 1\right)$ with common symplectic boundary $\left\{\frac 12\right\}$. After convexification of $X_1$ and $X_2$ we obtain two copies of $(D^2, \frac 12(xdy-ydx))$ together with a Weinstein hypersurface $\iota \colon \left\{\frac 12\right\} \hookrightarrow \partial D^2$ such that $T^\ast (0,1) \cong D^2 \#_{N(\left\{\frac 12\right\})} D^2$, see \cref{fig:ex_interval_1} and \cref{fig:ex_interval_2}. In this case $C = \varDelta^1$ is the $1$-simplex, and we obtain a simplicial decomposition of $T^\ast(0,1)$ as $(C, \boldsymbol V)$, where
		\[
			\boldsymbol V = \left\{D^2_1,D^2_2,N \left(\left\{\frac 12\right\}\right), \iota_1, \iota_2\right\}\, .
		\]
			\begin{figure}[!htb]
				\centering
				\includegraphics{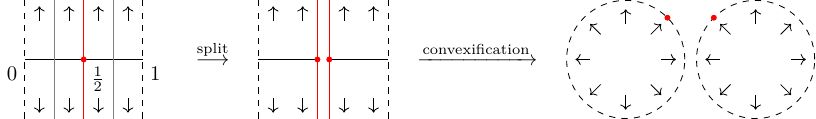}
				\caption{Splitting of the good sectorial cover $T^\ast (0,1) = T^\ast \left(0, \frac 34\right] \cup T^\ast \left[\frac 14, 1\right)$ along the hypersurface $T^\ast_{1/2}(0,1)$, followed by convexification.}\label{fig:ex_interval_1}
			\end{figure}
			\begin{figure}[!htb]
				\centering
				\includegraphics{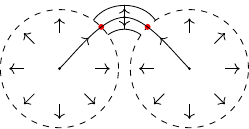}
				\caption{The result of gluing two copies of $D^2$ together using a simplicial handle $N(\text{pt}) \times D_\varepsilon T^\ast \varDelta^1$ along two copies of the Weinstein hypersurface $N(\text{pt}) \hookrightarrow \partial D^2$.}
				\label{fig:ex_interval_2}
			\end{figure}
			\end{ex}
			\begin{ex}\label{ex:correspondence2}
				Let $S^n = D^n_1 \cup D^n_2$ where $D^n_1 \cap D^n_2 \cong S^{n-1} \times (-\delta,\delta)$ for some $\delta > 0$. Then $T^\ast S^n = T^\ast D^n_1 \cup T^\ast D^n_2$ is a good sectorial covering, since
				\[
					T^\ast D^n_1 \cap T^\ast D^n_2 \cong T^\ast S^{n-1} \times T^\ast(-\delta,\delta)\, .
				\]
				Similar to the example in (1), there is a sectorial hypersurface $T^\ast S^{n-1} \times T^\ast_0 (-\delta,\delta) \subset T^\ast S^n$ which splits $T^\ast S^n$ into exactly two parts $X_1$ and $X_2$. The closure of each $X_i$ is the cotangent bundle of a closed $n$-disk. Via convexification, each $X_i$ corresponds to the Weinstein pair $(D^{2n}, N(S^{n-1}))$ where $\iota \colon N(S^{n-1}) \hookrightarrow \partial D^{2n}$ is the Weinstein hypersurface that is a cotangent neighborhood of the standard unknot in $D^{2n}$.
				
				The corresponding simplicial decomposition of $T^\ast S^n$ corresponding to the good sectorial cover of $T^\ast S^n$ is given by the following.
				\begin{itemize}
					\item The simplicial complex $C = \varDelta^1$.
					\item The set $\boldsymbol V = \left\{D^{2n}_1, D^{2n}_2,N(S^{n-1}), \iota_1, \iota_2\right\}$, where $N(S^{n-1}) \cong D_\varepsilon T^\ast S^{n-1}$ is a small cotangent neighborhood of the Legendrian unknot $S^{n-1} \subset \partial D^{2n}$, $D^{2n}_1$ and $D^{2n}_2$ are two copies of the standard Weinstein $2n$-ball
					\[
						\left(D^{2n}, \frac 12 \left(\sum_{i=1}^n (x_i dy_i + y_i dx_i)\right)\right)\, .
					\]
					Finally $\iota_1$ and $\iota_2$ are two copies of the Weinstein hypersurface $N(S^{n-1}) \hookrightarrow \partial D^{2n}$.
				\end{itemize}
				It is evident that gluing the simplicial handle $H^1_\varepsilon(N(S^{n-1}))$ to $D^{2n} \sqcup D^{2n}$ yields the standard Weinstein $T^\ast S^n$.
			\end{ex}
		\begin{ex}\label{ex:correspondence3}
			Let $X = T^\ast \mathring{\varDelta}^2$ where $\mathring{\varDelta}^2$ is the interior of a $2$-simplex centered at the origin in $\R^2$ with vertices $(0,1)$ and $\left(\pm \frac{\sqrt 3}2, - \frac 12\right)$. Equip $X$ with the radial Liouville vector field $Z = \frac 12 (x_1 \partial_{x_1}+x_2 \partial_{x_2}+y_1 \partial_{y_1}+y_2 \partial_{y_2})$. Consider the good sectorial cover
			\[
				T^\ast \mathring{\varDelta}^2 = T^\ast (B_1 \cap \mathring{\varDelta}^2) \cup T^\ast (B_2 \cap \mathring{\varDelta}^2) \cup T^\ast (B_3 \cap \mathring{\varDelta}^2)\, ,
			\]
			where
			\begin{align*}
				B_1 := B_{1+\varepsilon} \left(0,1\right), \qquad B_2 := B_{1+\varepsilon} \left(- \frac{\sqrt 3}{2}, - \frac 12\right), \qquad B_3 := B_{1+\varepsilon} \left(\frac{\sqrt 3}{2}, - \frac 12\right)\, ,
			\end{align*}
			are closed balls of radius $1+\varepsilon$ centered at the vertices of $\mathring{\varDelta}^2$, for some small $\varepsilon > 0$.
			\begin{figure}[!htb]
				\centering
				\includegraphics{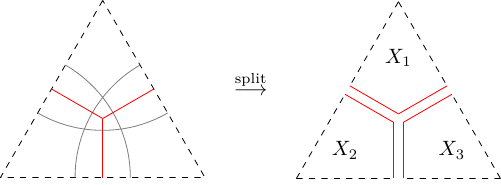}
				\caption{The cover $\mathring{\varDelta}^2 = \bigcup_{i=1}^3 (B_i \cap \mathring{\varDelta}^2)$ together with projection of three sectorial hypersurfaces with boundary (in red) having common boundary projecting to $(0,0)$.}\label{fig:splitting_along_hypersurfaces}
			\end{figure}
			Consider the three lines 
			\[
				h_1 := \left\{\left(x, \frac x{\sqrt 3}\right) \suchthat x \geq 0\right\} \cap \mathring{\varDelta}^2, \quad h_2 := \left\{\left(x, -\frac x{\sqrt 3}\right) \suchthat x \leq 0\right\} \cap \mathring{\varDelta}^2, \quad h_3 := \left\{(0, y) \suchthat y \leq 0\right\} \cap \mathring{\varDelta}^2\, .
			\]
			Then each $H_i := T^\ast h_i$ is a sectorial hypersurface with boundary in $T^\ast \mathring{\varDelta}^2$ for each $i \in \left\{1,2,3\right\}$. The $H_i$ have common horizontal boundary being $T^\ast_{(0,0)}\mathring{\varDelta}^2$. Then $H_1 \cup H_2 \cup H_3$ splits $T^\ast \mathring{\varDelta}^2$ into exactly three parts whose closures are denoted by $X_1$, $X_2$ and $X_3$ respectively, see \cref{fig:splitting_along_hypersurfaces}. Each $X_i$ is a Weinstein sector-with-corners. It has exactly two codimension two symplectic boundary strata $H_i$ and $H_j$, which meet in the codimension 4 symplectic boundary stratum $\left\{(0,0)\right\}$.

			By \cref{lma:construction_sector-with-corners} we first smooth the corners of each $X_i$ to obtain a Weinstein sector. The convexification of the smoothing of $X_i$ is a Weinstein manifold $X'_i$ together with a Weinstein hypersurface $\iota_{i,j,k} \colon H_i \#_{N(\left\{(0,0)\right\})} H_j \hookrightarrow \partial X'_k$ where $\left\{i,j,k\right\} = \left\{1,2,3\right\}$. This gives us a simplicial decomposition $(\varDelta^2, \boldsymbol V)$ of $T^\ast \mathring{\varDelta}^2$ where $\boldsymbol V = \left(\bigcup_{n=1}^3 \boldsymbol V_n\right) \cup \left\{X_1',X_2',X_3', \iota_{i,j,k}\right\}$ where
			\[
				\boldsymbol V_k = \left\{N(\left\{(0,0)\right\}), \iota_i, \iota_j\right\}\, ,
			\]
			for $\left\{i,j,k\right\} = \left\{1,2,3\right\}$.
			where $\iota_i \colon N(\left\{(0,0)\right\}) \hookrightarrow \partial H_i$ is a Weinstein hypersurface
		\end{ex}
		\subsection{Relation to sectorial descent}\label{sec:rel_to_sectorial_descent}
			One of the motivating results for this paper is the sectorial descent formula of Ganatra--Pardon--Shende for wrapped Fukaya categories, see \eqref{eq:gps_descent} below. Using the correspondence between good sectorial covers and simplicial decompositions in \cref{thm:one-to-one_corr_covers_and_decomps} we now describe the relation between our main theorem \cref{thm:ce_descent} and the sectorial descent result.

			Let $(C, \boldsymbol V)$ be a simplicial decomposition of a Weinstein manifold $X$. In \cref{sec:simplical_descent} we proved that there is a quasi-isomorphism
			\[
				CE^\ast(\varSigma(\boldsymbol h);X_0) \cong \colim_{\sigma_k\in C_k}CE^\ast(\varSigma_{\supset \sigma_k}(\boldsymbol h); X(\sigma_k)_0)\, .
			\]
			Now we relate each term in this diagram to wrapped Floer cohomology.
			\begin{lma}\label{lma:surgery_gives_wrapped_floer_of_intersection}
				There is a quasi-isomorphism of $A_\infty$-algebras
				\[
					CE^\ast(\varSigma_{\supset \sigma_k}(\boldsymbol h); X(\sigma_k)_0) \cong CW^\ast \left(C_{\sigma_k}; X(\sigma_k)\right)\, ,
				\]
				where $C_{\sigma_k}$ denotes the union of the cocore disks of the Weinstein handles attached to $\varSigma_{\supset \sigma_k}(\boldsymbol h)$.
			\end{lma}
			\begin{proof}
				By \cref{lma:attaching_spheres_stopped_away_from_sigma_k}, $\varSigma_{\supset \sigma_k}(\boldsymbol h)$ is the union of the attaching spheres for $X(\sigma_k)$ in $X(\sigma_k)_0$. The result then follows by applying the surgery formula \cite{bourgeois2012effect,ekholm2019hol}.
			\end{proof}
			One of the main results in \cite{ganatra2022sectorial} says that there is a functor $A_\infty$-categories
			\begin{equation}\label{eq:gps_descent}
				W(X) \longrightarrow \hocolim_{\varnothing \neq I \subset \left\{1,\ldots,m\right\}} \mathcal W \left(\bigcap_{i\in I} X_i\right)\, ,
			\end{equation}
			which induces a quasi-equivalence on triangulated envelopes. Using our main result \cref{thm:ce_descent} and \cref{lma:surgery_gives_wrapped_floer_of_intersection} we now have the following weaker version of \eqref{eq:gps_descent}.
			\begin{cor}\label{cor:sectorial_descent}
				Let $(C, \boldsymbol V)$ be a simplicial decomposition of a Weinstein manifold $X$. Then there is an isomorphism of $A_\infty$-algebras
				\[
					HW^\ast(C; X) \cong \colim_{\sigma_k\in C_k} HW^\ast(C_{\sigma_k}; X(\sigma_k)) \, ,
				\]
				where $C_{\sigma_k}$ denotes the union of the cocore disks of the Weinstein handles attached to $\varSigma_{\supset \sigma_k}(\boldsymbol h)$.
			\end{cor}
			\begin{proof}
				We consider the diagram $\left\{CE^\ast(\varSigma_{\supset \sigma_k}(\boldsymbol h) ; X(\sigma_k)_0)\right\}_{\sigma_k\in C}$ and pass to cohomology. From \cref{lma:surgery_gives_wrapped_floer_of_intersection} we get isomorphisms in cohomology which gives the desired result.
			\end{proof}
			\begin{rmk}
				The reason \cref{cor:sectorial_descent} follows from \eqref{eq:gps_descent} is due to stop removal and convexification. First apply stop removal to $X(\sigma_k)$ in the sense of \cite[Section 5.2]{asplund2021chekanov}, ie replace Weinstein manifolds of the form $(-\infty,0] \times (W \times \R)$ in $\boldsymbol V(\sigma_k)$ with $W \times B^2$, where $B^2$ is equipped with the Liouville vector field $\frac 12(x \partial_x + y \partial_y)$. The resulting manifold is Weinstein homotopic via some Weinstein handle cancelations to the convexification of the Weinstein sector $\bigcap_{i\in I} X_i$.
			\end{rmk}
	\section{Applications}\label{sec:applications}
			In \cref{sec:relative_version} we describe how our definition of a simplicial decomposition allows for the study of Legendrian submanifolds-with-boundary-and-corners. In \cref{sec:ce_cosheaf} we define the Chekanov--Eliashberg cosheaf and finally in \cref{sec:plumbing} we prove by direct computation that the Chekanov--Eliashberg dg-algebra of the Legendrian top attaching spheres of a plumbing of a collection of $T^\ast S^n$ for $n\geq 3$ is quasi-isomorphic to the Ginzburg dg-algebra.
		\subsection{Legendrian submanifolds-with-boundary-and-corners}\label{sec:relative_version}
			Suppose $(C, \boldsymbol V)$ is a simplicial decomposition of a Weinstein manifold $X$.
			\begin{dfn}
				A Legendrian submanifold \emph{relative to a simplicial decomposition} is a collection $\boldsymbol \varLambda$ consisting of one embedded Legendrian $(n-k-1)$-submanifold-with-boundary-and-corners $\varLambda_{\sigma_k} \subset \partial V^{(k)}_{\sigma_k}$ for each $\sigma_k\in C$, such that 
				\[
					\partial \varLambda_{\sigma_k} = \# \boldsymbol \varLambda_{\supsetneq \sigma_k} \subset \partial \# \boldsymbol V_{\supsetneq \sigma_k}\, ,
				\]
				where $\boldsymbol \varLambda_{\supsetneq \sigma_k} := \left\{\varLambda_{\sigma_i}\right\}_{\sigma_i \supsetneq \sigma_k}$, and $\# \boldsymbol \varLambda_{\supsetneq \sigma_k}$ is defined recursively in analogy to the Weinstein manifold $\# \boldsymbol V_{\supsetneq \sigma_k}$, see \cref{sec:construction_m_simplex_handle}. 
			\end{dfn}
			We point out that Legendrian submanifolds-with-boundary was studied from the point of view of sutured contact manifolds by Dattin \cite{dattin2022wrapped}.
			\begin{dfn}[Completion]\label{dfn:completion_of_relative_leg}
				By the same procedure as in \cref{dfn:attaching_spheres_simpl_handles}, we extend the various Legendrian submanifolds-with-boundary-and-corners in $\boldsymbol \varLambda$ over the various simplicial handles making up $X$. After gluing and smoothing it yields a smooth compact Legendrian $(n-1)$-submanifold in $\partial X$ which we denote by $\overline \varLambda$ and call the \emph{completion of $\boldsymbol \varLambda$}.
			\end{dfn}
			\begin{dfn}[Partial completion]
				Fix some $k$-face $\sigma_k\in C$ and consider $X(\sigma_k)$ (see \eqref{eq:x_stopped_away_from_sigma_k}). Following \cref{dfn:completion_of_relative_leg} but only using the Legendrians in $\boldsymbol \varLambda_{\supset \sigma_k} := \left\{\varLambda_{\sigma_i}\right\}_{\sigma_i \supset \sigma_k}$ yields a Legendrian $(n-1)$-submanifold with smooth boundary. The smooth boundary is partitioned into pieces, each of which lies in a contact manifold of the form $((\left\{0\right\} \times \partial W^{(\ell+1)}) \times \left\{0\right\}) \times T^\ast \varDelta^\ell$ for some $\ell$. For each such piece we take the Legendrian lift of the positive cone which now belongs to $((\left\{0\right\} \times W^{(\ell+1)}) \times \R) \times T^\ast \varDelta^\ell$. After possibly smoothing we denote the result by $\overline \varLambda_{\supset \sigma_k}$, which is a non-compact Legendrian $(n-1)$-submanifold, see \cref{fig:partial_completion}.
			\end{dfn}
			\begin{figure}[!htb]
				\centering
				\includegraphics{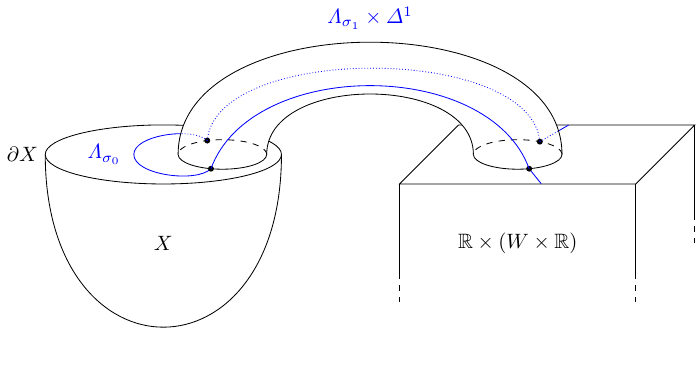}
				\caption{The partial completion $\overline \varLambda_{\supset \sigma_0}$. The left hand boundary of $\varLambda_{\sigma_1} \times \varDelta^1$ is $\varLambda_{\sigma_1} \subset \left\{0\right\} \times (\partial W \times \left\{0\right\})$ and the Legendrian lift of its positive cone is drawn.}\label{fig:partial_completion}
			\end{figure}
			In analogy with \cref{lma:diagram_of_sub_algebras} and \cref{lma:ce_spheres} we have the following.
			\begin{lma}
				For all $\mathfrak a > 0$ there exists some $\delta > 0$ and an arbitrary small perturbation of $\boldsymbol V_0$ such that for all $0 < \boldsymbol \varepsilon < \delta$ the following holds.
				\begin{enumerate}
					\item For each $k$-face $\sigma_k\in C_k$ there is an inclusion of dg-algebras induced by the inclusion on the set of generators
					\[
						CE^\ast(\overline \varLambda_{\supset \sigma_k}(\boldsymbol \varepsilon); X(\sigma_k)^{\boldsymbol \varepsilon}, \mathfrak a) \subset CE^\ast(\overline \varLambda(\boldsymbol \varepsilon); X^{\boldsymbol \varepsilon}, \mathfrak a)\, .
					\]
					\item For each inclusion of faces $\sigma_k \subset \sigma_{k+1}$ there is an inclusion of dg-algebras induced by the inclusion on the set of generators
					\[
						CE^\ast(\overline \varLambda_{\supset \sigma_{k+1}}(\boldsymbol \varepsilon); X(\sigma_{k+1})^{\boldsymbol \varepsilon},\mathfrak a) \subset CE^\ast(\overline \varLambda_{\supset \sigma_k}(\boldsymbol \varepsilon); X(\sigma_k)^{\boldsymbol \varepsilon}, \mathfrak a)\, .
					\]
				\end{enumerate}
			\end{lma}
			\begin{proof}
				Both statements follows from the same proof as in \cref{lma:diagram_of_sub_algebras} combined with \cref{lma:ce_spheres}.
			\end{proof}
			By general position we may assume that $\varLambda_{\sigma_k} \subset \partial V^{(k)}_{\sigma_k,0}$ belongs to the contact boundary of the subcritical part of $V^{(k)}_{\sigma_k}$, disjoint from the attaching spheres. Then we can describe $CE^\ast(\overline \varLambda; X)$ through the description of the Legendrian surgery isomorphism \cite{bourgeois2012effect}. Recall from \cref{sec:leg_attaching_data_for_simplex_handles} the definition of the union of attaching spheres $\varSigma(\boldsymbol h)$ of $X_0$. We use the following notation. 
			\begin{itemize}
				\item $\boldsymbol \varLambda \to \varSigma(\boldsymbol h)$ denotes a Reeb chord from $\varLambda_{\sigma_k}$ to $\partial \ell_{\sigma_k} \cup \bigcup_{\substack{\sigma_{k} \subset \sigma_{i} \\ k+1 \leq i \leq m}} (\ell_{\sigma_{i}} \times \R^{i-k-1})$ in $\partial V^{(k)}_{\sigma_k,0}$, for some $\sigma_k\in C_k$.
				\item $\boldsymbol \varLambda_{\supset \sigma_k} \to \varSigma_{\supset \sigma_k}(\boldsymbol h)$ denotes a Reeb chord from $\varLambda_{\sigma_j}$ to $\partial \ell_{\sigma_j} \cup \bigcup_{\substack{\sigma_j \subset \sigma_i \\ j+1 \leq i \leq m}} (\ell_{\sigma_{i}} \times \R^{i-j-1})$ in $\partial V^{(j)}_{\sigma_j,0}$, for some $\sigma_j \supset \sigma_k$.
				\item The following notations are used in an analogous way as above:
				\begin{itemize}[$\circ$]
					\item $\varSigma(\boldsymbol h) \to \varSigma(\boldsymbol h)$
					\item $\varSigma(\boldsymbol h) \to \boldsymbol \varLambda$
					\item $\boldsymbol \varLambda \to \boldsymbol \varLambda$
					\item $\varSigma_{\supset \sigma_k}(\boldsymbol h) \to \varSigma_{\supset \sigma_k}(\boldsymbol h)$
					\item $\varSigma_{\supset \sigma_k}(\boldsymbol h) \to \boldsymbol \varLambda_{\supset \sigma_k}$
					\item $\boldsymbol \varLambda_{\supset \sigma_k} \to \boldsymbol \varLambda_{\supset \sigma_k}$
				\end{itemize}
			\end{itemize}
			Without loss of generality, assume from now on that $V^{(0)}_{\sigma_0} \in \boldsymbol V$ is subcritical (see \cref{rmk:not_subcrit}). 
			\begin{lma}\label{lma:relative_version_surgery_desc_of_generators}
				For all $\mathfrak a > 0$ there exists some $\delta > 0$ and an arbitrary small perturbation of $\boldsymbol V_0$ such that for all $0 < \boldsymbol \varepsilon < \delta$ such that the following holds.
				\begin{enumerate}
					\item The generators of $CE^\ast(\overline \varLambda(\boldsymbol \varepsilon); X^{\boldsymbol \varepsilon})$ are in one-to-one correspondence with composable words of Reeb chords of action $< \mathfrak a$ which are the form $\boldsymbol \varLambda(\boldsymbol \varepsilon) \to \boldsymbol \varLambda(\boldsymbol \varepsilon)$ or
				\[
					\boldsymbol \varLambda(\boldsymbol \varepsilon) \to \varSigma(\boldsymbol h,\boldsymbol \varepsilon) \to \cdots \to \varSigma(\boldsymbol h, \boldsymbol \varepsilon) \to \boldsymbol \varLambda(\boldsymbol \varepsilon)\, .
				\]
				The differential is induced by the differential in the Chekanov--Eliashberg dg-algebra of $\overline \varLambda(\boldsymbol \varepsilon) \cup \varSigma(\boldsymbol h,\boldsymbol \varepsilon) \subset \partial X_0^{\boldsymbol \varepsilon}$.
				\item For each $\sigma_k \in C$ the generators of the dg-subalgebra $CE^\ast(\overline \varLambda_{\supset \sigma_k}(\boldsymbol \varepsilon); X(\sigma_k)^{\boldsymbol \varepsilon})$ are in one-to-one correspondence with composable words of Reeb chords of action $< \mathfrak a$ which are of the form $\boldsymbol \varLambda_{\supset \sigma_k}(\boldsymbol \varepsilon) \to \boldsymbol \varLambda_{\supset \sigma_k}(\boldsymbol \varepsilon)$ or
				\[
					\boldsymbol \varLambda_{\supset \sigma_k}(\boldsymbol \varepsilon) \to \varSigma_{\supset \sigma_k}(\boldsymbol h,\boldsymbol \varepsilon) \to \cdots \to \varSigma_{\supset \sigma_k}(\boldsymbol h,\boldsymbol \varepsilon) \to \boldsymbol \varLambda_{\supset \sigma_k}(\boldsymbol \varepsilon)\, .
				\]
				The differential is induced by the differential in the Chekanov--Eliashberg dg-algebra of $\overline \varLambda_{\supset \sigma_k}(\boldsymbol \varepsilon) \cup \varSigma(\boldsymbol h,\boldsymbol \varepsilon) \subset \partial X(\sigma_k)_0^{\boldsymbol \varepsilon}$.
				\end{enumerate}
			\end{lma}
			\begin{proof}
				This follows from the surgery description comparing generators before and after attachment of critical handles, see \cite[Theorem 5.10]{bourgeois2012effect}.
			\end{proof}
			Assuming no Reeb chord $\boldsymbol \varLambda \to \boldsymbol \varLambda$ in $\partial X$ passes through any top Weinstein handle we obtain the dg-algebra $CE^\ast(\overline \varLambda; X)$ as a colimit, similar to the main theorem \cref{thm:ce_descent}.
			\begin{thm}\label{thm:colimit_of_dgas_relative_case}
				Suppose $(C, \boldsymbol V)$ is a simplicial decomposition of $X$ such that every $V^{(0)}_{\sigma_0} \in \boldsymbol V$ is subcritical and that $\boldsymbol \varLambda$ is a Legendrian submanifold relative to $(C, \boldsymbol V)$. If there are no Reeb chords $\boldsymbol \varLambda \to \varSigma(\boldsymbol h)$ or if there are no Reeb chords $\varSigma(\boldsymbol h) \to \boldsymbol \varLambda$ then we have a quasi-isomorphism of dg-algebras over $\boldsymbol k$
				\[
					CE^\ast(\overline \varLambda; X) \cong \colim_{\sigma_k\in C} CE^\ast(\overline \varLambda_{\supset \sigma_k}; X(\sigma_k))\, .
				\]
			\end{thm}
			\begin{proof}
				By the surgery description of the generators of $CE^\ast(\overline \varLambda; X)$ in \cref{lma:relative_version_surgery_desc_of_generators}, the only Reeb chords of $\boldsymbol \varLambda(\boldsymbol \varepsilon)$ of action $< \mathfrak a$ are $\boldsymbol \varLambda(\boldsymbol \varepsilon) \to \boldsymbol \varLambda(\boldsymbol \varepsilon)$ in $\partial X_0^{\boldsymbol \varepsilon}$. Then the proof strategy is the same as for the proof of \cref{thm:ce_descent_first}.
			\end{proof}
			\begin{rmk}
				We note in particular that if $\boldsymbol V$ only consists of subcritical Weinstein manifolds, then the conditions in \cref{thm:colimit_of_dgas_relative_case} hold because $\varSigma(\boldsymbol h) = \varnothing$.
			\end{rmk}
			\begin{rmk}\label{rmk:not_subcrit}
				In case there is some $V^{(0)}_{\sigma_0}\in \boldsymbol V$ which is not subcritical, both \cref{lma:relative_version_surgery_desc_of_generators} and \cref{thm:colimit_of_dgas_relative_case} can still be made to hold the minor modification that we should replace $X_0$ (and correspondingly $X(\sigma_k)_0$) with $X'_0 := \# \boldsymbol V'_0$ where $\boldsymbol V'_0$ is obtained from $\boldsymbol V$ by replacing every $V^{(k)}_{\sigma_k}\in \boldsymbol V$ with its subcritical part $V^{(k)}_{\sigma_k,0}$ for $1\leq k\leq m$ and keeping the same Weinstein hypersurfaces. In other words, we do not remove the top Weinstein handles from the critical Weinstein manifolds $V^{(0)}_{\sigma_0}$, and we let $\varSigma(\boldsymbol h)$ be the union of those top attaching spheres of the top handles we do remove.
			\end{rmk}
			\begin{ex}\label{ex:unknot_pieces}
			Let $n\geq 4$ and consider the standard Weinstein $2n$-ball $(B^{2n}, \lambda)$ where $\lambda = \frac 12 \sum_{i=1}^n x_i dy_i - y_idx_i$. We consider a simplicial decomposition of $B^{2n}$ given by $(\varDelta^2, \boldsymbol V)$ where
				\begin{align*}
					\boldsymbol V &= \left\{B_1^{2n-2},B_2^{2n-2},B_3^{2n-2}, B^{4n-2}, \left\{\iota_i \colon B^{4n-2} \hookrightarrow \partial B_i^{2n-2}\right\}_{i=1}^3\right\} \\
					& \qquad \cup \left\{B_1^{2n}, B_2^{2n}, B_3^{2n}, \left\{\iota_{\#,k} \colon B_i^{2n-2} \#_{B^{2n-4}} B_j^{2n-2} \hookrightarrow \partial B_k^{2n}\right\}_{\left\{i,j\right\} = \left\{1,2,3\right\} \setminus \left\{k\right\}}\right\}
				\end{align*}
				see \cref{fig:decomposing_unknot} for an illustration with the $2$-simplex handle $B^{2n-4} \times D_\varepsilon T^\ast \varDelta^2$ hidden.
				\begin{figure}[!htb]
					\centering
					\includegraphics[scale=0.65]{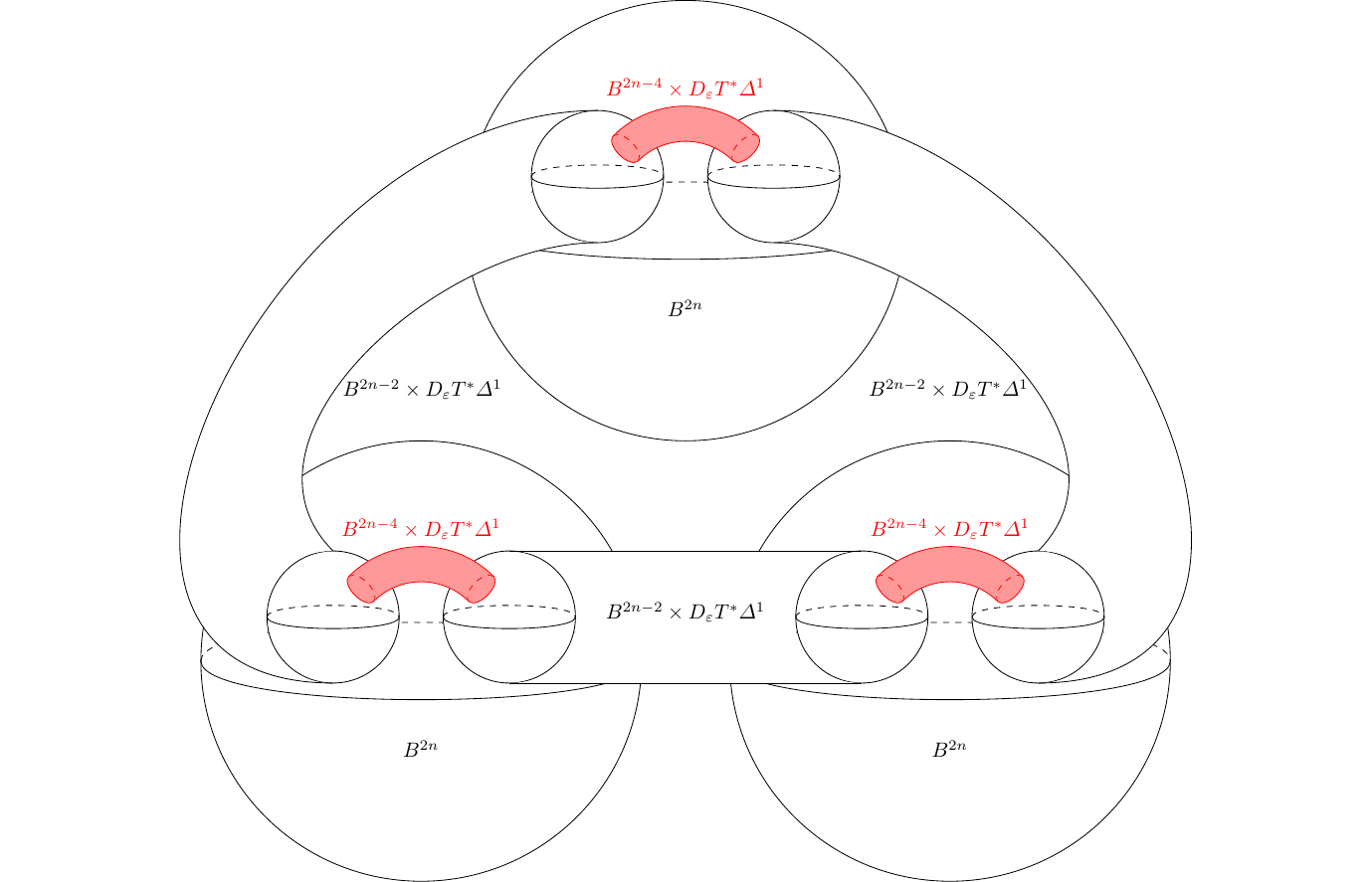}
					\caption{A simplicial decomposition $(\varDelta^2, \boldsymbol V)$, for the standard Weinstein $2n$-ball, with the $2$-simplex handle $B^{2n-4} \times D_\varepsilon T^\ast \varDelta^2$ hidden, to increase visual clarity.}\label{fig:decomposing_unknot}
				\end{figure}
				Now define the Legendrian submanifold $\boldsymbol \varLambda$ relative to $(C, \boldsymbol V)$ as follows, see \cref{fig:decomposing_unknot_piece}.
				\begin{itemize}
					\item Let $\varLambda_{\sigma_2} \subset \partial B^{2n-4}$ be the standard Legendrian $(n-3)$-unknot.
					\item Let $\varLambda_{\sigma_1} \subset \partial B^{2n-2}$ be the standard Legendrian $(n-2)$-unknot with one disk removed such that $\partial \varLambda_{\sigma_1} = \varLambda_{\sigma_2}$, for every $\sigma_1\in \varDelta^2_1$.
					\item Let $\varLambda_{\sigma_0} \subset \partial B^{2n}$ be be the standard Legendrian $(n-1)$-unknot with two disks removed such that $\partial \varLambda_{\sigma_0} = \varLambda_{\sigma_1} \#_{\varLambda_{\sigma_2}} \varLambda_{\sigma_1'}$, for every $\sigma_0\in \varDelta^2_0$ where $\sigma_1$ and $\sigma_1'$ are the two adjacent $1$-faces to $\sigma_0$.
				\end{itemize}
				\begin{figure}[!htb]
					\centering
					\includegraphics{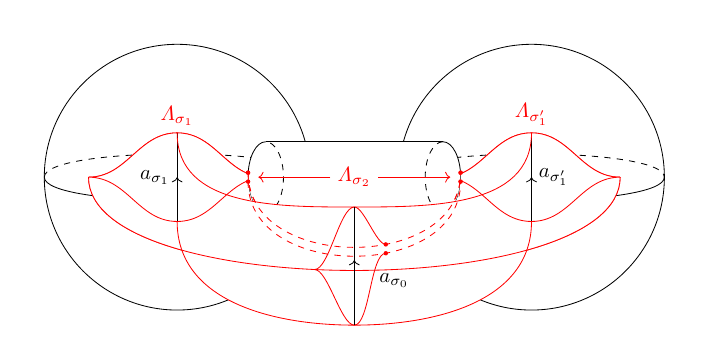}
					\caption{A piece of $\boldsymbol \varLambda$ in $\partial B^{2n}$ showing that $\partial \varLambda_{\sigma_0} = \varLambda_{\sigma_1} \#_{\varLambda_{\sigma_2}} \varLambda_{\sigma_1'}$ and moreover $\partial \varLambda_{\sigma_1} = \partial \varLambda_{\sigma_1'} = \varLambda_{\sigma_2} \subset \partial B^{2n-4}$.}\label{fig:decomposing_unknot_piece}
				\end{figure}
				Let $X := \# \boldsymbol V$, and let $\overline \varLambda \subset \partial X$ be the completion of $\boldsymbol \varLambda$, see \cref{dfn:completion_of_relative_leg}. Since every Weinstein manifold contained in $\boldsymbol V$ is subcritical, it follows from \cref{thm:colimit_of_dgas_relative_case} that $CE^\ast(\overline \varLambda ;X)$ is given by the following colimit
				\begin{equation}\label{eq:unknot_colim}
					CE^\ast(\overline \varLambda;X) \cong \colim\left(
					\begin{tikzcd}[row sep=scriptsize, column sep=tiny,ampersand replacement=\&]
						\&\&\mathcal A_{(\sigma_0)_1}\&\& \\
						\&\mathcal A_{(\sigma_1)_3} \urar \dlar\& \mathcal A_{\sigma_2} \lar \rar \dar \&\mathcal A_{(\sigma_1)_2} \ular \drar\& \\
						\mathcal A_{(\sigma_0)_2}\&\&\mathcal A_{(\sigma_1)_1} \ar[ll] \ar[rr]\&\&\mathcal A_{(\sigma_0)_3}
					\end{tikzcd}\right)\, ,
				\end{equation}
				where
				\begin{align*}
					\mathcal A_{\sigma_2} &:= \left\langle a_{\sigma_2}\right\rangle, \quad \partial a_{\sigma_2} = 0 \\
					\mathcal A_{\sigma_1} &:= \left\langle a_{\sigma_1}, a_{\sigma_2}\right\rangle, \quad \partial a_{\sigma_1} = a_{\sigma_2}, \quad \partial a_{\sigma_2} = 0 \\
					\mathcal A_{\sigma_0} &:= \left\langle a_{\sigma_0}, a_{\sigma_1}, a_{\sigma_1'}, a_{\sigma_2}\right\rangle, \quad \partial a_{\sigma_0} = a_{\sigma_1} - a_{\sigma_1'}, \quad \partial a_{\sigma_1} = \partial a_{\sigma_1'} = a_{\sigma_2}, \quad \partial a_{\sigma_2} = 0 \, .
				\end{align*}
				It is true by construction that $X \cong B^{2n}$ and that $\overline \varLambda$ is Legendrian isotopic to the standard unknot in $\partial X$. We can also algebraically construct a quasi-isomorphism between $CE^\ast(\overline \varLambda;X)$ and the dg-algebra $\mathbb F[a]$ where $\abs a = 1-n$ and $\partial a = 0$. This is done similarly to \cite[Section 8.4]{chekanov2002differential}, \cite[Section 3.1]{kalman2005contact} and \cref{sec:relating_ce_and_ginzburg}.
			\end{ex}
			\begin{rmk}\label{rmk:unknot_pieces}
				The significance of \cref{ex:unknot_pieces} is that the dg-subalgebras $\mathcal A_{\sigma_0}$ of $CE^\ast(\overline \varLambda; X)$ have geometric meaning. Namely, they are Chekanov--Eliashberg dg-algebras of unknots with boundary and corners, see \cref{fig:decomposing_unknot_piece}.

				This type of decomposition bears resemblance with the cellular Legendrian contact homology dg-algebra defined by Rutherford--Sullivan \cite{rutherford2020cellularI,rutherford2019cellularII,rutherford2019cellularIII} with one key exception that our resulting diagrams appearing in the main theorem \cref{thm:ce_descent}, \cref{thm:colimit_of_dgas_relative_case} and \eqref{eq:unknot_colim} behaves contravariantly with respect to a simplicial decomposition. Namely, any inclusion of faces $\sigma_k \subset \sigma_{k+1}$ gives a reverse inclusion of dg-subalgebras $\mathcal A_{\sigma_{k+1}}(\boldsymbol h) \subset \mathcal A_{\sigma_k}(\boldsymbol h)$, see \cref{cor:subalgebras_diagram_no_actions}. The cellular Legendrian contact homology dg-algebra of Rutherford--Sullivan on the other hand behaves covariantly with respect to a cellular decompositions, see \cite[Remark 8.1]{rutherford2019cellularII}.
			\end{rmk}
		\subsection{The Chekanov--Eliashberg cosheaf}\label{sec:ce_cosheaf}
			Let $(C, \boldsymbol V)$ be a simplicial decomposition of a Weinstein manifold $X$. Let $\boldsymbol C$ be the category with one object $\sigma_k$ for each $\sigma_k \in C$ and a unique morphism $\sigma_k \longrightarrow \sigma_\ell$ for each inclusion of faces $\sigma_k \subset \sigma_\ell$ in $C$. Recall that $\mathbf{dga}$ is the category of associative, non-commutative, non-unital dg-algebras over varying non-unital rings, see \cref{sec:simplical_descent} for details.
			\begin{dfn}[Chekanov--Eliashberg cosheaf]\label{dfn:ce_cosheaf}
				Let $(C, \boldsymbol V)$ be a simplicial decomposition of the Weinstein manifold $X$. Define the \emph{Chekanov--Eliashberg cosheaf} of $(C, \boldsymbol V)$ to be the following functor.
				\begin{align*}
					\mathcal{CE} \colon \boldsymbol C^{\text{op}} &\longrightarrow \textbf{dga} \\
					\sigma_k &\longmapsto (CE^\ast(\varSigma_{\supset \sigma_k}(\boldsymbol h); X(\sigma_k)_0), \boldsymbol k_{\sigma_k}) \\
					(\sigma_k \to \sigma_\ell) &\longmapsto ((CE^\ast(\varSigma_{\supset \sigma_\ell}(\boldsymbol h); X(\sigma_\ell)_0), \boldsymbol k_{\sigma_\ell}) \to (CE^\ast(\varSigma_{\supset \sigma_k}(\boldsymbol h); X(\sigma_k)_0)), \boldsymbol k_{\sigma_k})
				\end{align*}
			\end{dfn}
			The following is now a reformulation of the main theorem \cref{thm:ce_descent}.
			\begin{thm}[Cosheaf property]\label{thm:cosheaf_property}
				The functor $\mathcal{CE}$ satisfies that there is a quasi-isomorphism of dg-algebras over $\boldsymbol k$
				\[
					CE^\ast(\varSigma(\boldsymbol h); X_0) \cong \colim \mathcal{CE}
				\]
			\end{thm}
			\begin{rmk}
				\begin{enumerate}
					\item With the terminology used in \cite{shepard1985cellular} and \cite[Section 4]{curry2014sheaves} $\mathcal{CE}$ is a \emph{cellular cosheaf}.
					\item Let $X$ be a Weinstein manifold and consider a category with objects being Weinstein subsectors of $X$. Equip this category with a coverage consisting of good sectorial covers, which turns it into a site $\textbf{S}(X)$. Then we could define $\mathcal{CE}$ as a precosheaf $\mathcal{CE} \colon \textbf{S}(X)^{\text{op}} \longrightarrow \textbf{dga}$, which indeed is a cosheaf on the site $\textbf{S}(X)$ in the ordinary sense by \cref{thm:ce_descent} and \cref{thm:one-to-one_corr_covers_and_decomps}.
				\end{enumerate}
			\end{rmk}

			It is sensible to study the full Chekanov--Eliashberg cosheaf $\mathcal{CE}$, ie keeping track of the entire diagram $\left\{CE^\ast(\varSigma_{\supset \sigma_k}(\boldsymbol h); X(\sigma_k)_0)\right\}_{\sigma_k\in C}$ rather than just presenting $CE^\ast(\varSigma(\boldsymbol h);X)$ as a colimit. This is because each dg-subalgebra bears geometric significance as observed for instance in the following situations.
			\begin{itemize}
				\item If $C = \varDelta^1$ and $\boldsymbol V = \left\{V^{(1)},X_1^{2n},X_2^{2n}, V^{(1)} \hookrightarrow \partial X_i\right\}$ then \cref{thm:cosheaf_property} specializes to the quasi-isomorphism
				\[
					CE^\ast(\varSigma(\boldsymbol h); (X_1 \#_V X_2)_0) \cong \colim\left(\begin{tikzcd}[row sep=scriptsize, column sep=1mm]
						CE^\ast(\varSigma_{\supset \sigma_1}; X(\sigma_1)_0) \rar \dar & CE^\ast(\varSigma_{\supset (\sigma_0)_2}; X((\sigma_0)_2)_0) \\
						CE^\ast(\varSigma_{\supset (\sigma_0)_1}; X((\sigma_0)_1)_0) & 
					\end{tikzcd}\right)\, .
				\]
				In this situation the dg-subalgebras reveal the definition of the Chekanov--Eliashberg dg-algebra for singular Legendrians in \cite{asplund2021chekanov}. Namely, the dg-subalgebras $CE^\ast(\varSigma_{\supset (\sigma_0)_i}; X((\sigma_0)_i)_0)$ were denoted by $CE^\ast((V,h);X_i)$ in loc.\@ cit.\@ and defined as the Chekanov--Eliashberg dg-algebra of $\Skel V \subset \partial X_i$ thought of as a singular Legendrian.
				\item In the situation of \cref{ex:unknot_pieces}, each dg-subalgebra $CE^\ast(\varLambda_{\sigma_0}; X(\sigma_0)_0)$ has the geometric meaning that it is the Chekanov--Eliashberg dg-algebra of a Legendrian submanifold-with-boundary-and-corners.
			\end{itemize}
		\subsection{Plumbing of cotangent bundles of spheres}\label{sec:plumbing}
			Let $Q$ be a quiver (without potential) with vertex set $Q_0$ and arrow set $Q_1$. A path in $Q$ is a sequence of arrows $a_n \cdots a_1$ such that the head of $a_{i+1}$ is the tail of $a_i$ for $i\in \left\{1,\ldots,n-1\right\}$, and the path algebra of $Q$ is the free module with basis consisting of all paths in $Q$, including a trivial path $e_v$ of length $0$ for each $v\in Q_0$.

			We now review the definition of the Ginzburg dg-algebra $\mathscr G^n_Q$ from \cite{ginzburg2006calabi} following \cite[Section 6.2]{keller2011deformed} and \cite[Section 2]{lekili2020homological}. Given a quiver $Q$, consider a graded quiver $\overline Q$ with vertex set $Q_0$ and arrows consisting of the three following kinds.
			\begin{itemize}
				\item An arrow $g \colon v\to w$ in degree $0$ for each $g \colon v\to w$ in $Q_1$.
				\item An arrow $g^\ast \colon w \to v$ in degree $2-n$ for each $g \colon v\to w$ in $Q_1$.
				\item An arrow $h_v \colon v\to v$ in degree $1-n$ for each $v \in Q_0$.
			\end{itemize}
			We define $\mathscr G^n_Q$ to be the path algebra of $\overline Q$ equipped with the differential	$d$ given on arrows by
			\[
				d g = dg^\ast = 0, \qquad dh_v = \sum_{g \colon v\to \bullet} g g^\ast - \sum_{g \colon \bullet \to v} g^\ast g, \quad v\in Q_0\, ,
			\]
			and extended to the whole path algebra by linearity and the Leibniz rule. The Ginzburg dg-algebra $\mathscr G^n_Q$ is quasi-isomorphic to the so-called $n$-Calabi--Yau completion of the path algebra of $Q$, and hence it is homologically smooth and $n$-Calabi--Yau \cite[Theorem 5.2 and Theorem 6.3]{keller2011deformed}.

			 Etgü and Lekili \cite{etgu2019fukaya} computed the Chekanov--Eliashberg dg-algebra of the Legendrian attaching spheres of a plumbing of copies of $T^\ast S^n$ and showed by direct computation that it is quasi-isomorphic to $\mathscr G^n_Q$ for $n = 2$ where $Q$ is a Dynkin quiver of type $A$ and $D$. The result for the remaining Dynkin cases $Q = E_6,E_7,E_8$ was proven by Lekili and Ueda \cite[Section 5]{lekili2020homological}.

			Moreover, in the case $n \geq 3$, Lekili and Ueda \cite[Section 3]{lekili2020homological} used Koszul duality \cite[Theorem 68]{ekholm2017duality} to prove that there is a quasi-isomorphism
			\begin{equation}\label{eq:ce_ginzburg}
				CE^\ast(\varLambda^{n-1}_Q) \cong \mathscr G^n_Q\, ,
			\end{equation}
			for Dynkin quivers $Q$, where $\varLambda^{n-1}_Q$ is the union of the Legendrian attaching spheres of a plumbing of copies of $T^\ast S^n$ according to the plumbing quiver $Q$. A direct computation in case $n = 3$ and for Dynkin quivers of type A was done by Li \cite[Proposition 7.2]{li2019koszul} using the cellular dg-algebra of Rutherford--Sullivan \cite{rutherford2020cellularI,rutherford2019cellularII,rutherford2019cellularIII}.

			In this section we generalize \eqref{eq:ce_ginzburg} by direct computation for any $n\geq 3$ and any quiver (not only Dynkin ones) using simplicial decompositions.

			Before delving into computations and the proof of \eqref{eq:ce_ginzburg} we give a brief outline of the strategy.
			\begin{enumerate}
				\item Find a simplicial decomposition with $\dim C = 1$ of the plumbing of copies of $T^\ast S^n$.
				\item Compute the Chekanov--Eliashberg dg-algebra associated to each piece of the simplicial decomposition. The Legendrian submanifold in each piece is either an unknot (with boundary) or a Hopf link (with boundary).
				\item After taking colimits, the result is almost already quasi-isomorphic to the Ginzburg algebra by inspection, except that there are too many generators. After getting rid of the extra generators via a certain chain homotopy equivalence the result follows.
			\end{enumerate}
			\subsubsection{From a plumbing graph to a simplicial decomposition}\label{sec:from_plumb_to_simplicial_complex}
				Let $n \geq 3$ and let $Q$ be a quiver. Let $C$ be the graph which is obtained by taking the underlying graph of $Q$, and adding one extra vertex in the middle of each edge. For convenience, we consider a partition $C_0 = V \sqcup E$ where $V$ corresponds to elements in $Q_0$, and $E$ corresponds to elements in $Q_1$ and are the newly added vertices, see \cref{fig:graph_simpl_complex}.

				\begin{figure}[!htb]
					\centering
					\includegraphics[scale=1.5]{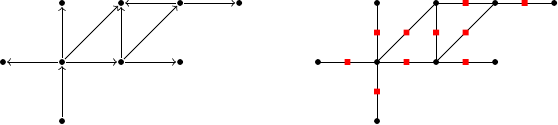}

					\vspace{5mm}

					\includegraphics[scale=1.5]{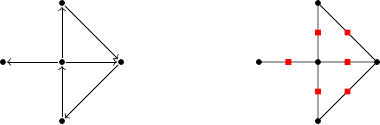}
					\caption{Left: Plumbing quivers $Q$. Right: The corresponding graph $C$ with vertex set $V \sqcup E$. The vertices represented by black dots $\bullet$ belong to $V$, and the vertices represented by red squares $\textcolor{red}{\blacksquare}$ belong to $E$.}\label{fig:graph_simpl_complex}
				\end{figure}

				We describe the plumbing of $T^\ast S^n$ as a simplicial decomposition $(C, \boldsymbol V)$ of the standard Weinstein $2n$-dimensional ball $B^{2n}$ with a number of Weinstein $1$-handles attached, together with a Legendrian $\boldsymbol \varLambda$ relative to $(C, \boldsymbol V)$, whose completion $\overline \varLambda$ is the union of Legendrian attaching spheres $\varLambda^{n-1}_Q$. The set $\boldsymbol V$ consists of
				\begin{enumerate}
					\item One copy of $V^{(1)}_{\sigma_1} := B^{2n-2}$ with the standard Liouville form $\lambda = \frac 12 \sum_{i=1}^{n-1} (x_i dy_i - y_i dx_i)$ for each edge $\sigma_1\in C_1$.
					\item One standard Weinstein ball $(B^{2n},\frac 12 \sum_{i=1}^{n} (x_i dy_i - y_i dx_i))$ for each vertex $\sigma_0\in C_0$.
					\item A Weinstein hypersurface 
					\[
						\bigsqcup_{\sigma_1 \supset \sigma_0} V^{(1)}_{\sigma_1} \hookrightarrow \partial B^{2n}\, ,
					\]
					for each vertex $\sigma_0\in C_0$.
				\end{enumerate}
				The Legendrian $\boldsymbol \varLambda$ relative to $(C, \boldsymbol V)$ consists of the following.
				\begin{enumerate}
					\item For each $\sigma_0 \in E$ let $\varLambda_{\sigma_0}$ be two linked Legendrian $(n-1)$-disks $\varLambda_1 \cup \varLambda_2$ such that $\partial \varLambda_i \subset \partial V^{(1)}_{(\sigma_1)_i}$ is the standard Legendrian $(n-2)$-unknot, see \cref{fig:locally_at_edge_vertex}.
					\item For each $\sigma_0\in V$ let $\varLambda_{\sigma_0}$ be the standard Legendrian $(n-1)$-unknot with $k$ open disks removed, where $k$ is the valency of the vertex $\sigma_0 \in C_0$. Each component of $\partial \varLambda_{\sigma_0}$ is a standard Legendrian $(n-2)$-unknot in $\partial V^{(1)}_{\sigma_1}$ for $\sigma_1 \supset \sigma_0$, see \cref{fig:locally_at_vertex_vertex}.
				\end{enumerate}
				\begin{figure}[!htb]
					\centering
					\includegraphics[scale=1.5]{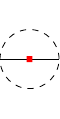}
					\hspace{15mm}
					\includegraphics{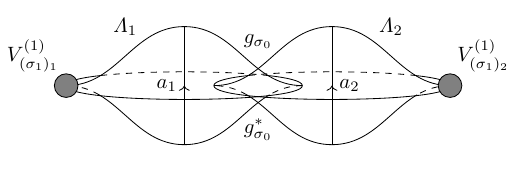}
					\includegraphics{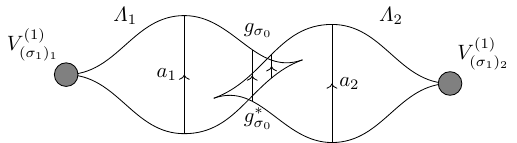}
					\caption{Left: Local depiction of a vertex $\sigma_0\in E$. Right: Two copies of $B^{2n-2}$ with $\varLambda_{\sigma_0}$ being two linked disks with boundary on $B^{2n-2} \sqcup B^{2n-2}$ in a Darboux chart. Bottom: A slice of $\varLambda_{\sigma_0}$ in generic position.}\label{fig:locally_at_edge_vertex}
				\end{figure}
				\begin{figure}[!htb]
					\centering
					\includegraphics[scale=1.5]{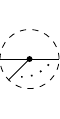}
					\hspace{30mm}
					\includegraphics{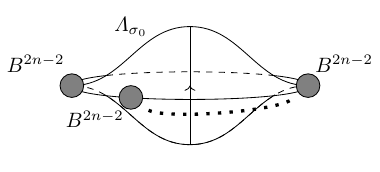}
					\caption{Left: Local depiction of a vertex $\sigma_0\in V$ with valency $k$. Right: $k$ copies of $B^{2n-2}$ with $\varLambda_{\sigma_0}$ being a punctured Legendrian $(n-1)$-unknot with boundary on $\bigsqcup_{i=1}^k B^{2n-2}$ in a Darboux chart.}\label{fig:locally_at_vertex_vertex}
				\end{figure}
				\begin{figure}[!htb]
					\centering
					\includegraphics{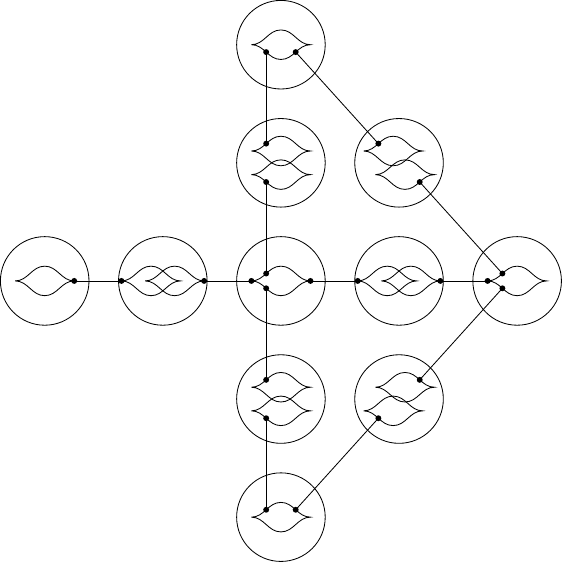}
					\caption{The simplicial decomposition $(C, \boldsymbol V)$ together with a Legendrian $\boldsymbol \varLambda$ corresponding to the plumbing of $T^\ast S^n$ according to the plumbing quiver shown in the bottom left of \cref{fig:graph_simpl_complex}.}\label{fig:globally_simplicial_decomp_plumbing}
				\end{figure}
			Let $\varLambda^{n-1}_Q := \overline \varLambda$ be the completion of $\boldsymbol \varLambda$ as defined in \cref{dfn:completion_of_relative_leg}. The Weinstein manifold $X := \# \boldsymbol V_0$ is Weinstein isomorphic to a standard Weinstein $2n$-ball with $c$ $1$-handles attached (the number of simple cycles in $C$). The Legendrian $\varLambda^{n-1}_Q$ is a union of Legendrian spheres in $\partial X \overset{\text{diffeo.}}{\cong} \#^c(S^1 \times S^{2n-2})$ (regular topological connected sum) being the attaching spheres for the Weinstein top handles of the plumbing of copies of $T^\ast S^n$ according to the plumbing quiver $Q$.

			\subsubsection{Computation of the Chekanov--Eliashberg dg-algebra}\label{sec:computation_of_ce_dga}
				Our main object of study is the Chekanov--Eliashberg dg-algebra of $\varLambda^{n-1}_Q$. By \cref{thm:colimit_of_dgas_relative_case}, $CE^\ast(\varLambda^{n-1}_Q; X)$ is the colimit of a diagram consisting of dg-subalgebras associated to each face of $C$ since $X$ is subcritical. Using standard representations of Legendrian spheres and disks in Darboux charts in standard contact spheres, we describe a quasi-isomorphic dg-algebra $\mathcal A \cong CE^\ast(\varLambda^{n-1}_Q;X)$. The dg-algebra $\mathcal A$ is given by the colimit of a diagram consisting of $\mathcal A_{\sigma_0}$ and $\mathcal A_{\sigma_1}$ for each $\sigma_0,\sigma_1 \in C$ together with inclusions induced by the inclusion on the set of generators $\mathcal A_{\sigma_1} \hookrightarrow \mathcal A_{\sigma_0}$ for each $\sigma_1 \supset \sigma_0$. Each dg-subalgebra $\mathcal A_{\sigma_1}$ is generated by a single generator $b_{\sigma_1}$ in degree $2-n$, and the differential is given by $\partial b_{\sigma_1} = 0$.

				Next we describe each dg-subalgebra $\mathcal A_{\sigma_0}$ for $\sigma_0\in C_0 = V \sqcup E$.
				\begin{description}
					\item[$\sigma_0\in V$] The dg-subalgebra $\mathcal A_{\sigma_0}$ is generated by $a_{\sigma_0}$ in degree $1-n$ and $b_{\sigma_1}$ in degree $2-n$ for each $\sigma_1 \supset \sigma_0$. The differential is given on the generators by
					\[
						\partial a_{\sigma_0} = \sum_{\sigma_1 \supset \sigma_0} b_{\sigma_1}, \qquad \partial b_{\sigma_1} = 0\, .
					\]

					Using the local picture in a Darboux ball as shown in \cref{fig:locally_at_vertex_vertex} we now give geometric meaning to the dg-subalgebra using standard representatives in Darboux charts. The generator $a_{\sigma_0}$ is the unique Reeb chord of $\varLambda_{\sigma_0}$ in $\partial B^{2n}$ and each $b_{\sigma_1}$ corresponds to the unique Reeb chord of $\partial \varLambda_{\sigma_0}$ in $\partial B^{2n-2}$. The differentials are found via counting Morse flow trees \cite{ekholm2007morse}.
					\item[$\sigma_0\in E$] The dg-subalgebra $\mathcal A_{\sigma_0}$ is generated by $a_1$ and $a_2$ both in degree $1-n$, $b_1$ and $b_2$ both in degree $2-n$, $g_{\sigma_0}$ in degree $0$ and $g^\ast_{\sigma_0}$ in degree $2-n$. The differential is given on the generators by
					\[
						\partial a_1 = b_1 - g_{\sigma_0}g^\ast_{\sigma_0}, \qquad \partial a_2 = b_2 + g^\ast_{\sigma_0} g_{\sigma_0}, \qquad \partial b_1 = \partial b_2 = \partial g_{\sigma_0} = \partial g^\ast_{\sigma_0} = 0\, .
					\]

					\begin{figure}[!htb]
						\centering
						\includegraphics{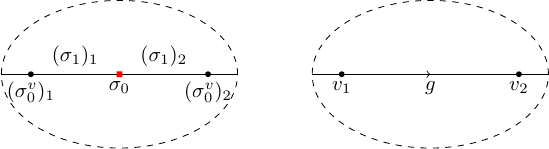}
						\caption{Labeling of the generators of $\mathcal A_{\sigma_0}$ depending on the orientation of the underlying quiver}\label{fig:edge_ori}
					\end{figure}

					Note that we use the following convention depending on the orientation of the edge $g\in Q_1$ corresponding to $\sigma_0\in E$. The vertices adjacent to $\sigma_0$ are denoted by $(\sigma_0^v)_1$ and $(\sigma_0^v)_2$, respectively, and they correspond to vertices $v_1,v_2 \in Q_0$. We enumerate them such that $g \colon v_1 \longrightarrow v_2$ is oriented from $v_1$ to $v_2$, see \cref{fig:edge_ori}.

					Using the local picture in a Darboux ball as shown in \cref{fig:locally_at_edge_vertex} we now give geometric meaning to the dg-subalgebra using standard representatives in Darboux charts. The generators $a_1$ and $a_2$ are the unique Reeb chords of $\varLambda_1$ and $\varLambda_2$ in $\partial B^{2n}$, respectively. The generators $b_1$ and $b_2$ are the unique Reeb chords of $\partial \varLambda_1$ and $\partial \varLambda_2$ in $\partial B^{2n-2}$, respectively. The generator $g_{\sigma_0}$ is the unique mixed Reeb chord going from $\varLambda_2$ to $\varLambda_1$ and $g^\ast_{\sigma_0}$ is the unique mixed Reeb chord going from $\varLambda_1$ to $\varLambda_2$. The differentials are found via counting Morse flow trees \cite{ekholm2007morse}.
				\end{description}
			\subsubsection{Relating the Chekanov--Eliashberg dg-algebra with the Ginzburg dg-algebra}\label{sec:relating_ce_and_ginzburg}
				Fix $\sigma_1\in C_1$ which is adjacent to $\sigma_0^v\in V$ and $\sigma_0^e \in E$. Then we may write $\varLambda^{n-1}_Q = \bigcup_{\sigma_0^v \in V} \varSigma_{\sigma_0^v}$ where $\varSigma_{\sigma_0^v} \subset \partial X$ is a Legendrian $(n-1)$-sphere for each $\sigma_0^v\in V$. The main simple observation is that $\varSigma_{\sigma_0^v}$ is a standard Legendrian $(n-1)$-unknot after performing a number of ``Reidemeister II moves'', see \cref{fig:reidemeister_ii_canceling}. We now relate $\mathcal A$ with the Ginzburg dg-algebra by ``undoing the Reidemeister II moves''. We show that $\mathcal A$ is chain homotopy equivalent to $\mathscr G^n_Q$.

				\begin{figure}[!htb]
					\centering
					\includegraphics{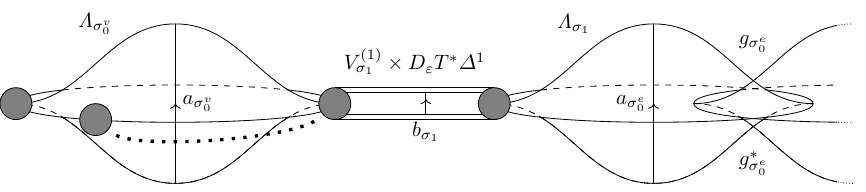}
					\caption{Local picture of $\varSigma_{\sigma_0^v}$ in a Darboux chart near $\sigma_1 \in C_1$ adjacent to $\sigma_0^v \in V$ and $\sigma_0^e \in E$. The left part of the figure lives in $\partial V^{(0)}_{\sigma_0^v} = \partial B^{2n}$, the right part of the figure lives in $\partial V^{(0)}_{\sigma_0^e} = \partial B^{2n}$, and they are connected by the $1$-simplex handle $V^{(1)}_{\sigma_1} \times D_\varepsilon T^\ast \varDelta^1$.}\label{fig:reidemeister_ii_canceling}
				\end{figure}

				A part of a standard representative of $\varSigma_{\sigma_0^v}$ in a Darboux chart is depicted in \cref{fig:reidemeister_ii_canceling}. Using this representation we have from \cref{sec:computation_of_ce_dga} that the differential on the generators of $\varSigma_{\sigma_0^v}$ is given by
				\[
					\partial a_{\sigma_0^v} = \sum_{\sigma_1 \supset \sigma_0^v} b_{\sigma_1}, \qquad\partial a_{\sigma_0^e} = b_{\sigma_1} + \boldsymbol g_{\sigma_0^e}, \qquad \partial b_{\sigma_1} = \partial g_{\sigma_0^e} = \partial g^\ast_{\sigma_0^e} = 0\, ,
				\]
				where we define
				\[
					\boldsymbol g_{\sigma_0^e} := \begin{cases}
						-g_{\sigma_0^e}g^\ast_{\sigma_0^e} &\text{if $g$ is oriented away from $\sigma_0^v$} \\
						g^\ast_{\sigma_0^e}g_{\sigma_0^e} &\text{if $g$ is oriented towards $\sigma_0^v$}
					\end{cases}\, ,
				\]
				where $g\in Q_1$ is the arrow corresponding to $\sigma_0^e \in E$. Then we define the following map on generators.
				\begin{align}
					\tau_{\sigma_1} \colon \mathcal A &\longrightarrow \mathcal A'_{\sigma_1} \label{eq:tau_edge} \\
					a_{\sigma_0^e} &\longmapsto 0 \nonumber \\
					b_{\sigma_1} &\longmapsto -\boldsymbol g_{\sigma_0^e} \nonumber\\
					x &\longmapsto x \quad \text{for any other generator}\nonumber
				\end{align}
				and extend it to the whole of $\mathcal A$ by linearity and multiplicativity. Here we define $\mathcal A'_{\sigma_1}$ to be the dg-algebra freely generated by all generators of $\mathcal A$ after removing $a_{\sigma_0^e}$ and $b_{\sigma_1}$, and $\mathcal A'_{\sigma_1}$ is equipped with the differential that is defined to be the same on generators as in $\mathcal A$ with the only exception that
				\[
					\partial a_{\sigma_0^v} = -\boldsymbol g_{\sigma_0^e} + \sum_{\substack{\sigma_1' \supset \sigma_0^v \\ \sigma_1' \neq \sigma_1}} b_{\sigma_1}\, .
				\]
				The dg-algebra $\mathcal A'_{\sigma_1}$ is obtained by ``undoing the Reidemeister II move'' at $\sigma_1\in C_1$ via the map $\tau_{\sigma_1}$. It follows from \cite[Section 8.4 and Lemma 2.2]{chekanov2002differential} that $\tau_{\sigma_1}$ is in fact a chain homotopy equivalence. We give a direct proof of this fact following \cite[Section 3.1]{kalman2005contact} by constructing a homotopy inverse.
				\begin{lma}\label{lma:edge_canceling_quiso}
					The map $\tau_{\sigma_1}$ is chain homotopy equivalence.
				\end{lma}
				\begin{proof}
					First we check that it is a dg-algebra map. The only non-trivial check is for the generator $a_{\sigma_0^e}$.
					\[
						\tau_{\sigma_1}(\partial a_{\sigma_0^e}) = \tau_{\sigma_1}(b_{\sigma_1}+\boldsymbol g_{\sigma_0^e}) = - \boldsymbol g_{\sigma_0^e} + \boldsymbol g_{\sigma_0^e} = 0 = \partial \tau_{\sigma_1}(a_{\sigma_0^e})
					\]
					Next, we prove that it is chain homotopy equivalence by defining the following dg-algebra map on generators:
					\begin{align*}
						\varphi_{\sigma_1} \colon \mathcal A'_{\sigma_1} &\longrightarrow \mathcal A \\
						a_{\sigma_0^v} & \longmapsto a_{\sigma_0^v} - a_{\sigma_0^e} \\
						x &\longmapsto x \quad \text{for any other generator}
					\end{align*}
					and extend $\varphi_{\sigma_1}$ to all of $\mathcal A'_{\sigma_1}$ by linearity and multiplicativity. Indeed, it is a dg-algebra map since
					\[
						\partial(\varphi_{\sigma_1}(a_{\sigma_0^v})) = \partial(a_{\sigma_0^v} - a_{\sigma_0^e}) = - \boldsymbol g_{\sigma_0^e} + \sum_{\substack{\sigma_1' \supset \sigma_0^v \\ \sigma_1' \neq \sigma_1}} b_{\sigma_1'} = \varphi_{\sigma_1}(\partial a_{\sigma_0^v})\, .
					\]
					We now define the following map on generators
					\begin{align*}
						K_{\sigma_1} \colon \mathcal A &\longrightarrow \mathcal A \\
						b_{\sigma_1} &\longmapsto a_{\sigma_0^e} \quad \\
						x &\longmapsto 0 \quad \text{for any other generator}
					\end{align*}
					and extend it by linearity to a $(\id, \varphi_{\sigma_1} \tau_{\sigma_1})$-derivation, meaning that $K_{\sigma_1}(xy) = K_{\sigma_1}(x)(\varphi_{\sigma_1} \tau_{\sigma_1})(y) + (-1)^{\abs x} x K_{\sigma_1}(y)$. Then we can check that $K_{\sigma_1}$ is a chain homotopy between $\varphi_{\sigma_1} \tau_{\sigma_1}$ and $\id$, meaning that $K_{\sigma_1} \partial + \partial K_{\sigma_1} = \id - \varphi_{\sigma_1} \tau_{\sigma_1}$. For any generator $x$ such that $(\varphi \circ \tau)(x) = x$ and $K_{\sigma_1}(x) = 0$, we easily see that this relation holds. For other generators we check it directly.
					\begin{align*}
						(K_{\sigma_1} \partial + \partial K_{\sigma_1})(a_{\sigma_0^e}) &= K(b_{\sigma_1} + \boldsymbol g_{\sigma_0^e}) = a_{\sigma_0^e} = (\id - \varphi_{\sigma_1} \tau_{\sigma_1})(a_{\sigma_0^e}) \\
						(K_{\sigma_1} \partial + \partial K_{\sigma_1})(b_{\sigma_1}) &= \partial(a_{\sigma_0^e}) = b_{\sigma_1} + \boldsymbol g_{\sigma_0^e} = b_{\sigma_1} - \varphi_{\sigma_1}(- \boldsymbol g_{\sigma_0^e}) = (\id - \varphi_{\sigma_1} \tau_{\sigma_1})(b_{\sigma_1}) \\
						(K_{\sigma_1} \partial + \partial K_{\sigma_1})(a_{\sigma_0^v}) &= K \left(\sum_{\sigma_1' \supset \sigma_0^v} b_{\sigma_1'}\right) = a_{\sigma_0^e} = a_{\sigma_0^v} - \varphi_{\sigma_1}(a_{\sigma_0^v}) = (\id - \varphi_{\sigma_1} \tau_{\sigma_1})(a_{\sigma_0^v})\, .
					\end{align*}
				\end{proof}
				We now define $\tau := \prod_{\sigma_1 \in C_1} \tau_{\sigma_1}$ which should be understood as concatenation of the $\tau_{\sigma_1}$ with respect to some total order on $C_1$. By its definition in \eqref{eq:tau_edge} we see that each map $\tau_{\sigma_1}$ can be defined on any $\mathcal A'_{\sigma_1'}$, and furthermore that $\tau_{\sigma_1} \circ \tau_{\sigma_1'} = \tau_{\sigma_1'} \circ \tau_{\sigma_1}$. Therefore we may consider the dg-algebra map
				\begin{alignat}{2}\label{eq:tau}
					\tau \colon \mathcal A &\longrightarrow \mathcal A'\\
					a_{\sigma_0^e} &\longmapsto 0 \quad &&\text{for all $\sigma_0^e \in E$} \nonumber \\
					b_{\sigma_1} &\longmapsto -\boldsymbol g_{\sigma_0^e} \quad &&\text{for all $\sigma_1\in C_1$} \nonumber\\
					x &\longmapsto x &&\text{for any other generator}\nonumber
				\end{alignat}
				where the resulting dg-algebra $\mathcal A'$ is generated by $a_{\sigma_0^v}$ in degree $1-n$ for each $\sigma_0^v \in V$, $g_{\sigma_0^e}$ in degree $0$ for each $\sigma_0^e \in E$ and $g^\ast_{\sigma_0^e}$ in degree $2-n$ for each $\sigma_0^e \in E$. The differential on $\mathcal A'$ is given by
					\[
						\partial a_{\sigma_0^v} = \sum_{\sigma_1 \supset \sigma_0^v} - \boldsymbol g_{\sigma_0^e} = \sum_{g \colon \sigma_0^v \to \bullet} g_{\sigma_0^e}g^\ast_{\sigma_0^e} - \sum_{g \colon \bullet \to \sigma_0^v} g^\ast_{\sigma_0^e}g_{\sigma_0^e}, \qquad \partial g_{\sigma_0^e} = \partial g^\ast_{\sigma_0^e} = 0\, .
					\]
				\begin{thm}\label{thm:ginzburg}
					There is a chain homotopy equivalence of dg-algebras $\mathcal A \simeq \mathscr G^n_Q$.
				\end{thm}
				\begin{proof}
					The map $\tau$ defined in \eqref{eq:tau} is a chain homotopy equivalence by \cref{lma:edge_canceling_quiso}. Finally it is clear that we have an isomorphism of dg-algebras $\mathcal A' \cong \mathscr G^n_Q$ via $a_{\sigma_0^v} \longmapsto h_v$, $g_{\sigma_0^e} \longmapsto g$ and $g^\ast_{\sigma_0^e} \longmapsto g^\ast$ which finishes the proof.
				\end{proof}
				\begin{cor}\label{cor:ce_dga_ginzburg}
					There is a quasi-isomorphism of dg-algebras $CE^\ast(\varLambda^{n-1}_Q;X) \cong \mathscr G^n_Q$
				\end{cor}
				\begin{proof}
					Follows immediately from \cref{thm:ginzburg} and from the fact that $\mathcal A$ is quasi-isomorphic to $CE^\ast(\varLambda^{n-1}_Q;X)$ by construction since it is defined using standard representations of Legendrian spheres and disks in Darboux charts in standard contact spheres.
				\end{proof}
\bibliographystyle{alpha}
\bibliography{simplicialdescent}
\end{document}